\documentclass[11pt,a4paper]{scrartcl}

\usepackage[utf8]{inputenc}
\usepackage[T1]{fontenc}
\usepackage[english]{babel}
\usepackage{verbatim}
\usepackage{abstract}
\usepackage{url}
\usepackage[hidelinks]{hyperref}

\usepackage[pdftex]{graphicx}
\usepackage{latexsym}
\usepackage{amsmath,amssymb,amsthm, mathtools,tikz}
\usepackage{stmaryrd}
\usetikzlibrary{cd}
\usepackage{pifont}
\usepackage[a4paper, left=2.7cm, right=2.7cm, bottom=3.2cm, top=3cm]{geometry}
\usepackage{makecell}

\usepackage{caption}

\usepackage[nottoc]{tocbibind}
\usepackage{hyperref,aliascnt}
\usepackage{slashed}

\newcommand{\ita}[1]{\textit{#1}}

\newcommand*\diff{\mathop{}\!\textup{d}}

\newcommand{\col}{\colon \thinspace}
\newcommand{\pr}{\textup{pr}}

\newcommand{\vol}{\textup{vol}}

\newcommand{\End}{\textup{End}}
\newcommand{\Ad}{\textup{Ad}}
\newcommand{\ad}{\textup{ad}}
\newcommand{\e}{\textup{e}}
\newcommand{\id}{\textup{id}}
\newcommand{\diag}{\textup{diag}}

\newcommand{\VertC}[2]{\Vert_{C^{{#1}}{#2}}}
\newcommand{\VertWC}[3]{\Vert_{C^{{#1}}_{{#2}}{#3}}}
\newcommand{\VertWCt}[3]{\Vert_{C^{{#1}}_{{#2},t}{#3}}}
\newcommand{\VertH}[2]{\Vert_{C^{{#1},\alpha}{#2}}}
\newcommand{\VertWH}[3]{\Vert_{C^{{#1},\alpha}_{{#2}}{#3}}}
\newcommand{\WsH}[3]{]_{C^{{#1},\alpha}_{{#2}}{#3}}}
\newcommand{\VertWHt}[3]{\Vert_{C^{{#1},\alpha}_{{#2},t}{#3}}}
\newcommand{\WsHt}[3]{]_{C^{{#1},\alpha}_{{#2},t}{#3}}}
\newcommand{\im}{\textup{Im}\thinspace }
\newcommand{\re}{\textup{Re}\thinspace }

\newcommand{\piorb}{\mathop{}\pi_1^{\textup{orb}}}
\newcommand{\G}{\textup{G}}
\newcommand{\SO}{\textup{SO}}
\renewcommand{\O}{\textup{O}}
\newcommand{\SU}{\textup{SU}}
\newcommand{\GL}{\textup{GL}}

\newcommand{\U}{\textup{U}}
\renewcommand{\Im}{\textup{Im}\thinspace}
\renewcommand{\Re}{\textup{Re}\thinspace}
\renewcommand{\H}{\textup{H}}

\title{A smooth family of $\G_2$-instantons over a generalised Kummer construction}
\author{Dominik Gutwein}

\numberwithin{equation}{section}

\newcommand{\mynewtheorem}[2]{
  \newaliascnt{#1}{dummycounter}
  \newtheorem{#1}[#1]{#2}
  \aliascntresetthe{#1}
  \expandafter\def\csname #1autorefname\endcsname{#2}
}

\theoremstyle{plain}
\mynewtheorem{theorem}{Theorem}
\mynewtheorem{proposition}{Proposition}
\mynewtheorem{lemma}{Lemma}
\mynewtheorem{corollary}{Corollary}
\mynewtheorem{introtheorem}{Theorem}

\mynewtheorem{introproposition}{Proposition}

\theoremstyle{definition}
\mynewtheorem{definition}{Definition}
\mynewtheorem{assumption}{Assumption}
\mynewtheorem{notation}{Notation}
\mynewtheorem{example}{Example}

\theoremstyle{remark}
\mynewtheorem{remark}{Remark}

\addto\extrasenglish{

}
\begin{document}
\maketitle

\begin{abstract}
We construct a smooth 1-parameter family of $\G_2$-instantons over a generalised Kummer construction desingularising the $\G_2$-orbifold discovered in \cite[Example~18]{Joyce-GeneralisedKummer2}. For this we extend the gluing construction for $\G_2$-instantons developed in \cite{Walpuski-InstantonsKummer} to Kummer constructions resolving $\G_2$-orbifolds whose singular strata are of codimension 6 and to connections (and entire families of connections) whose linearised instanton operator has a non-trivial cokernel. In order to overcome the corresponding obstructions, we utilise a $\mathbb{Z}_2$-action on the ambient manifold. More precisely, we perturb the (family of) pre-glued almost-instantons inside the class of $\mathbb{Z}_2$-invariant connections, which has the advantage that only the $\mathbb{Z}_2$-invariant locus of the cokernel needs to vanish. We then prove that the instantons that we construct over the resolution of the orbifold in \cite[Example~18]{Joyce-GeneralisedKummer2} are all infinitesimally rigid and non-flat. Moreover, we show that the resulting curve into the moduli space of $\G_2$-instantons modulo gauge is injective, that is, no two distinct instantons within the family are gauge-equivalent. To the author's knowledge this is the first example of a smooth 1-parameter family of instantons over a compact $\G_2$-manifold.
\end{abstract}

\section{Introduction}

A $\G_2$-instanton is a connection $A$ on a principal bundle $P \to Y$ over a $\G_2$-manifold $(Y,\phi)$ whose curvature $F_A\in \Omega^2(Y,\mathfrak{g}_P)$ satisfies the equation \[ F_A \wedge \psi =0 \] where $\psi \coloneqq * \phi$ is the coassociative 4-form on $Y$. These $\G_2$-instantons first appeared in the physics literature \cite{CorriganEtAl-Omega-ASD} and are absolute minima of the Yang--Mills functional (and therefore Yang--Mills connections). A slightly alternative (but equivalent) perspective is that $\G_2$-instantons are the critical points of a $\G_2$-analogue of the Chern--Simons functional on 3-manifolds~\cite{DonaldsonThomas-HigherDimGaugeTheory}. Since its suggestion by Donaldson and Thomas~\cite{DonaldsonThomas-HigherDimGaugeTheory}, one of the driving themes in the theory of special holonomy has been the possibility of constructing enumerative invariants based on such instantons (formally mirroring the familiar Casson-Floer picture for low dimensional manifolds; cf.~\cite{DonaldsonSegal-Gauge}). However, a rigorous construction of such an invariant (even a complete understanding of the list of its ingredients; cf. \cite[Section~6.2]{DonaldsonSegal-Gauge}, \cite[Chapter~6]{Walpuski-Diss}, \cite{Haydys-G2Instantons_and_SW}, \cite{DoanWalpuski-countingAssociatives}) is met with great analytical difficulties mainly arising from various non-compactness phenomena (cf. \cite{Tian-Gauge+Calibrations}, \cite{Walpuski-instantons-associatives-fueter}, \cite{Bera-CSAssociatives}). 

At the same time, the construction of concrete examples of $\G_2$-instantons, which could shed light onto the shape of such a potential invariant and could one day help to compute it, remains a difficult task. In this article we add to the list of known examples by constructing a smooth 1-parameter family of $\G_2$-instantons over a compact $\G_2$-manifold that arises via Joyce's generalised Kummer construction. The deformation theory of $\G_2$-instantons is governed by an elliptic operator, which in the examples constructed in this article has a 1-dimensional kernel (due to the deformations within the family). Since the Fredholm index of any elliptic operator over an odd-dimensional compact manifold is zero, the instantons constructed in this article have therefore a one-dimensional obstruction space. However, the obstructions are anti-invariant under a $\mathbb{Z}_2$-action on the underlying manifold and the constructed instantons are therefore '$\mathbb{Z}_2$-unobstructed'. To the author's knowledge this is the first example of a smooth 1-parameter family of instantons over a compact manifold and shows, in particular, that in the presence of non-trivial group actions instantons may not lie isolated in their moduli space (even when equivariant transversality holds).
 
\textbf{The generalised Kummer construction:} This is historically the first of currently three construction methods (\cite{Joyce-GeneralisedKummer2}, \cite{Kovalev-TCS}, and \cite{JoyceKarigiannis-Kummer}) for producing compact $\G_2$-manifolds and was developed by Joyce in~\cite{Joyce-GeneralisedKummer1, Joyce-GeneralisedKummer2}. It starts with a flat $\G_2$-orbifold, whose singular set is well-behaved (see \autoref{Ass:Codim6Singularities}). The construction then requires a set of resolution data, which in the context of this article roughly consists of a collection of Calabi--Yau 3-folds -- each one resolving the transverse singularity model for one singular stratum of the orbifold -- with an asymptotically locally Euclidean (ALE) geometry. At the topological level, the generalised Kummer construction produces a smooth manifold $\hat{Y}$ by replacing a neighbourhood of each singular stratum by a neighborhood in a certain fiber bundle whose fiber consists of the respective Calabi--Yau 3-fold in the resolution data. Geometrically, it first produces a $\G_2$-structure $\tilde{\phi} \in \Omega^3(\hat{Y})$ with small torsion which is then corrected to a torsion-free $\G_2$-structure $\phi \in \Omega^3(\hat{Y})$ via a singular perturbation. In order to apply this perturbation-scheme the fiberwise geometry of the glued-in Calabi--Yau bundle needs to be scaled down. In fact, by decreasing the scale-factor further, the Kummer construction produces a 1-parameter family of compact $\G_2$-manifolds\footnote{Technically, the smooth manifold $\hat{Y}$ in the construction also depends on $t$. However, since its diffeomorphism-type
remains constant for $t \in (0,\varepsilon)$, we will regard it in this section as one fixed manifold.} $(\hat{Y},\phi_t)_{t\in (0,\varepsilon)}$ that degenerates for $t\to 0$ back to the original orbifold. 

\textbf{Summary of results:} This article constructs $\G_2$-instantons over generalised Kummer constructions by perturbing connections which are in a quantified sense already close to being instantons. These almost-instantons are constructed from so-called gluing data (specified in \autoref{def:gluing_data}), which roughly consist of a flat connection over the orbifold and an asymptotically flat Hermitian Yang--Mills connection over each resolving ALE Calabi--Yau manifold. In fact, we allow for a family of such gluing data, producing a smooth family of almost-instantons. Moreover, when deforming these almost-instantons to genuine instantons we make use of the fact that many generalised Kummer constructions admit finite group actions preserving the coassociative 4-form (cf. \autoref{rem:EquivariantGeneralisedKummer}). If the gluing data is compatible with such a group action (in the sense of \autoref{rem:equiv_gluing_data}), then the almost-instantons are also invariant and can subsequently be deformed inside the class of invariant connections. This has the advantage that in order to perturb these connections, one only needs to show that the invariant part of the obstruction space associated to the linearised instanton operator vanishes. Summarising this discussion, our main theorem can schematically be stated as follows:

\begin{introtheorem}[\autoref{theo:perturbing_almost_instantons}, \autoref{prop:pregluing}, \autoref{prop:pregluing-family}, and \autoref{prop:LinearEstimate}]\label{introthm:G2-instantons-existence} 
Let $(\hat{Y},\phi_t)_{t\in (0,\varepsilon)}$ be a degenerating family of $\G_2$-manifolds arising from the generalised Kummer construction resolving an orbifold whose singular strata are all of codimension 6.\footnote{We concentrate on orbifolds with codimension 6 singularities in order to complement the constructions in \cite{Walpuski-InstantonsKummer} and \cite{Platt-G2Instantons}. However, most of our results can easily be adapted for orbifolds with codimension 4 singularities.} Furthermore, let $H$ be a finite (possibly trivial) group acting on $\hat{Y}$ as in \autoref{rem:EquivariantGeneralisedKummer}. For an $\mathfrak{F}$-family of $H$-equivariant gluing data (as defined in \autoref{def:gluing_data}, \autoref{rem:equiv_gluing_data}, and \autoref{prop:pregluing-family}), where $\mathfrak{F}\subset \mathbb{R}$ is a compact interval\footnote{We restrict to intervals because this amounts to a somewhat simpler notation and is the relevant case for our example. However, a generalisation of these results to any compact manifold $\mathfrak{F}$ (possibly with boundary) is straight-forward.}, there exists for each $t\in (0,\varepsilon)$ a bundle $\pi \col \hat{P} \to \hat{Y}$ together with a smooth $\mathfrak{F}$-family of $H$-invariant connections $(\tilde{A}_{t,\mathfrak{f}})_{\mathfrak{f}\in \mathfrak{F}} \subset \mathcal{A}(\hat{P})$ which are (in a quantified sense) close to being $\G_2$-instantons over $(\hat{Y},\phi_t)$. If the gluing data additionally satisfies \autoref{ass:invertible_linearisations} with respect to the action of $H$, then there exists an $\varepsilon^\prime$ such that for each $t\in (0,\varepsilon^\prime)$, the family of connection $(\tilde{A}_{t,\mathfrak{f}})_{\mathfrak{f}\in \mathfrak{F}}$ can be deformed to a smooth family of $\G_2$-instantons.
\end{introtheorem}

We then apply this theorem to prove the following:
\begin{introtheorem}[\autoref{prop:example18-existence-instantons}, \autoref{prop:example18-non-flat-instantons}, and \autoref{prop:example18-non-gauge-equivalent-instantons}]\label{introprop:instantons-examples}
Let $(Y_0,\phi_0)$ be the $\G_2$-orbifold constructed in \cite[Example~18]{Joyce-GeneralisedKummer2} and let $\mathfrak{F}\subset \mathbb{R}\setminus \pi \mathbb{Z}$ be a compact interval. There exists a $\G_2$-manifold $(\hat{Y},\phi_t)$ arising from $(Y_0,\phi_0)$ via the generalised Kummer construction and an $\SO(14)$-bundle $\pi\col \hat{P} \to \hat{Y}$ such that for sufficiently small $t>0$ there is a smooth $\mathfrak{F}$-family of $\G_2$-instantons $(A_{t,\mathfrak{f}})_{\mathfrak{f}\in \mathfrak{F}} \subset \mathcal{A}(\hat{P})$. All instantons in this family are infinitesimally irreducible (in the sense of \autoref{def:infinitesimally-irreducible}) and non-flat. Furthermore, the induced curve $\mathfrak{F}\to \mathcal{A}(\hat{P})/\mathcal{G}, \mathfrak{f}\mapsto [A_{t,\mathfrak{f}}]$ into the space of connections modulo gauge is injective.
\end{introtheorem}

\textbf{Comparison to previous constructions:} To the author's knowledge, there are three prior methods for constructing $\G_2$-instantons over compact $\G_2$-manifolds developed in \cite{Walpuski-InstantonsKummer}, \cite{WalpuskiSaEarp-Instantons_over_TCS}, \cite{Platt-G2Instantons}. Two of these methods (\cite{Walpuski-InstantonsKummer} and \cite{Platt-G2Instantons}) focus on manifolds arising as orbifold resolutions and are in spirit very similar to our construction. In fact, in this article we extend the construction method originally developed in \cite{Walpuski-InstantonsKummer} (which in the situation that we are considering becomes the same as the method in \cite{Platt-G2Instantons}) to Kummer constructions resolving orbifolds with codimension 6 singularities, to families of connections, and, in particular, to connections whose linearised instanton-operator has a non-trivial cokernel. As described above, we overcome the corresponding obstructions in the perturbation-scheme by utilizing symmetries of the underlying manifold. A somewhat similar technique has recently been developed to produce associative submanifolds in $\G_2$-manifolds arising out of the generalised Kummer construction \cite{DPSDTW-Associatives}. The fact that we perturb families of connections causes additional work compared to the construction in \cite{Walpuski-InstantonsKummer}. This can for example be seen in \autoref{sec: instantons over Kummer} where additional work is needed to prove the smoothness of the deformed family (regarded as a map from an interval $\mathfrak{F} \subset \mathbb{R}$ into the space of connections equipped with its Fréchet space structure) and to derive estimates on the derivatives of this family. These estimates are needed later in \autoref{prop:example18-non-gauge-equivalent-instantons} to prove that the constructed path into the space of connections modulo gauge is injective.

\textbf{Outline of the paper:} In \autoref{sec: Background} we recall the required background on the generalised Kummer construction, the construction of ALE Calabi--Yau 3-folds, and asymptotically flat Hermitian Yang--Mills connections over such ALE Calabi--Yau 3-folds. 

In \autoref{sec: instantons over Kummer} we prove with \autoref{theo:perturbing_almost_instantons} an abstract existence theorem for $\G_2$-instantons over generalised Kummer constructions. It roughly states that whenever an $\mathfrak{F}$-family of $H$-invariant connections $(\tilde{A}_{\mathfrak{f}})_{\mathfrak{f} \in \mathfrak{F}}$ (for $\mathfrak{F}\subset \mathbb{R}$ an interval and $H$ a finite group acting on $\hat{Y}$ preserving $* \phi_t$) is in a quantified sense sufficiently close to being a family of $\G_2$-instanton and if the $H$-invariant cokernel of the linearised instanton operators associated to each $\tilde{A}_{\mathfrak{f}}$ vanishes, then $(\tilde{A}_{\mathfrak{f}})_{\mathfrak{f} \in \mathfrak{F}}$ can be perturbed to a nearby family of $H$-invariant instantons. Moreover, we prove that this family of instantons has the same order of differentiability (seen as a map from $\mathfrak{F}$ into the space of connections) as the initial family of connections and establish estimates on the derivatives of this family. Ultimately, we give a sufficient condition for the perturbed connections to be infinitesimally irreducible. 

In \autoref{sec:approxG2instantons_from_gluingdata} we then give the construction of the family of model connections to which we later apply \autoref{theo:perturbing_almost_instantons}. \autoref{prop:pregluing} and \autoref{prop:pregluing-family} also prove some of the estimates required for the application of \autoref{theo:perturbing_almost_instantons} and  \autoref{prop: derivative estimates on deformed family of instantons}. 

\autoref{sec: linear analysis} studies the linearised instanton operator associated to the connections constructed in the previous section. It ultimately establishes the last missing condition of \autoref{theo:perturbing_almost_instantons} (the linear estimate) under \autoref{ass:invertible_linearisations} on the gluing data. 

\autoref{sec:Example18} considers the $\G_2$-orbifold discovered in~\cite[Example~18]{Joyce-GeneralisedKummer2}, finds a resolution thereof together with an $\mathbb{Z}_2$-action, and a 1-parameter family of $\mathbb{Z}_2$-equivariant gluing data to which \autoref{introthm:G2-instantons-existence} above can be applied. Furthermore, it proves that the resulting family of instantons consists of elements that are non-flat, infinitesimally rigid, and pairwise not gauge-equivalent.

\subsection{Acknowledgements}

This article is an extension of Chapter 4 of my PhD-thesis \cite{Gutwein-Diss} and I am very grateful to my PhD-supervisor Thomas Walpuski for his endless support and guidance. I would also like to express my gratitude to Lorenzo Foscolo and Sebastian Goette for their many helpful suggestions and remarks on my thesis. Furthermore, I am thankful to Gorapada Bera, Thibault Langlais, Viktor Majewski, and Jacek Rzemieniecki for their constructive tips, helpful discussions, and proofreading. I would also like to thank Mateo Galdeano and Daniel Platt for answering my questions on \cite{GaldeanoPlattTanakaWang-spin7instantons} and for their valuable advice on finding instantons. Moreover, I am grateful to Benoit Charbonneau for his interest in this work and helpful discussions on invariant instantons. While working on this article, I was supported by the Simons Collaboration “Special Holonomy in Geometry, Analysis, and Physics” (during my PhD) and later by the Deutsche Forschungsgemeinschaft (DFG, German Research Foundation) under SFB-Geschäftszeichen 1624 – Projektnummer 506632645. 

\subsection{Notation}

In this article we use the following notation: 
\begin{itemize}
\item If $E\to Y$ is a vector bundle, then we denote for $k\in \mathbb{N}_0$ and $\alpha \in (0,1)$ by $C^{k,\alpha}(E)$ the space of sections of $E$ which are of regularity $C^{k,\alpha}$.
\item For two Banach spaces $B_1,B_2$ we will denote the Banach space of bounded linear maps $B_1 \to B_2$ by $\textup{Lin}(B_1,B_2)$.
\item $c>0$ will denote a generic constant whose value may change from one appearance to the next. This constant will always be $t$-independent, except once in the proof of \autoref{theo:perturbing_almost_instantons} where this dependence is noted.
\item If $Y$ carries a Riemannian metric $g$, then for any vector $v \in TY$ we will denote the associated covector by $v^\flat \coloneqq g(v,\cdot) \in T^*Y$. Similarly, we will denote for a covector $a \in T^*Y$ by $a^\sharp \in TY$ the vector defined by $a = g(a^\sharp,\cdot)$.
\end{itemize}

\section{Background} \label{sec: Background}

In this section we first review the necessary background on Joyce's generalised Kummer construction and explain in \autoref{sec:Ricci-Flat ALE metrics on crepant resolutions} a method for finding ALE Calabi--Yau 3-folds, which play an integral part in this construction. In the last section we discuss Hermitian Yang--Mills connections over ALE Calabi--Yau 3-folds which are needed for the construction of $\G_2$-instantons later in this article.

\subsection{Joyce's generalised Kummer construction}\label{Sec:generalisedKummer}

This section explains the generalised Kummer construction, a method developed (and extended) by Joyce~\cite{Joyce-GeneralisedKummer1,Joyce-GeneralisedKummer2,Joyce-Black} to produce $\G_2$-manifolds as desingularisations of certain flat $\G_2$-orbifolds. This section follows the presentations in~\cite[Section~2]{DPSDTW-Associatives}, \cite[Section~2.1]{Gutwein-coassociatives}, and \cite[Chapter~2]{Gutwein-Diss} very closely. The following class of examples fixes our conventions and serves as model for the $\G_2$-manifolds (and orbifolds) considered in this article:
\begin{example}\label{ex_G2Model}
Let $(Z,\omega,\Omega)$ be a Calabi--Yau manifold or orbifold of complex dimension 3 with Kähler form $\omega \in \Omega^2(Z)$ and holomorphic volume form $\Omega \in \Omega^3(Z,\mathbb{C})$. These are related via $\frac{1}{6}\omega^3 = \frac{i}{8}\Omega\wedge\bar{\Omega}$. 
\begin{enumerate}
\item The product $\mathbb{R} \times Z$ carries a torsion-free $\G_2$-structure defined by
\begin{equation}\label{eq:modelG2Structure}
\phi \coloneqq \diff s \wedge \omega + \im \Omega \in \Omega^3(\mathbb{R} \times Z)
\end{equation}
where $s$ denotes the coordinate on $\mathbb{R}$. The corresponding coassociative $4$-form is given by $\psi = \frac{1}{2}\omega	\wedge \omega + \diff s \wedge \re \Omega$.
\item Assume that $\textup{Crys}< \O(\mathbb{R}) \ltimes  \mathbb{R}$ is a crystallographic group that acts on $Z$ via $\rho \col \textup{Crys} \to \textup{Isom}(Z,g_\omega)$ such that for any $(\pm 1,v) \in \textup{Crys}$
\begin{equation}\label{eq:CalabiYauEquivariance}
\rho(\pm 1,v)^* \omega = \pm \omega \quad \textup{and} \quad \rho(\pm 1,v)^*\im\Omega =\im \Omega. 
\end{equation}
The $3$-form $\phi$ is invariant under the product action on $\mathbb{R} \times Z$ and descends therefore to a torsion-free $\G_2$-structure on the quotient $Y\coloneqq (\mathbb{R} \times Z)/\textup{Crys}$. By a slight abuse of notation we denote the corresponding $3$-form on $Y$ by $\phi$ as well. Note that whenever $\textup{Crys}$ is a Bieberbach group (i.e. a 1-dimensional lattice) then the action is free and taking the quotient does not introduce additional singularities in $Y$.
\end{enumerate}  
\end{example}

This article concentrates on $\G_2$-orbifolds that satisfy the following assumption:

\begin{assumption}\label{Ass:Codim6Singularities}
Let $(Y_0,\phi_0)$ be a compact flat $\G_2$-orbifold and denote by $\mathcal{S}$ the set of connected components of the singular set of $Y_0$. We assume that for every $S \in \mathcal{S}$ there exist
\begin{enumerate}
\item A finite subgroup $\Gamma_S < \SU(3)$ which acts freely on $\mathbb{C}^3\setminus\{0\}$, a lattice $\Lambda_S< \mathbb{R}$, and a group action $\rho_S \col \Lambda_S \to N_{\O(\mathbb{C}^3)}(\Gamma_S)/\Gamma_S \subset \textup{Isom}(\mathbb{C}^3/\Gamma_S,g_0)$ that satisfies \eqref{eq:CalabiYauEquivariance}. Denote by \[(Y_S \coloneqq (\mathbb{R} \times \mathbb{C}^3/\Gamma_S)/\Lambda_S,\phi_S)\] the corresponding $\G_2$-orbifold from \autoref{ex_G2Model}. 
\item An open set \[ \mathcal{V}_{S} \coloneqq (\mathbb{R} \times B_{\kappa}(0)/\Gamma_S)/\Lambda_S \subset Y_S \] for a fixed $\kappa>0$ and an open embedding $\mathtt{J}_S \col \mathcal{V}_{S} \to Y_0$ with $S \subset \mathtt{J}_S(\mathcal{V}_{S})$ and $\mathtt{J}_S^*\phi_0 = \phi_S$. The parameter $\kappa$ is chosen such that $\overline{\mathtt{J}_{S_1}(\mathcal{V}_{S_1})}\cap \overline{\mathtt{J}_{S_2}(\mathcal{V}_{S_2})} = \emptyset $ for any two $S_1 \neq S_2 \in \mathcal{S}$.
\end{enumerate}
\end{assumption}
\begin{remark}\label{rem: finite subgroups of SU(3)}
From the classification of finite subgroups of $\SU(3)$ (see~\cite[Chapter~1]{YauYung-3D_Gorenstein_singularities} and the references therein) follows that if $\Gamma_S<\SU(3)$ acts freely on $\mathbb{C}^3\setminus \{0\}$, then $\Gamma_S$ is in fact abelian (see~\cite[Theorem~23]{YauYung-3D_Gorenstein_singularities}).
\end{remark}

\begin{definition}[{cf.~\cite[Definition~11.4.1]{Joyce-Black} and~\cite[Definition~2.6]{DPSDTW-Associatives}}]\label{def:RData}
Let $(Y_0,\phi_0)$ be a flat $\G_2$-orbifold satisfying \autoref{Ass:Codim6Singularities}. A set of resolution data consists for every $S \in \mathcal{S}$ of the following:
\begin{enumerate}
\item A fixed choice of $(\Gamma_S, \Lambda_S,\rho_S,\mathtt{J}_S)$ as in \autoref{Ass:Codim6Singularities}.
\item \label{def:ALE-space}  An asymptotically locally Euclidean Calabi--Yau manifold $(\hat{Z}_S,\tau_S,\hat{\omega}_S,\hat{\Omega}_S)$ (also ALE Calabi--Yau manifold for short) asymptotic to $\mathbb{C}^3/\Gamma_S$. That is, a resolution $\tau_S \col \hat{Z}_S \to \mathbb{C}^3/\Gamma_S$ together with a Calabi--Yau structure $(\hat{\omega}_S,\hat{\Omega}_S)$ on $\hat{Z}_S$ that satisfies 
\begin{equation}\label{eq:ALE-Kählerform_decay}
\big\vert \nabla^k \big({\tau_S}_*\hat{\omega}_S - \omega_0\big) \big\vert = \mathcal{O}(r^{-6-k}) \textup{ as $r \to \infty$}
\end{equation} 
where $\omega_0$ denotes the flat Kähler structure on $\mathbb{C}^3/\Gamma_S$. The norm and covariant derivatives are hereby taken with respect to the flat metric on $(\mathbb{C}^3\setminus\{0\})/\Gamma_S$.
\item \label{bul: lifting group action in definition of R-data} A group action $\hat{\rho}_S \col \Lambda_S \to \textup{Isom}(\hat{Z}_S,g_{\hat{\omega}_S})$ which leaves $(\hat{\omega}_S,\hat{\Omega}_S)$ invariant (in the sense of \eqref{eq:CalabiYauEquivariance}) and makes $\tau_S$ equivariant.
\end{enumerate}
\end{definition}
\begin{remark}
Regarding the previous definition:
\begin{enumerate}
\item \label{bul:ALE-asymptotic-coordinates} In the definition of an ALE-manifold one can relax the condition on $\tau_S$ to be defined on the entire manifold $\hat{Z}_S$. In fact, it suffices that $\tau_S \col (\hat{Z}_S\setminus \hat{K}_S) \to (\mathbb{C}^3 \setminus B_{R_S}(0)) /\Gamma_S$ is a diffeomorphism defined outside a $\Lambda_S$-invariant compact set $\hat{K}_S \subset \hat{Z}_S$. 
\item The condition $\frac{1}{6}\hat{\omega}_S^3 = \frac{i}{8} \hat{\Omega}_S \wedge \bar{\hat{\Omega}}_S$ determines the holomorphic volume form $\hat{\Omega}_S$ on $(\hat{Z}_S,\hat{\omega}_S)$ only up to a constant $\theta \in S^1 \subset \mathbb{C}$. Since $\tau_S \col \hat{Z}_S \to \mathbb{C}^3/\Gamma_S$ is holomorphic, we have a holomorphic function $f \col \mathbb{C}^3 \to \mathbb{C}$ such that $(\tau_S)_* \hat{\Omega}_S = f \Omega_0$ (where $\Omega_0$ is the standard holomorphic volume form $\diff z_1 \wedge \diff z_2 \wedge \diff z_3$ on $\mathbb{C}^3/\Gamma_S$). The ALE-condition~\eqref{eq:ALE-Kählerform_decay}, the identity $\frac{1}{6}\hat{\omega}^3_S = \frac{i}{8}\hat{\Omega}_S \wedge \bar{\hat{\Omega}}_S$, and the maximum principle then imply that $f$ is a constant function of norm 1. Thus, after modifying $\hat{\Omega}_S$, we can assume in the following that $(\tau_S)_*\hat{\Omega}_S = \Omega_0$ holds.
\end{enumerate}
\end{remark}

For a given orbifold $Y_0$, a set of resolution data, and a positive parameter $t>0$ we define the following sets: 
\begin{align*}
\mathcal{V}_{\kappa/8}\coloneqq \bigsqcup_{S\in \mathcal{S}} \mathcal{V}_{S,\kappa/8} \quad &\textup{for} \quad \mathcal{V}_{S,\kappa/8}\coloneqq (\mathbb{R} \times B_{\kappa/8}(0)/\Gamma_S)/\Lambda_S \subset (\mathbb{R} \times \mathbb{C}^3/\Gamma_S)/\Lambda_S \\
\mathcal{V}_{\kappa}\coloneqq \bigsqcup_{S\in \mathcal{S}} \mathcal{V}_{S,\kappa} \quad &\textup{for} \quad \mathcal{V}_{S,\kappa} \coloneqq (\mathbb{R} \times B_{\kappa}(0)/\Gamma_S)/\Lambda_S \subset (\mathbb{R} \times \mathbb{C}^3/\Gamma_S)/\Lambda_S \\
\hat{\mathcal{V}}_{\kappa/8}^t \coloneqq \bigsqcup_{S\in \mathcal{S}} \hat{\mathcal{V}}_{S,\kappa/8}^t \quad &\textup{for} \quad \hat{\mathcal{V}}_{S,\kappa/8}^t \coloneqq (\mathbb{R} \times (t\tau_S)^{-1}(B_{\kappa/8}(0)/\Gamma_S))/\Lambda_S \subset (\mathbb{R} \times \hat{Z}_S)/\Lambda_S \\
\hat{\mathcal{V}}_{\kappa}^t \coloneqq \bigsqcup_{S\in \mathcal{S}} \hat{\mathcal{V}}_{S,\kappa}^t \quad &\textup{for} \quad \hat{\mathcal{V}}_{S,\kappa}^t \coloneqq(\mathbb{R} \times (t \tau_S)^{-1} (B_{\kappa}(0)/\Gamma_S))/\Lambda_S \subset (\mathbb{R} \times \hat{Z}_S)/\Lambda_S
\end{align*}

Denote by $\mathtt{J} \col \mathcal{V}_{\kappa} \to Y_0$ and $t \tau \col \hat{\mathcal{V}}_{\kappa}^t \to \mathcal{V}_{\kappa}$ the maps induced by all $\{\mathtt{J}_S\}_{S\in \mathcal{S}}$ and $\{t\tau_S\}_{S \in \mathcal{S}}$, respectively. 

\begin{definition}[{\cite[Proof of Theorem~2.2.1]{Joyce-GeneralisedKummer2}}]\label{def:OrbifoldResolution}
Given a flat $\G_2$-orbifold $(Y_0,\phi_0)$ and a set of resolution data, Joyce defines a 1-parameter family of smooth manifolds by \[\hat{Y}_t \coloneqq \big(Y_0 \setminus \mathtt{J}(\mathcal{V}_{\kappa/8})\big) \cup \hat{\mathcal{V}}_{\kappa}^t/\sim \] where $\hat{\mathcal{V}}_{\kappa}^t\setminus \hat{\mathcal{V}}_{\kappa/8}^t \ni y \sim \mathtt{J}(t\tau(y)) \in \mathtt{J}(\mathcal{V}_{\kappa}\setminus \mathcal{V}_{\kappa/8})$.
\end{definition}

Joyce (cf.~\cite[Proof of Theorem~2.2.1]{Joyce-GeneralisedKummer2}) then equips $\hat{Y}_t$ for sufficiently small $t$ with a closed $\G_2$-structure $\tilde{\phi}_t$ which we explain next.

For $S\in \mathcal{S}$ let $(\hat{Z}_S,\tau_S,\hat{\omega}_S,\hat{\Omega}_S)$ be the ALE Calabi--Yau 3-fold used to resolve $S$. There exist $\sigma_S^t \in \Omega^1((\mathbb{C}^3\setminus \{0\})/\Gamma_S)$ with $ \vert \nabla^k \sigma_S^t \vert = t^{6}\mathcal{O}(r^{-5-k})$ such that $(t\tau_S)_* (t^2\hat{\omega}_S) - \omega_0 = \diff \sigma_S^t$ (cf.~\cite[Theorem~8.2.3]{Joyce-Black}). Let $\tilde{\phi}_t \in \Omega^3(\hat{Y}_t)$ be defined as  
\begin{align*}
\tilde{\phi}_t \coloneqq 
\begin{cases}
\phi_0 &\textup{ over $Y_0\setminus \mathtt{J}(\mathcal{V}_{\kappa})$} \\
\diff s \wedge \big(t^2 \hat{\omega}_S - \diff (\chi_t \cdot (t\tau_S)^*\sigma_S^t)\big) +  \im t^3 \hat{\Omega}_S &\textup{ over each $\hat{\mathcal{V}}_{S,\kappa}^t$}
\end{cases}
\end{align*}
where $\chi_t \coloneqq \chi \circ (t\tau)$ for a smooth non-decreasing function $\chi \col [0,\kappa] \to [0,1]$ with 
\begin{align*}
\chi(s) = \begin{cases}
0 \textup{ for $s \leq \kappa/4$}\\
1 \textup{ for $s \geq \kappa/2$.}
\end{cases}
\end{align*}
\begin{proposition}[{\cite[Proof of Theorem~2.2.1]{Joyce-GeneralisedKummer2}}]
Let $(Y_0,\phi_0)$ be a flat $\G_2$-orbifold satisfying \autoref{Ass:Codim6Singularities} and let $\mathcal{R}$ be a set of resolution data. Denote by $(\hat{Y}_t,\tilde{\phi}_t)$ be the smooth manifold and $3$-form constructed above. Then there exists a positive constant $T_{pK} = T_{pK}(Y_0,\phi_0,\mathcal{R})$ such that $\tilde{\phi}_t$ defines a closed $\G_2$-structure on $\hat{Y}_t$ for $t<T_{pK}$.
\end{proposition}

Joyce then perturbs the $\G_2$-structure $\tilde{\phi}_t$ to a torsion-free $\G_2$-structure $\phi_t$. In order to state a (slightly enhanced) version of his result we need to introduce the following norms:

First, let $r_t \col \hat{Y}_t \to [0,\kappa]$ be the function defined by 
\begin{align*}
r_t(y) \coloneqq 
\begin{cases}
\vert t \tau(z) \vert &\textup{ if $y=[s,z] \in \hat{\mathcal{V}}_{\kappa}^t$} \\
\kappa & \textup{ if $y \in Y_0 \setminus \mathtt{J}(\mathcal{V}_{\kappa})$}
\end{cases}
\end{align*}

For $t \in (0,T_{pK})$, $k\in \mathbb{N}_0$, $\alpha \in (0,1)$, a weight exponent $\beta \in \mathbb{R}$, and a domain $U \subset \hat{Y}_t$ we define (as in~\cite[Section~6]{Walpuski-InstantonsKummer}) the following weighted Hölder norms:
\begin{align*}
[f\WsHt{0}{\beta}{(U)} &\coloneqq \sup_{x,y,\in U,\ d(x,y)\leq w_t(x,y)} w_t(x,y)^{\alpha-\beta} \frac{\vert f(x)-f(y) \vert}{d(x,y)^\alpha}  \\
\Vert f \VertWCt{0}{\beta}{(U)} & \coloneqq \big\Vert w_t^{-\beta} f \big{\VertC{0}{(U)}} \\
\Vert f \VertWHt{k}{\beta}{(U)} & \coloneqq \sum_{j=0}^k \left\Vert \nabla^j f \right\VertWCt{0}{\beta-j}{(U)} + \left[\nabla^j f\right\WsHt{0}{\beta-j}{(U)}
\end{align*}
where \[w_t(y) \coloneqq t + r_t(y) \quad \text{and} \quad w_t(x,y)\coloneqq \min\{w_t(x),w_t(y)\}.\] Whenever $f$ is a function, then all derivatives, inner products, and distances in the definition of the above norms are taken with respect to $\tilde{g}_t$ (induced from the $\G_2$-structure $\tilde{\phi}_t$) and its associated Levi--Civita connection. If $f$ is a tensor field (or, more generally, a section of a vector bundle equipped with an inner-product and a metric connection), then we use parallel transport to compare two values of $f$ lying in different fibers. If we do not specify a domain, then we consider $U= \hat{Y}_t$.

These norms satisfy (cf.~\cite[Equation~(6.5) and~(6.6)]{Walpuski-InstantonsKummer}
\begin{equation}\label{eq:weighted-t-Höldernorms-multiplication}
\Vert f\cdot g\VertWHt{k}{\beta}{(U)} \leq \Vert f\VertWHt{k}{\beta_1}{(U)}\cdot \Vert g\VertWHt{k}{\beta_2}{(U)} 
\end{equation}  
for $\beta=\beta_1+\beta_2$ and 
\begin{equation}\label{eq:weighted-t-Höldernorms-estimates}
(T_{pK}+\kappa)^{\beta_1-\beta_3} \Vert f \VertWHt{k}{\beta_1}{(U)} \leq \Vert f \VertWHt{k}{\beta_3}{(U)} \leq t^{\beta_2-\beta_3} \Vert f \VertWHt{k}{\beta_2}{(U)} 
\end{equation} 
for $\beta_1 \leq \beta_2\leq \beta_3$.

The following existence theorem was first proven by Joyce in~\cite{Joyce-GeneralisedKummer1}. Its enhancement to $C^{k,\alpha}_{0,t}$-norms comes from~\cite[Proposition~4.20 and Remark~4.26]{Walpuski-InstantonsKummer}.\footnote{\cite[Proposition~4.20 and Remark~4.26]{Walpuski-InstantonsKummer} are formulated for $\G_2$-orbifolds whose singular strata are of codimension $4$. However, the same proof also holds for orbifolds satisfying \autoref{Ass:Codim6Singularities}.}
\begin{theorem}[{\cite{Joyce-GeneralisedKummer2}[Theorem~2.2.1],\cite[Theorem~A]{Joyce-GeneralisedKummer1}, and~\cite[Proposition~4.20 and Remark~4.26]{Walpuski-InstantonsKummer}}]\label{theo:torsionfree_G2_structure}
Let $(Y_0,\phi_0)$ be a compact and flat $\G_2$-orbifold satisfying \autoref{Ass:Codim6Singularities} and let $\mathcal{R}$ be a set of resolution data. Furthermore, let $k\in \mathbb{N}_0$ and $\alpha \in (0,1)$. Then there are positive constants $T_{K} = T_{K}(Y_0,\phi_0,\mathcal{R},\alpha,k)$ with $T_{K} \leq T_{pK}$ and $c_{K}=c_{K}(Y_0,\phi_0,\mathcal{R},\alpha,k)$ such that for any $t\in (0,T_{K})$ there exists a torsion-free $\G_2$-structure $\phi_t$ on $\hat{Y}_t$ with \[ \Vert \phi_t - \tilde{\phi}_t \VertWHt{k}{0}{} < c_{K} t^{1/2}. \]
\end{theorem}

\begin{remark}\label{rem:improvedEstimates}
Going through the proofs of~\cite[Theorem~2.2.1]{Joyce-GeneralisedKummer2} and~\cite[Theorem~A]{Joyce-GeneralisedKummer1} one can check that the factor of $t^{1/2}$ in the estimate of the previous theorem can be improved to $t^{5/2}$. This is due to the observation that whenever $Y_0$ has only codimension $6$ singularities, then the exponents inside condition i) of~\cite[Theorem~A]{Joyce-GeneralisedKummer1} can be improved to $t^{6}$ (cf.~\cite[Section~11.5.5]{Joyce-Black}).
\end{remark}

\begin{remark}[{Equivariant generalised Kummer construction (cf.~\cite[Remark~2.20]{DPSDTW-Associatives}}]\label{rem:EquivariantGeneralisedKummer}
Let $H$ be a finite group and $\lambda_0 \col H \to \textup{Isom}(Y_0,g_{\phi_0})$ be a group action that fixes $\psi_0 \coloneqq * \phi_0$. An $H$-equivariant resolution data is a resolution data with the additional property that for every $h\in H$ and $S\in \mathcal{S}$: \begin{enumerate}
\item We have \[ \Gamma_{hS} = \Gamma_S, \quad \Lambda_{hS}=\Lambda_S, \quad \text{and} \quad \rho_{hS}=\rho_S.\] Furthermore, we demand that for every $[S]\in \mathcal{S}/H$ there is a homomorphism $ \lambda_{[S]} \col H \to \big(N_{(\O(\mathbb{R})\ltimes \mathbb{R})\times (N_{\O(\mathbb{C}^3)}(\Gamma_S)/\Gamma_S)}(\Lambda_S)\big)/\Lambda_S$ that satisfies \[\lambda_0(h)\circ \mathtt{J}_S = \mathtt{J}_{hS} \circ [\lambda_{[S]}(h) ].\] Here $[\lambda_{[S]}(h)]$ denotes the induced map on $\mathcal{V}_{S,\kappa} = \mathcal{V}_{hS,\kappa}$.
\item The resolving ALE spaces also satisfy $(\hat{Z}_S,\tau_S,\hat{\omega}_S,\hat{\Omega}_S)=(\hat{Z}_{hS},\tau_{hS},\hat{\omega}_{hS},\hat{\Omega}_{hS})$.
\item We have a lift $\hat{\lambda}_{[S]} \col H \to \big(N_{(\O(\mathbb{R})\ltimes \mathbb{R) \times}\textup{Isom}(\hat{Z}_{[S]},\hat{g}_{[S]})}(\Lambda_{[S]})\big)/\Lambda_{[S]}$ of $\lambda_{[S]}$ along $\tau_{[S]}$ such that for every $h \in H$ \[[\hat{\lambda}_{[S]}(h)]^*\psi_{[S]} = \psi_{[S]}\] (where $\psi_{[S]}$ is the coassociative 4-form associated to the Calabi--Yau structure $(\hat{\omega}_{[S]},\hat{\Omega}_{[S]})$ as in \autoref{ex_G2Model}).
\end{enumerate}
In this case we obtain an induced action $\hat{\lambda} \col H \to \textup{Diff}(\hat{Y}_t)$ and can choose $\sigma_{[S]}^t$ in the pre-gluing construction to be $H$-invariant. This implies that $\hat{\lambda}(h)^*\tilde{\psi}_t = \tilde{\psi}_t$ and therefore $\hat{\lambda}(h)^*\tilde{g}_t = \tilde{g}_t$ (cf. the following remark) for every $h\in H$. The Kummer construction can then be performed $H$-equivariantly (see for example \cite[Section~4.2]{Joyce-GeneralisedKummer2}) and $\psi_t =*_{\phi_t} \phi_t$ associated to the torsion-free $\G_2$-structure of \autoref{theo:torsionfree_G2_structure} is also invariant under $H$.
\end{remark}
\begin{remark}\label{rem: stabiliser 4-form}
Let $\phi_0 \in \Lambda^3(\mathbb{R}^7)^*$ be the model $\G_2$-structure over $\mathbb{R}^7$ (obtained by choosing $(Z,\omega,\Omega)=(\mathbb{C}^3,\omega_0,\Omega_0)$ in \autoref{ex_G2Model}) and $\psi_0 = * \phi_0 \in \Lambda^4(\mathbb{R}^7)^*$ its coassociative 4-form. It is well known that the stabiliser of $\psi_0$ in $\GL(\mathbb{R}^7)$ equals \[\textup{Stab}_{\GL(\mathbb{R}^7)}(\psi_0) = \G_2 \times \mathbb{Z}_2 < \O(\mathbb{R}^7) \] where $(R,\pm 1) \in \G_2\times \mathbb{Z}_2$ corresponds to the isomorphism $\pm R \in \GL(\mathbb{R}^7)$ (see for example \cite[Proposition~2.1.3]{Gutwein-Diss} for an elementary proof of this fact). Thus, demanding that the group action $\lambda_0 \col H \to \textup{Isom}(Y_0,g_{\phi_0})$ in the previous remark fixes $\psi_0$ is slightly weaker than demanding it to preserve $\phi_0$.  
\end{remark}

\subsection{ALE Calabi--Yau structures on crepant resolutions}
\label{sec:Ricci-Flat ALE metrics on crepant resolutions}

As discussed in the previous section, the generalised Kummer construction requires ALE Calabi--Yau 3-folds resolving orbifolds of the form $\mathbb{C}^3/\Gamma$ where $\Gamma<\SU(3)$ is a finite group acting freely on $\mathbb{C}^3\setminus \{0\}$ (see \autoref{def:ALE-space} of \autoref{def:RData}). This section summarises the construction of such resolutions of $\mathbb{C}^3/\Gamma$ via Kähler reduction as carried out in \cite{Sardo-Infirri-ParitalResolutions} and the existence of ALE Calabi--Yau structures on them as proven in~\cite[Section~4]{WalpuskiDegeratu-HYM_on_crepant_resolutions} via Joyce's proof of the ALE Calabi conjecture~\cite[Theorem~8.2.3]{Joyce-Black}. We then give a condition under which certain group actions on $\mathbb{C}^3/\Gamma$ lift to the resolving ALE space (as required in \autoref{bul: lifting group action in definition of R-data} of \autoref{def:RData}). All the material below except \autoref{bul: lifting group actions to resolution} (which is analogous to the case of ALE hyperkähler 4-manifolds) is contained in~\cite[Sections~2,3,4]{WalpuskiDegeratu-HYM_on_crepant_resolutions} and its presentation is similar to the one given for ALE hyperkähler $4$-manifolds in~\cite[Remark~2.15]{DPSDTW-Associatives} and~\cite[Section~2.2.2]{Gutwein-coassociatives}.

\begin{enumerate}
\item Let $\Gamma<\SU(3)$ be a finite group acting freely on $\mathbb{C}^3\setminus \{0\}$ and denote by $\mathbb{C}[\Gamma]\coloneqq \textup{Maps}(\Gamma,\mathbb{C})$ the regular representation of $\Gamma$ equipped with its standard Hermitian inner product. Furthermore, define $G \coloneqq \mathbb{P}\U(\mathbb{C}[\Gamma])^\Gamma$ and \[ N \coloneqq \left\{ B \in \big( \mathbb{C}^3 \otimes_\mathbb{C} \End(\mathbb{C}[\Gamma]) \big)^\Gamma \mid [B \wedge B]=0 \in \Lambda^2\mathbb{C}^3 \otimes \End(\mathbb{C}[\Gamma]) \right\} \] where $[ \cdot \wedge \cdot]$ is the simultaneous wedge-product $\mathbb{C}^3 \otimes \mathbb{C}^3 \to \Lambda^2 \mathbb{C}^3$ and commutator $\End(\mathbb{C}[\Gamma])\otimes \End(\mathbb{C}[\Gamma])\to \End(\mathbb{C}[\Gamma])$. Equip $(\mathbb{C} \otimes \End(\mathbb{C}[\Gamma]))^\Gamma$ with the canonical (flat) Kähler structure and $N$ with the induced Kähler structure as a subvariety. 

\item The adjoint action of $G$ on $(\mathbb{C}^3 \otimes \End(\mathbb{C}[\Gamma]))^\Gamma$ is Hamiltonian and has a distinguished moment map (written down in~\cite[Proposition~3.1]{WalpuskiDegeratu-HYM_on_crepant_resolutions}) \[\mu \col (\mathbb{C}^3 \otimes \End(\mathbb{C}[\Gamma]))^\Gamma \to \mathfrak{g}^*.\] \cite[Section~4.2]{Sardo-Infirri-ParitalResolutions} proves that the Kähler quotient $(N \cap \mu^{-1}(0))/G$ is isometric to $\mathbb{C}^3/\Gamma$.

\item Let $\mathfrak{z}^*$ be the annihilator of $[\mathfrak{g},\mathfrak{g}]$ in $\mathfrak{g}^*$, i.e. all elements in $\mathfrak{g}^*$ fixed by the coadjoint action of $G$. An element $\zeta \in \mathfrak{z}^*$ is called \[ \textit{generic} \quad \Leftrightarrow \quad \textup{$\zeta (i \pi_R) \neq 0$ for every non-trivial proper subrepresentation $R\subset \mathbb{C}[\Gamma]$} \] where $\pi_R \col \mathbb{C}[\Gamma] \to R$ is the orthogonal projection onto $R$.\footnote{Note that the center of $\mathfrak{u}(\mathbb{C}[\Gamma])^\Gamma$ is spanned by $\{ i\pi_R \mid \textup{for $R \subset \mathbb{C}[\Gamma]$ irreducible subrepresentation} \}$ (since by \autoref{rem: finite subgroups of SU(3)} $\Gamma$ is abelian, all irreducible representations are 1-dimensional). This implies that set of generic elements in $\mathfrak{z}^*$ is given by the complement of a finite union of linear codimension 1 subspaces (cf.~\cite[Section~2,3]{WalpuskiDegeratu-HYM_on_crepant_resolutions}). The set of generic elements is therefore open and dense.}  Furthermore, we define $\mathfrak{g}^*_\mathbb{Q}$ to be the set of linear maps $\zeta \col \mathfrak{g} \to \mathbb{R}$ with $\zeta(i\pi_R) \in \mathbb{Q}$ for every subrepresentation $R\subset \mathbb{C}[\Gamma]$.

\item \label{bul:ALE-Crepant-Resolutions-via-KählerQuotients} If $\zeta \in \mathfrak{z}^*\cap \mathfrak{g}^*_\mathbb{Q}$ is generic, then the Kähler quotient \[\hat{Z}_\zeta \coloneqq \big(N \cap \mu^{-1}(\zeta)\big) /G \] is a smooth Kähler manifold underlying a crepant resolution $\tau_\zeta \col \hat{Z}_\zeta \to \mathbb{C}^3/\Gamma$ (see \cite[Theorem~1.1]{BridegelandKingReid--Mukai_McKay}, \cite[Theorem~2.5]{CrawIshii-Flops}, and~\cite[Proposition~3.6]{WalpuskiDegeratu-HYM_on_crepant_resolutions}).\footnote{Moreover, \cite[Section~2.1]{Yamagishi--Craw_Ishii_conjecture} argues that this result also holds true for generic $\zeta \in \mathfrak{z}^*$ that do not lie in $\mathfrak{g}_\mathbb{Q}^*$.} Conversely, \cite[Theorem~1.1]{CrawIshii-Flops} proves that every projective crepant resolution is biholomorphic to $\hat{Z}_\zeta$ for some generic $\zeta\in \mathfrak{z}^*\cap \mathfrak{g}^*_\mathbb{Q}$. \cite[Theorem~5.1]{Sardo-Infirri-ParitalResolutions} shows that the induced Kähler structure $\tilde{\omega}_\zeta$ on $\hat{Z}_\zeta$ is ALE of rate (at least) $4$ (i.e. it satisfies the conditions of \autoref{def:ALE-space} of \autoref{def:RData} with exponents $r^{-4-k}$).

\item Using Joyce's proof of the ALE Calabi conjecture~\cite[Theorem~8.2.3]{Joyce-Black},~\cite[Theorem~4.1]{WalpuskiDegeratu-HYM_on_crepant_resolutions} shows that there exists a unique Ricci-flat ALE Kähler structure $\hat{\omega}_\zeta$ on $\hat{Z}_\zeta$ in the Kähler class of $\tilde{\omega}_\zeta$. This $\hat{\omega}_\zeta$ is ALE of order $6$.

\item \label{bul: lifting group actions to resolution} Let $U$ be an element in the normalizer $N_{\U(\mathbb{C}^3)}(\Gamma)$. Then $U$ acts on $\Gamma$ by conjugation and we extend this action to a complex linear map $\textup{conj}_{U} \in \U(\mathbb{C}[\Gamma])$. The standard representation of $U$ on $\mathbb{C}^3$ tensored with the adjoint action of $\textup{conj}_{U}$ on $\End(\mathbb{C}[\Gamma])$ induces an action on $(\mathbb{C}^3 \otimes \End(\mathbb{C}[\Gamma]))^\Gamma$ which restricts to $N$. The moment map satisfies \[ \mu \circ (U \otimes \Ad_{\textup{conj}_{U}}) = \Ad_{\textup{conj}_{U}^{-1}}^* \circ \mu. \] Thus, if $\Ad^*_{\textup{conj}_{U}^{-1}} \zeta = \zeta$, we obtain a holomorphic isometry $\hat{U} \in \textup{Isom}(\hat{Z}_\zeta,\tilde{\omega}_\zeta)$ satisfying $\tau_\zeta \circ \hat{U} = U \circ \tau_\zeta$. By the uniqueness of the Ricci-flat ALE metric inside the same Kähler class as $\tilde{\omega}_\zeta$, we also obtain $\hat{U}^*\hat{\omega}_\zeta = \hat{\omega}_\zeta$.
\end{enumerate}

\subsection{HYM connections over ALE Calabi--Yau 3-folds}

In \autoref{sec:approxG2instantons_from_gluingdata} we construct approximate $\G_2$-instantons over resolutions of $\G_2$-orbifolds satisfying \autoref{Ass:Codim6Singularities}. One ingredient of this construction are asymptotically flat Hermitian Yang--Mills connections over ALE $3$-folds whose theory we review in this section.

Let $(\hat{Z},\tau,\hat{\omega},\hat{\Omega})$ be an ALE Calabi--Yau $3$-fold asymptotic to $\mathbb{C}^3/\Gamma$ as defined in \autoref{def:ALE-space} of \autoref{def:RData}. Furthermore, let $G$ be a Lie group whose Lie algebra $\mathfrak{g}$ has been equipped with an $\Ad$-invariant inner product and $\pi \col \hat{P} \to \hat{Z}$ be a principal $G$-bundle. We denote the space of connections on $\hat{P}$ by $\mathcal{A}(\hat{P})$ and for any $A \in \mathcal{A}(P)$ its associated connection $1$-form by $\theta_A \in \Omega^1(\hat{P},\mathfrak{g})^G$. For the following discussion we fix 
\begin{itemize}
\item a local framing at infinity, i.e. a principal $G$-bundle $ \pi_\infty \col P_\infty \to (\mathbb{C}^3\setminus \{0\})/\Gamma$ together with a bundle isomorphism $\tilde{\tau} \col \hat{P}_{\vert \hat{Z}\setminus \tau^{-1}(0)} \to P_\infty$ covering $\tau$, and 
\item a flat connection $A_\infty$ on $P_\infty$.
\end{itemize} 
\begin{definition}
For $\beta \in \mathbb{R}$ we denote by $\mathcal{A}_\beta(\hat{P},A_\infty)$ the space of connections $A\in \mathcal{A}(\hat{P})$ that satisfy \[\vert \nabla^k (\tilde{\tau}_* \theta_A - \theta_{A_\infty} ) \vert = \mathcal{O}(r^{\beta-k})\textup{ as $r \to \infty$.}\] The covariant derivative is hereby induced by $A_\infty$ and the  Levi--Civita connection associated to the flat metric on $(\mathbb{C}^3\setminus \{0\})/\Gamma$. Similarly, the norm is induced by the flat metric and the $\Ad$-invariant inner product on $\mathfrak{g}$.

Similarly, for any associated vector bundle $\hat{P} \times_G V$ (where $V$ is a linear $G$-module) we define $\Omega^k_\beta(\hat{Z},\hat{P}\times_G V)$ as the space of $\omega \in \Omega^k(\hat{Z},\hat{P}\times_G V)$ with \[ \vert \nabla^k\tilde{\tau}_*\omega \vert = \mathcal{O}(r^{\beta-k}) \textup{ as $r\to \infty$.}\]
\end{definition}
\begin{definition}\label{def:HYMconnection}
A connection $A\in \mathcal{A}_\beta(\hat{P},A_\infty)$ is called \ita{Hermitian Yang--Mills asymptotic to $A_\infty$ of order $\beta$} if \[ \Lambda_{\hat{\omega}} F_A = 0 \qquad \textup{and} \qquad F_A^{0,2} = 0 \] where $\Lambda_{\hat{\omega}}$ is the dual Lefschetz operator associated to $\hat{\omega}$. The set of hermitian Yang--Mills connections asymptotic to $A_\infty$ of order $\beta$ is denoted by $ \mathcal{A}^{\textup{HYM}}_{\beta}(\hat{P},A_\infty)$.
\end{definition}
\begin{remark}
If no order is specified we will always assume $\beta = -5$.
\end{remark}
\begin{remark}\label{rem:HYMaltCondition}
In the following we tacitly assume that $G$ is a \ita{real} Lie group. In the above definition we therefore complexify the adjoint bundle in order to project to $F_A^{0,2}$. An equivalent condition in this case (with the benefit of not requiring the extra step of complexification) is (cf. \cite[Proposition~1.2.30]{Huybrechts-Complex}) \[ \hat{\omega} \wedge \hat{\omega} \wedge F_A = 0 \qquad \textup{and} \qquad \re \hat{\Omega} \wedge F_A = 0. \]
\end{remark}

In the following, we denote by $\mathfrak{g}_{\hat{P}}^\mathbb{C} \coloneqq \mathfrak{g}_{\hat{P}}\otimes_\mathbb{R} \mathbb{C}$ the complexified adjoint bundle.

\begin{proposition}\label{prop:HYMConnection improve decay}
Assume that $A \in \mathcal{A}_\beta^{\textup{HYM}}(\hat{P},A_\infty)$ for some $\beta\in(-5,-1)$. Then there exists a \[u\in \mathcal{G}_{0,\beta+1}\coloneqq \{u \in  \mathcal{G} \mid \vert \nabla^k(\tilde{\tau}_* u - \textup{Id} ) \vert = \mathcal{O}(r^{\beta+1-k})) \}\] such that $u^*A \in \mathcal{A}_{-5}^{\textup{HYM}}(\hat{P},A_\infty)$. Here, $\mathcal{G}$ denotes the group of bundle automorphisms of $\hat{P}$ that cover the identity. Furthermore, we assumed (for simplicity) that $G$ is a subgroup of $\GL(V)$ for a vector space $V$ (and $\mathcal{G}$ can therefore be regarded as a subset of the linear bundle $\End(\hat{P}\times_G V)$) in order to make sense of the difference $\tilde{\tau}_*u-\id$.
\end{proposition}
\begin{remark}
We made the assumption $\beta<-1$ in order to ensure that the based gauge group $\mathcal{G}_{0,\beta+1}$ acts on $\mathcal{A}_\beta(\hat{P},A_\infty)$. This is used in the proof when the connection is put in Coulomb gauge.
\end{remark}
A similar result was stated without proof in~\cite[Remark~4.8]{WalpuskiDegeratu-HYM_on_crepant_resolutions}. We outline a proof for the convenience of the reader.
\begin{proof}
Denote by $\mathcal{A}_{\textup{cpt}}(\hat{P},A_\infty)$ the set of connections which are equal to $A_\infty$ outside a compact set. There exists a $B \in \mathcal{A}_{\textup{cpt}}(\hat{P},A_\infty)$ which can be put in Coulomb gauge relative to $A$.\footnote{This can be proven as~\cite[Theorem~3.2]{FreedUhlenbeck-Instantons} by using the results of~\cite[Section~1 and~2]{Bartnik-MassofALF}; cf.~\cite[Proposition~2.3]{Nakajima-ASDonALE}. When working with weighted Sobolev norms, one has to slightly enlarge $\beta < \beta^\prime<-1$. A similar argument as in the current proof then improves the decay of the gauge transformation $u\in \mathcal{G}_{0,\beta^\prime+1}$ to $\beta+1$} That is, there exists a $u\in \mathcal{G}_{0,\beta+1}$ such that $\diff_A^*((u^{-1})^*\theta_B-\theta_A)=0$, or, equivalently, $\diff_B^*(u^*\theta_A-\theta_B)=0$. 

Next, we write $u^*\theta_A -\theta_B \eqqcolon \alpha+\bar{\alpha}$ with $\alpha\in \Omega^{0,1}_\beta(\hat{Z},\mathfrak{g}_{\hat{P}}^\mathbb{C})$. Together with the gauge-fixing, the Hermitian Yang--Mills equation on $A = B + (\alpha +\bar{\alpha})$ becomes
\begin{align*}
2\bar{\partial}_B\alpha &= -2F_B^{0,2}-[\alpha\wedge\alpha]\\
2\bar{\partial}_B^* \alpha &= 2\Re(\bar{\partial}_B^* \alpha) + 2i \im(\bar{\partial}^*_B \alpha) + i \Lambda_{\hat{\omega}} \diff_B(\alpha+\overline{\alpha}) - i \Lambda_{\hat{\omega}} \diff_B(\alpha +\bar{\alpha}) \\ &= 2i \im(\bar{\partial}^*_B\alpha) + i \Lambda_{\hat{\omega}} \diff_B(\alpha+\overline{\alpha}) +  i\Lambda_{\hat{\omega}} F_B + i\Lambda_{\hat{\omega}} [\alpha\wedge \bar{\alpha}].
\end{align*}

Outside of a compact subset of $\hat{Z}$ the connection $B$ agrees with $A_\infty$ and the twisted Kähler identites~\cite[Lemma~5.2.3]{Huybrechts-Complex} imply therefore that \[ \Lambda_{\hat{\omega}} \diff_B (\alpha + \bar{\alpha}) = \Lambda_{\hat{\omega}} \partial_B \alpha + \Lambda \bar{\partial}_B \bar{\alpha} = -2 \Im(\bar{\partial}^*_B \alpha) \] holds outside of this subset. This shows that we can write the equations above schematically as $D \alpha = f + Q(\alpha)$ where $D= 2(\bar{\partial}_B+\bar{\partial}_B^*)$, $f$ has compact support, and $\vert \nabla^k Q(\alpha)\vert = \mathcal{O}(r^{2\beta-k})$. Thus, $D\alpha$ has more decay than expected (this again uses $\beta<-1$). Up to a constant, $D$ is a twisted Dirac operator and its critical rates (which correspond to the eigenvalues of the Dirac operator on $S^5/\Gamma$ shifted by $-\frac{5}{2}$; cf. \cite[Proof of Theorem~6.1]{WalpuskiDegeratu-HYM_on_crepant_resolutions}) do not intersect the interval $(-5,0)$ (cf.~\cite[Proposition~6.1]{Bartnik-MassofALF}). This implies that the decay of $\alpha$ can iteratively be improved until $\beta = -5$ (cf.~\cite[Theorem~1.14]{Bartnik-MassofALF},~\cite[Theorem~1.17]{Bartnik-MassofALF}, and~\cite[Section~I.5]{LockhartMcOwen-ellipticOperators_nonCompactMfds}).
\end{proof}

\begin{definition}\label{def:HYM_infinitesimal_deformations}
For $A \in \mathcal{A}^{\textup{HYM}}_\beta(\hat{P},A_\infty)$ we define the space of infinitesimal deformations as  \[\mathcal{H}^1_{A,\beta}\coloneqq \left\{\alpha\in \Omega^{0,1}_\beta(Z,\mathfrak{g}_{\hat{P}}^\mathbb{C}) \mid (\bar{\partial}_A + \bar{\partial}_A^*)\alpha=0 \right\}. \] The connection is called \ita{infinitesimally rigid} if $\mathcal{H}^1_{A,\beta} = 0$.
\end{definition}

The following proposition can be proven as \autoref{prop:HYMConnection improve decay}:
\begin{proposition}[{\cite[Proposition~5.5]{WalpuskiDegeratu-HYM_on_crepant_resolutions}}]\label{prop:kernel-HYM-decay}
Let $A \in \mathcal{A}^{\textup{HYM}}_\beta(\hat{P},A_\infty)$ for $\beta<0$. Then every $\alpha \in \mathcal{H}^1_{A,\beta}$ lies in $\mathcal{H}^1_{A,-5}$. (In fact, any $\alpha \in \ker(\bar{\partial}_A+\bar{\partial}_A^*) \cap \Omega^{0,1}(\hat{Z},\mathfrak{g}_{\hat{P}}^\mathbb{C})$ that satisfies $\lim_{r\to \infty} \vert \alpha \vert = 0$ lies in $\mathcal{H}^1_{A,-5}$.)
\end{proposition}

\begin{lemma}\label{lem:delbar_harmonic_with_decay}
Let $\beta<0$ and $A\in \mathcal{A}^{\textup{HYM}}_\beta(\hat{P},A_\infty)$ and assume that either \[\eta \in \Omega^{0,0}_\beta(\hat{Z},\mathfrak{g}_{\hat{P}}^\mathbb{C})\quad \textup{satisfies} \quad\bar{\partial}_A^*\bar{\partial}_A\eta = 0\] or  \[\eta \in \Omega^{0,3}_\beta(\hat{Z},\mathfrak{g}_{\hat{P}}^\mathbb{C})\quad \textup{satisfies} \quad\bar{\partial}_A\bar{\partial}_A^*\eta = 0.\] Then $\eta =0$.
\end{lemma}
This result is for example contained in~\cite[Proof of Proposition~6.4]{WalpuskiDegeratu-HYM_on_crepant_resolutions}. As the argument is short, we have included a (slightly simplified) version of the proof.
\begin{proof}
Since $g_{\hat{\omega}}$ is Ricci-flat and $A$ is Hermitian Yang--Mills, the Weitzenböck formula (see for example~\cite[Corollary~1.1.4]{Oliveira-MonopolesDiss}) implies that \[\nabla_A^*\nabla_A \eta = 2 \bar{\partial}_A^*\bar{\partial}_A\eta = 0 \quad \textup{if $\eta \in \Omega^{0,0}_\beta(\hat{Z},\mathfrak{g}_{\hat{P}}^\mathbb{C})$}\] or \[\nabla_A^*\nabla_A \eta = 2 \bar{\partial}_A\bar{\partial}_A^*\eta = 0 \quad \textup{if $\eta \in \Omega^{0,3}_\beta(\hat{Z},\mathfrak{g}_{\hat{P}}^\mathbb{C})$}.\] Together with the Bochner formula this leads to \[ \Delta \vert \eta \vert^2 = 2 \langle \nabla_A^*\nabla_A \eta,\eta \rangle - 2 \vert \nabla_A \eta \vert^2 \leq 0. \] The function $\vert \eta \vert^2$ is therefore subharmonic and decays to zero at infinity. Thus, by the maximum principle, $\vert \eta \vert^2 = 0$.
\end{proof}

Let now $\Gamma<\SU(3)$ be a fixed finite subgroup acting freely on $\mathbb{C}^3\setminus\{0\}$ and let $\zeta \in \mathfrak{z}^* \cap (\mathfrak{pu}(\mathbb{C}[\Gamma])^\Gamma)^*_\mathbb{Q}$ be generic as defined in \autoref{sec:Ricci-Flat ALE metrics on crepant resolutions}. Let $(\hat{Z}_\zeta,\tau_\zeta,\hat{\omega}_\zeta,\hat{\Omega}_\zeta)$ be the corresponding ALE Calabi--Yau manifold asymptotic to $\mathbb{C}^3/\Gamma$ as described in \autoref{sec:Ricci-Flat ALE metrics on crepant resolutions}. The following proposition follows from the work 
in~\cite{WalpuskiDegeratu-HYM_on_crepant_resolutions} and \autoref{prop:HYMConnection improve decay} (or~\cite[Remark~4.8]{WalpuskiDegeratu-HYM_on_crepant_resolutions}). The last statement about the monodromy of the asymptotic connection can be proven as~\cite[Proposition~2.2 (ii)]{KronheimerNakajima-YangMillsInstantons_on_ALE}.
\begin{proposition}[{\cite[Theorem~1.1]{WalpuskiDegeratu-HYM_on_crepant_resolutions} and~\cite[Proof of Lemma~5.1]{WalpuskiDegeratu-HYM_on_crepant_resolutions}}]\label{prop:irred_HYM_on_tautological_bundles}
For every unitary representation $\nu \col \Gamma \to \U(n)$ there exists a locally framed $\U(n)$-bundle $\pi_\nu \col \hat{P}_\nu \to \hat{Z}_\zeta$ together with an infinitesimally rigid HYM connection $A_\nu \in \mathcal{A}_{-5}^{\textup{HYM}}(P_\nu,A_\infty)$ asymptotic to the flat connection $A_\infty$ over $(\mathbb{C}^3\setminus\{0\})/\Gamma$ whose monodromy corresponds to $\nu$.
\end{proposition} 
\begin{remark}
Since $\Gamma<\SU(3)$ is abelian by \autoref{rem: finite subgroups of SU(3)}, the locally framed $\U(n)$-bundles and connections in the previous proposition decompose into a direct sum of $\U(1)$-bundles and connections.
\end{remark}
Recall from \autoref{sec:Ricci-Flat ALE metrics on crepant resolutions} \autoref{bul: lifting group actions to resolution} that for every $U\in N_{\U(\mathbb{C}^3)}(\Gamma)$ with $\Ad_{\textup{conj}_{U}^{-1}}^*\zeta = \zeta$ there exists a holomorphic lift $\hat{U}$ to $\hat{Z}_\zeta$ preserving the Calabi--Yau metric. This implies, in particular, that for every element $U\in C_{\U(\mathbb{C}^3)}(\Gamma)$ in the centraliser of $\Gamma$ we obtain such a lift $\hat{U}$ (because conjugation by $U$ induces the trivial action on $\mathbb{C}[\Gamma]$; i.e. $\textup{conj}_{U} = 1$ in the notation of the aforementioned paragraph). The following can be deduced from~\cite[Sections~3 and~4]{WalpuskiDegeratu-HYM_on_crepant_resolutions} and \autoref{prop:HYMConnection improve decay} (note for this that the gauge transformation in \autoref{prop:HYMConnection improve decay} can be chosen invariantly):
\begin{lemma}\label{lem:group_lift_to_HYMbundle}
Let $\nu \col \Gamma \to \U(n)$ be a representation and $U\in C_{\U(\mathbb{C}^3)}(\Gamma)$ be an element of the centraliser of $\Gamma$. Denote by $\hat{U}$ the induced isometry on $\hat{Z}_\zeta$. There exists a lift of $\hat{U}$ to a bundle isomorphism of $\hat{P}_\nu$ which preserves the HYM connection $A_\nu$.
\end{lemma}
\begin{remark}
This remark explains the bundle isomorphism of the previous lemma and, in particular, the condition that $U \in C_{\U(\mathbb{C}^3)}(\Gamma)$ in more detail. Recall from \autoref{sec:Ricci-Flat ALE metrics on crepant resolutions} that $\hat{Z}_\zeta$ is constructed via Kähler reduction \[ \hat{Z}_\zeta = (N\cap \mu^{-1}(\zeta))/\mathbb{P}\U(\mathbb{C}[\Gamma])^\Gamma.\] Similarly, the $\U(n)$-bundle of \autoref{prop:irred_HYM_on_tautological_bundles} is constructed in \cite[Equation~(3.4)]{WalpuskiDegeratu-HYM_on_crepant_resolutions} as the associated bundle \[ \hat{P}_\nu \coloneqq (N\cap \mu^{-1}(\zeta))\times_{\mathbb{P}\U(\mathbb{C}[\Gamma])^\Gamma} \U(n) \] (where the homomorphism $\mathbb{P}\U(\mathbb{C}[\Gamma])^\Gamma \to \U(n)$ is in a certain way induced by the representation $\nu$). If $U\in C_{\U(\mathbb{C}^3)}(\Gamma)$, then the induced map $\textup{conj}_U\in \U(\mathbb{C}[\Gamma])$ is trivial. In this case, the action of $U$ on $N$ described in \autoref{bul: lifting group actions to resolution} of \autoref{sec:Ricci-Flat ALE metrics on crepant resolutions} commutes with the action of $\mathbb{P}\U(\mathbb{C}[\Gamma])^\Gamma$ (and restricts to $N\cap \mu^{-1}(\zeta)$). This induces the action on $\hat{P}_\nu$.
\end{remark}

In \autoref{sec:Example18} we use the Hermitian Yang--Mills connections of \autoref{prop:irred_HYM_on_tautological_bundles} to construct $\G_2$-instantons on a certain $\SO(14)$-bundle. For this, we consider the connections of \autoref{prop:irred_HYM_on_tautological_bundles} as $\SO(2n)$-connections (induced by the inclusion $\U(n) < \SO(2n)$). The following proposition ensures that these connections are still infinitesimally rigid (and therefore satisfy \autoref{ass:invertible_linearisations}) when considered as $\SO(2n)$-connections. 

\begin{proposition}\label{prop:rigid_SO(n)-connections}
For a unitary representation $\nu \col \Gamma \to \U(n)$ let $(\pi_\nu \col \hat{P}_\nu \to \hat{Z}_\zeta,A_\nu)$ be as in \autoref{prop:irred_HYM_on_tautological_bundles}. Furthermore, denote by $\hat{E}_\nu$ the unitary vector bundle associated to $\hat{P}_\nu$. For any $k \in \mathbb{N}_0$ we define $\pi\col \SO(\underline{\mathbb{R}}^k\oplus \hat{E}_\nu) \to \hat{Z}$ to be the $\SO(2n+k)$-bundle associated to the Euclidean vector bundle $\underline{\mathbb{R}}^k \oplus \hat{E}_\nu$. This comes equipped with a canonical local framing and an Hermitian Yang--Mills connection $A$ (induced by $A_\nu$ and the trivial connection). This connection is infinitesimally rigid and asymptotic to the flat connection associated to the (real) representation $\nu_{\textup{triv}}^{\oplus k} \oplus \nu$.
\end{proposition}
A similar statement and proof for ALE hyperkähler 4-manifolds appeared in \cite[Proof of Proposition~7.8]{GaldeanoPlattTanakaWang-spin7instantons} and~\cite[Proposition~5.22]{Walpuski-InstantonsKummer}.
\begin{proof}
We only prove the infinitesimal rigidity of $A$. For this note that
\begin{align*}
\mathfrak{so}\big(\underline{\mathbb{R}^k} \oplus \hat{E}_\nu\big)^\mathbb{C}  &\cong \mathfrak{so}_\mathbb{C}\big( \underline{\mathbb{C}^k} \oplus \hat{E}_\nu^\mathbb{C}\big) \cong \mathfrak{so}_\mathbb{C}\big( \underline{\mathbb{C}^k} \oplus \hat{E}_\nu \oplus \hat{E}_\nu^* \big) \cong \mathfrak{so}_\mathbb{C}\big( \underline{\mathbb{C}^k} \oplus \hat{E}_\nu \oplus \hat{E}_{\nu^*} \big) 
\end{align*}
and that $A$ is mapped to the connection on the right-hand bundle that is induced by the trivial connection on $\underline{\mathbb{C}^k}$ and $A_{\nu}\oplus A_{\nu^*}$ on $\hat{E}_{\nu}\oplus \hat{E}_{\nu^*}$. Furthermore, this bundle is a parallel subbundle of $\End_\mathbb{C}(\underline{\mathbb{C}^k}\oplus \hat{E}_\nu \oplus \hat{E}_{\nu^*}\big)=\mathfrak{u}\big(\underline{\mathbb{C}^k}\oplus \hat{E}_\nu \oplus \hat{E}_{\nu^*}\big)^\mathbb{C}$. By \autoref{prop:irred_HYM_on_tautological_bundles}, $\mathcal{H}^1_{A_{\textup{triv}} \oplus A_\nu \oplus A_{\nu^*},-5} = 0$ and therefore also $\mathcal{H}^1_{A,-5} = 0$.
\end{proof}

\section{$\G_2$-instantons over generalised Kummer constructions} \label{sec: instantons over Kummer}

The first part of this section reviews the theory of $\G_2$-instantons over generalised Kummer constructions relevant for the rest of this article. In particular, we discuss the (non-linear) instanton map and prove an estimate on its quadratic part. We then give in \autoref{theo:perturbing_almost_instantons} sufficient conditions under which a family of connections over a generalised Kummer construction may be perturbed to a family of $\G_2$-instantons. Furthermore, we prove estimates on the derivatives of this family and show under which hypothesis the instantons in the family are infinitesimally irreducible.

\subsection{Gauge theory on generalised Kummer constructions}\label{subsec: gauge theory on Kummer}

Let $(Y,\phi)$ be a $7$-manifold together with a coclosed $\G_2$-structure and $\pi \col P \to Y$ a principal $G$-bundle where $G$ is a Lie group whose Lie algebra $\mathfrak{g}$ has been equipped with an $\Ad$-invariant inner product.
\begin{definition}[{cf.~\cite[Proposition~3.7]{Walpuski-InstantonsKummer}}]\label{def:G2-instanton}
A connection $A\in \mathcal{A}(P)$ is called a $\G_2$-instanton if $F_A \wedge \psi = 0$ where $\psi = *\phi$.
\end{definition}
The following two examples are well-known.
\begin{example}
Any flat connection over $Y$ is a $\G_2$-instanton.
\end{example}
\begin{example}\label{ex:HYM_is_G2-instanton}
Let $(Z,\omega,\Omega)$ be a Calabi--Yau $3$-fold and $\pi \col P \to Z$ be a principal bundle over $Z$. Furthermore, let $A \in \mathcal{A}(P)$ be a Hermitian Yang--Mills connection as defined in~\autoref{def:HYMconnection}. Its pullback to $\mathbb{R}\times Z$ is a $\G_2$-instanton when $\mathbb{R}\times Z$ is equipped with the $\G_2$-structure described in \autoref{ex_G2Model}. This follows immediately from \autoref{rem:HYMaltCondition}.
\end{example}
\begin{proposition}[{\cite[Proposition~3.7]{Walpuski-InstantonsKummer}}]\label{prop:AugmentedInstanton}
If $Y$ is compact, then a connection $A\in \mathcal{A}(P)$ is a $\G_2$-instanton if and only if there exists a $\xi \in \Omega^7(Y,\mathfrak{g}_P)$ such that 
\begin{equation}\label{eq:AugmentedInstanton}
F_A \wedge \psi - \diff_A^* \xi = 0.
\end{equation}
\end{proposition}
Since the proof is rather short, we have included it here in order to be self-contained.
\begin{proof}
If $A\in \mathcal{A}(P)$ is a $\G_2$-instanton, then it satisfies~\eqref{eq:AugmentedInstanton} with $\xi=0$. Conversely, assume that $A$ and $\xi$ satisfy~\eqref{eq:AugmentedInstanton}. The Bianchi-identity and closedness of $\psi$ imply \[-\diff_A\diff_A^* \xi = \diff_A\big( F_A\wedge\psi-\diff_A^*\xi)\big) = 0.\] Thus, $\xi$ is harmonic and therefore $\diff_A^* \xi =0$ because $Y$ is closed.
\end{proof}
\begin{remark}
The advantage of the augmented equation~\eqref{eq:AugmentedInstanton} is that its linearisation (together with the Coulomb gauge condition) is elliptic.
\end{remark}

Throughout the rest of this section we denote by $(\hat{Y}_t,\phi_t)$ for $t \in (0,T_K)$ a $1$-parameter family of $\G_2$-manifolds obtained from Joyce's generalised Kummer construction (as specified in \autoref{def:OrbifoldResolution} and \autoref{theo:torsionfree_G2_structure}). Furthermore, let $\pi \col P \to \hat{Y}_t$ be a principal $G$-bundle.\footnote{Since the diffeomorphism type of $\hat{Y}_t$ does not change, we can assume that $\pi\col P \to \hat{Y}_t$ is (up to isomorphism) defined for all $t\in(0,T_K)$.} For any $A \in \mathcal{A}(P)$ we define 
\begin{align*}
\Upsilon_A \col \Omega^1(\hat{Y}_t,\mathfrak{g}_P)\oplus \Omega^7(\hat{Y}_t,\mathfrak{g}_P) &\to \Omega^6(\hat{Y}_t,\mathfrak{g}_P)\oplus \Omega^0(\hat{Y}_t,\mathfrak{g}_P) \\
(a,\xi) &\mapsto \big(F_{A+a} \wedge\psi_t  - \diff_{A+a}^{\tilde{*}_t}\xi,\diff_A^{\tilde{*}_t} a \big)
\end{align*}
where we have written $\diff^{\tilde{*}_t}$ to emphasise that the formal adjoints are taken with respect to the metric and orientation induced by $\tilde{\phi}_t$ (the closed but not torsion-free $\G_2$-structure on $\hat{Y}_t$ as described below \autoref{def:OrbifoldResolution}). In the following, we will omit this explicit dependence on $\tilde{\phi}_t$ from our notation. $\Upsilon_A$ can be written as \[\Upsilon_A(a,\xi) = \Upsilon_A(0) + L_A (a,\xi) + Q_A(a,\xi)\] for 
\begin{align*}
L_A \coloneqq \begin{pmatrix}
\psi_t \wedge \diff_A & -\diff_A^* \\
\diff_A^* & 0
\end{pmatrix} 
\quad \textup{and} \quad Q_A(a,\xi) \coloneqq \left(\tfrac{1}{2} [a\wedge a]\wedge \psi_t + [i_{a^\sharp}, \xi],0 \right)
\end{align*}
where $[i_{a^\sharp}, \xi]$ denotes insertion on the level of forms and the Lie bracket on $\mathfrak{g}_P$. Whenever $\Upsilon_A(a,\xi)=0$, then $A+a$ is a $\G_2$-instanton (in Coulomb gauge relative to $A$) by \autoref{prop:AugmentedInstanton}. 
\begin{remark}\label{rem:definition_UpsilonA_different_adjoints}
In the definition of $\Upsilon_A$ we used the formal adjoints with respect to $\tilde{g}_t$ because the Hodge star $*_{\tilde{\phi}_t}$ is a parallel isometry with respect to $\tilde{g}_t$. This makes estimates in norms taken with respect to this metric slightly simpler. However, since the Hodge star is determined by a (universal) smooth function of the $\G_2$-structure, we have by \autoref{theo:torsionfree_G2_structure} \[ \Vert *_{\phi_t} \VertWHt{k}{0}{} \leq \tilde{c} \Vert \phi_t- \tilde{\phi}_t \VertWHt{k}{0}{} + \Vert *_{\tilde{\phi}_t} \VertWHt{k}{0}{} \leq c \] for all $t<T_K$ and a $t$-independent constant $c>0$ (where we consider the Hodge star as a section of the bundle $\End(\Lambda^* T^*\hat{Y}_t)$ with its induced inner product). This implies that all statements below also hold if we take the adjoints in the definition of $\Upsilon_A$ with respect to $g_t$ (cf.~\cite[Sections~6,7,8]{Walpuski-InstantonsKummer}).
\end{remark}
\begin{proposition}[{cf.~\cite[Section~8]{Walpuski-InstantonsKummer}}]\label{prop:DecompositionNonlinearMap}
The linearisation $L_A$ is an elliptic operator whose formal adjoint satisfies $L_A^{\tilde{*}_t} = \tilde{*}_t L_A \tilde{*}_t$ (in the following we will again drop the dependence on $\tilde{\phi}_t$ from our notation). Furthermore, for any $k\in \mathbb{N}_0$ and $\alpha\in(0,1)$, there exists a $c_Q>0$ such that for any $t\in(0,T_K)$, $\beta
\in \mathbb{R}$, and $\underline{a}_1,\underline{a}_2 \in \Omega^1(\hat{Y}_t,\mathfrak{g}_P)\oplus \Omega^7(\hat{Y}_t,\mathfrak{g}_P)$ we have \[ \Vert Q_A(\underline{a}_1) - Q_A(\underline{a}_2)\VertWHt{k}{2\beta}{} \leq c_Q \big( \Vert \underline{a}_1 \VertWHt{k}{\beta}{} + \Vert \underline{a}_2 \VertWHt{k}{\beta}{}\big) \Vert \underline{a}_1 - \underline{a}_2 \VertWHt{k}{\beta}{}. \] 
\end{proposition}
\begin{proof}
The first part follows from $\diff^{*}_{A\vert \Omega^k} = (-1)^{7(k-1)+1} *\diff_A*$, $*^2=1$, $i_{\psi_t^\sharp} = *( \psi_t \wedge (* \cdot))$, and the closedness of $\psi_t$. In order to prove the quadratic estimate, we write $\underline{a}_i \coloneqq (a_i,\xi_i)$ and note that $[i_{{a_i}^\sharp},\xi_i] = \tilde{*}_t [a_i \wedge \tilde{*}_t\xi_i]$. Since $\psi_t = *_{\phi_t} \phi_t$ is determined by a (universal) smooth function on the $\G_2$-structure we have by \autoref{theo:torsionfree_G2_structure} \[ \Vert \psi_t \VertWHt{k}{0}{} \leq \Vert \psi_t - \tilde{\psi}_t \VertWHt{k}{0}{} + \Vert \tilde{\psi}_t \VertWHt{k}{0}{} \leq c \] for sufficiently small $t$ and a $t$-independent $c>0$. We therefore obtain by~\eqref{eq:weighted-t-Höldernorms-multiplication}
\begin{align*}
\Vert Q_A(\underline{a}_1) - Q_A(\underline{a}_2)\VertWHt{k}{2\beta}{} & \leq c \big( \Vert [a_1\wedge a_1]-[a_2 \wedge a_2] \VertWHt{k}{2\beta}{}  + \Vert [a_1\wedge*\xi_1]-[a_2\wedge*\xi_2] \VertWHt{k}{2\beta}{} \big) \\
& \leq c \big(\big( \Vert [(a_1-a_2)\wedge (a_1+a_2)] \VertWHt{k}{2\beta}{} +\Vert [a_1\wedge*(\xi_1-\xi_2)]\VertWHt{k}{2\beta}{} \\
& \qquad  +\Vert [a_1-a_2\wedge*\xi_2]\VertWHt{k}{2\beta}{} \big) \\
& \leq c \big( \Vert \underline{a}_1 \VertWHt{k}{\beta}{} + \Vert \underline{a}_2 \VertWHt{k}{\beta}{}\big) \Vert \underline{a}_1 - \underline{a}_2 \VertWHt{k}{\beta}{}.  \qedhere
\end{align*}
\end{proof}

\begin{remark}
Note that $Q_A$ and therefore also the constant $c_Q$ in the previous proposition are independent of the connection $A\in \mathcal{A}(P)$.
\end{remark}

Let $\lambda\col H \to \textup{Isom}(\hat{Y}_t,\tilde{g}_{t})$ now be an action by a (possibly trivial) group $H$ preserving $\psi_t$.
\begin{definition}
A lift of $\lambda \col H \to \textup{Isom}(\hat{Y}_t,\tilde{g}_{t})$ to $\pi \col P \to \hat{Y}_t$ is a homomorphism $\tilde{\lambda}\col H \to \textup{Isom}(P)$ such that for every $h\in H$ the bundle isomorphism $\tilde{\lambda}(h)$ covers $\lambda(h)$ (i.e. $\lambda(h) \circ \pi = \pi \circ \tilde{\lambda}(h)$).
\end{definition}
\begin{lemma}
The lift $\tilde{\lambda} \col H \to \textup{Isom}(P)$ induces an $H$-action on $\Omega^k(\hat{Y}_t,\mathfrak{g}_P)$ defined by \[ \big(\tilde{\lambda}(h)^*\eta \big) (v_1,\dots,v_k) \coloneqq \tilde{\lambda}(h)^{-1}\big(\eta(D\lambda(h)(v_1),\dots,D\lambda(h)(v_k))\big) \] for $\eta \in \Omega^k(\hat{Y}_t,\mathfrak{g}_P)$ and for all $v_1,\dots,v_k \in T_y\hat{Y}_t$ and $y \in \hat{Y}_t$. If $A \in \mathcal{A}(P)$ satisfies $\tilde{\lambda}(h)^*\theta_A = \theta_A$ for every $h\in H$ (where $\theta_A$ is the connection 1-form associated to $A$), then $\Upsilon_A$ is an $H$-equivariant map.
\end{lemma}
\begin{remark}
The assumption that $\hat{\lambda}(h)$ preserves both $\tilde{g}_t$ and $\psi_t$ (and therefore also $g_t$ by \autoref{rem: stabiliser 4-form}) was made so that the previous lemma holds. All group actions which come via the equivariant generalised Kummer construction (as explained in \autoref{rem:EquivariantGeneralisedKummer}) satisfy this assumption. Note, however, that if one takes the adjoints in the definition of $\Upsilon_A$ with respect to $g_t$ (cf. \autoref{rem:definition_UpsilonA_different_adjoints}), then it suffices that $H$ only fixes $\psi_t$.
\end{remark}
\begin{proof}
For any (not necessarily $H$-invariant) connection $B\in \mathcal{A}(P)$ and any $\eta \in \Omega^k(\hat{Y}_t,\mathfrak{g}_P)$ we have $\tilde{\lambda}(h)^* (\diff_B \eta) = \diff_{\tilde{\lambda}(h)^*B} \tilde{\lambda}(h)^*\eta$. Since $\lambda(h)$ preserves $\tilde{g}_t$, the same holds true with respect to $\diff_B^*$ and, similarly, \[ \tilde{\lambda}(h)^*(F_{B} \wedge \psi_t) = F_{\tilde{\lambda}(h)^*B} \wedge \psi_t. \qedhere \]
\end{proof}

\subsection{Perturbing families of invariant almost-instantons}\label{subsec: perturbing families of connections}

In the following we will prove in \autoref{theo:perturbing_almost_instantons} the main result of this section. This gives sufficient conditions under which a family of $H$-invariant connections may be perturbed to a nearby family of ($H$-invariant) $\G_2$-instantons. Furthermore, we prove that the perturbed family of instantons has the same regularity (seen as a map into $\mathcal{A}(P)$) as the initial family of connections and give in \autoref{prop: derivative estimates on deformed family of instantons} estimates on its derivatives (which are needed later in the proof of \autoref{prop:example18-non-gauge-equivalent-instantons}). The proof of these results requires the following extension of Banach's fixed-point Theorem to a family of contractions, which is well known.

\begin{lemma}\label{lem: fixed-points of families}
Let $(B,\Vert \cdot \Vert)$ be a Banach space and $U \coloneqq \overline{B_R(0)}\subset B$ be a closed ball of radius $R\in (0,\infty]$. Assume further that $\mathfrak{F}$ is a manifold (possibly non-compact and/or with boundary), and $E \col \mathfrak{F} \times U \to U$ is a family of contractions (i.e. $E_{\mathfrak{f}} \coloneqq E(\mathfrak{f},\cdot)$ is a contraction for every $\mathfrak{f} \in \mathfrak{F}$) that satisfies $\Vert E_{\mathfrak{f}}(b_1) - E_{\mathfrak{f}}(b_2) \Vert \leq C \Vert b_1 - b_2 \Vert$ for every $\mathfrak{f}\in \mathfrak{F}$ and $b_1,b_2\in U$ for an $\mathfrak{f}$-independent constant $C<1$. For any $\mathfrak{f}\in \mathfrak{F}$ we denote by $a_{\mathfrak{f}}\in U$ the unique fixed-point of $E_{\mathfrak{f}}$ and by $\textup{Fix}\col \mathfrak{F} \to U$ the map $\mathfrak{f} \mapsto a_{\mathfrak{f}}$. The following hold:
\begin{enumerate}
\item If $E$ is continuous (as a map $\mathfrak{F} \times U \to U$), then so is $\textup{Fix}$.
\item If $E$ is $C^k$ (as a map $\mathfrak{F} \times U \to U$), then so is $\textup{Fix}$. Furthermore, its derivative $D\textup{Fix} \col T \mathfrak{F} \to U\times B$ satisfies for every $\mathfrak{f}\in \mathfrak{F}$ and $v\in T_\mathfrak{f}\mathfrak{F}$ 
\begin{align*}
\Vert (D \textup{Fix})_\mathfrak{f}(v) \Vert &\leq (1- \Vert (\partial_BE)_{(\mathfrak{f},a_\mathfrak{f})} \Vert_{\textup{op}})^{-1} \Vert (\partial_{\mathfrak{F}} E)_{(\mathfrak{f},a_\mathfrak{f})}(v) \Vert \\
& \leq  (1- c_\mathfrak{f})^{-1} \Vert (\partial_{\mathfrak{F}} E)_{(\mathfrak{f},a_\mathfrak{f})}(v) \Vert.
\end{align*} 
Here $\partial_{\mathfrak{F}}E$ and $\partial_BE$ denote the respective partial derivatives of $E$ in $\mathfrak{F}$ and $B$ direction and $c_\mathfrak{f}<C$ denotes the optimal constant for which $ \Vert E_\mathfrak{f}(b_1) -E_\mathfrak{f}(b_2) \Vert \leq c_\mathfrak{f} \Vert b_1 - b_2 \Vert$ for every $b_1,b_2 \in U$.
\end{enumerate}
\end{lemma}

\begin{proof}
In order to prove the first point, let $\mathfrak{f}_1,\mathfrak{f}_2 \in \mathfrak{F}$. Then 
\begin{align*}
\Vert a_{\mathfrak{f}_1} -a_{\mathfrak{f}_2} \Vert & \leq \Vert E_{\mathfrak{f}_1} (a_{\mathfrak{f}_1}) -E_{\mathfrak{f}_2}(a_{\mathfrak{f}_1}) \Vert + \Vert E_{\mathfrak{f}_2}(a_{\mathfrak{f}_1}) -E_{\mathfrak{f}_2}(a_{\mathfrak{f}_2}) \Vert \\
& \leq \Vert E_{\mathfrak{f}_1}(a_{\mathfrak{f}_1}) -E_{\mathfrak{f}_2}(a_{\mathfrak{f}_1}) \Vert +C \Vert a_{\mathfrak{f}_1} -a_{\mathfrak{f}_2} \Vert 
\end{align*}
where $C<1$ is $\mathfrak{f}_2$-independent by the assumption. Absorbing the second term on the right-hand side gives 
\begin{equation} \label{eq: continuity family contractions} 
\Vert a_{\mathfrak{f}_1} -a_{\mathfrak{f}_2} \Vert \leq C^\prime \Vert E_{\mathfrak{f}_1}(a_{\mathfrak{f}_1}) -E_{\mathfrak{f}_2}(a_{\mathfrak{f}_1}) \Vert.
\end{equation} 
When $E$ is continuous, then $\mathfrak{f}_2 \to \mathfrak{f}_1$ implies therefore $a_{\mathfrak{f}_2} \to a_{\mathfrak{f}_1}$ which proves the first assertion. 

For the second point, we work locally and assume that $\mathfrak{F}$ is an open subset of $\mathbb{R}^n$ or $\mathbb{R}^n_{x_1\geq 0}$ and prove differentiability of $\textup{Fix}$ at $\mathfrak{f}_0=0$. By the differentiability of $E$ we have
\begin{align*}
a_\mathfrak{f} - a_0 &= E_\mathfrak{f}(a_\mathfrak{f}) - E_0(a_0) = (DE)_{(0,a_0)}(\mathfrak{f},a_{\mathfrak{f}}-a_0) + o (\vert \mathfrak{f} \vert +\Vert a_{\mathfrak{f}}-a_0 \Vert) \\
&= (\partial_B E)_{(0,a_0)}(a_\mathfrak{f} - a_0) + (\partial_{\mathfrak{F}} E)_{(0,a_0)} \cdot \mathfrak{f} + o (\vert \mathfrak{f} \vert+\Vert a_\mathfrak{f} -a_0 \Vert)
\end{align*}
where $DE$ denotes the (total) differential of $E$ and $\partial_{\mathfrak{F}}E$ and $\partial_BE$ the partial derivatives in $\mathfrak{F}$ and $B$ direction. Since $\Vert E_0(b) - E_0(a_0) \Vert \leq c_0 \Vert b - a_0 \Vert$ for every $b \in U$ with $c_0<1$, we obtain $\Vert (\partial_BE)_{(0,a_0)} \Vert_{\textup{op}} = \Vert (DE_0)_{a_0} \Vert_{\textup{op}} \leq c_0 <1$. The operator $(\textup{Id} - (\partial_BE)_{(0,a_0)})$ is therefore invertible via the Neumann series of $(\partial_BE)_{(0,a_0)}$. This gives \[ a_\mathfrak{f} - a_0 = (\textup{Id} - (\partial_BE)_{(0,a_0)} )^{-1} \big((\partial_{\mathfrak{F}} E)_{(0,a_0)} \cdot \mathfrak{f} + o (\vert \mathfrak{f} \vert+\Vert a_\mathfrak{f} -a_0 \Vert)\big). \] The differentiability of $E$ together with~\eqref{eq: continuity family contractions} imply that $o(\vert \mathfrak{f}\vert + \Vert a_\mathfrak{f} -a_0 \Vert)=o(\vert \mathfrak{f} \vert)$. The previous formula for $a_\mathfrak{f}-a_0$ shows therefore that $\textup{Fix}$ is differentiable at $\mathfrak{f}=0$ with derivative \[ (D\textup{Fix})_\mathfrak{f} = (\textup{Id} - (\partial_BE)_{(\mathfrak{f},a_\mathfrak{f})} )^{-1} \big( (\partial_{\mathfrak{F}} E)_{(\mathfrak{f},a_\mathfrak{f})}(\ \cdot\ )\big). \] Since inversion and concatenation are smooth operations, $E \in C^k$ implies via bootstrapping that $\textup{Fix}\in C^k$.
\end{proof}

In the following let $(\hat{Y}_t,\phi_t)$ for $t\in(0,T_K)$ again be a fixed 1-parameter family of $\G_2$-manifolds obtained from Joyce's generalised Kummer construction together with an $H$-action $\lambda \col H \to \textup{Isom}(\hat{Y}_t,\tilde{g}_t)$ by a (possibly trivial) group $H$ preserving $\psi_t$.

\begin{theorem}\label{theo:perturbing_almost_instantons}
Let $\alpha \in (0,1)$ be a fixed Hölder exponent and $\gamma,C>0$ and $\varepsilon\geq 0$ be constants with $\varepsilon<\gamma$. Then there exist constants $0<T_I<T_K$ and $c_I>0$ depending only $\alpha,\gamma,C,\varepsilon$ with the following significance: Let $t\in (0,T_I)$ and $\pi \col P \to \hat{Y}_t$ be a principal $G$-bundle together with a lift of the $H$-action. Assume further that $\mathfrak{F} \subset \mathbb{R}$ is a connected subset (i.e. a point or an interval -- possibly non-compact and/or with boundary) and that $\mathfrak{F}\ni \mathfrak{f}\mapsto \tilde{A}_{\mathfrak{f}}\in \mathcal{A}(P)$ is a family of $H$-invariant connections that satisfies 
\begin{enumerate}
\item \label{bul:pregluing-error} $\Vert F_{\tilde{A}_{\mathfrak{f}}} \wedge \psi_t \VertWHt{0}{-2-\varepsilon}{} \leq C t^{\gamma}$ for every $\mathfrak{f} \in \mathfrak{F}$
\item \label{bul:linearEstimate} $\Vert \underline{b} \VertWHt{2}{-\varepsilon}{} \leq C \Vert L_{\tilde{A}_{\mathfrak{f}}}L_{\tilde{A}_{\mathfrak{f}}}^* \underline{b} \VertWHt{0}{-2-\varepsilon}{}$ for every $H$-invariant $\underline{b}\in \Omega^6(\hat{Y}_t,\mathfrak{g}_P) \oplus \Omega^0(\hat{Y}_t,\mathfrak{g}_P)$ and every $\mathfrak{f}\in \mathfrak{F}$
\item \label{bul: operator norm estimate on L_A^*} $\Vert L_{\tilde{A}_{\mathfrak{f}}}^* \underline{b} \VertWHt{1}{-1-\varepsilon}{} \leq C \Vert  \underline{b} \VertWHt{2}{-\varepsilon}{}$  for every $H$-invariant $\underline{b}\in \Omega^6(\hat{Y}_t,\mathfrak{g}_P) \oplus \Omega^0(\hat{Y}_t,\mathfrak{g}_P)$ and every $\mathfrak{f}\in \mathfrak{F}$.
\end{enumerate}
All norms in the expressions above are taken with respect to the metric induced by $\tilde{g}_t$ and the $\Ad$-invariant inner product on $\mathfrak{g}$. Similarly, covariant derivatives are taken with respect to the Levi--Civita connection and a fixed $B\in \mathcal{A}(P)$. Under these assumptions there exists an $H$-invariant family of $\G_2$-instantons $\mathfrak{F}\ni \mathfrak{f} \mapsto A_{\mathfrak{f}}\in \mathcal{A}(P)$ with $\Vert \tilde{A}_{\mathfrak{f}}-A_{\mathfrak{f}} \VertWHt{1}{-1-\varepsilon}{} \leq c_I t^{\gamma}$ for every $\mathfrak{f} \in \mathfrak{F}$. Furthermore, if the map $\mathfrak{f} \mapsto \tilde{A}_{\mathfrak{f}}$ lies in $C^{L}(\mathfrak{F}, \mathcal{A}(P))$ for $L \in \mathbb{N}_0 \cup \{\infty\}$, then $\mathfrak{f} \mapsto A_{\mathfrak{f}}$ also lies in $C^{L}(\mathfrak{F}, \mathcal{A}(P))$.
\end{theorem}
\begin{remark}
Throughout this article we consider $\mathcal{A}(P)$ as an affine Fréchet space and a map $\textup{Fam} \col \mathfrak{F} \to \mathcal{A}(P)$ (for an interval $\mathfrak{F}\subset \mathbb{R}$) to be of regularity $C^1$ if the map 
\begin{align*}
D \textup{Fam} \col \mathfrak{F} \times \mathbb{R} &\to \Omega^1(\hat{Y}_t,\mathfrak{g}_P) \\
(\mathfrak{f},v) &\mapsto (D\textup{Fam})_\mathfrak{f}(v) \coloneqq \lim_{s \to 0} \tfrac{\textup{Fam}(\mathfrak{f}+s v)-\textup{Fam}(\mathfrak{f})}{s}
\end{align*}
exists and is (jointly) continuous. Higher regularity is defined in a similar manner.
\end{remark}
\begin{remark}\label{rem:perturbing almost instantons weakening assumptions}
Regarding weakening the assumptions of the previous theorem:
\begin{enumerate}
\item\label{bul: perturbing for general families} Denote by $\textup{Fam} \col \mathfrak{F} \to \mathcal{A}(P)$ the family of connections of the previous theorem, i.e. $\textup{Fam}(\mathfrak{f})\coloneqq \tilde{\mathcal{A}}_\mathfrak{f}$. We have made the assumption that $\mathfrak{F}$ is an interval because this leads to the somewhat simpler notation $(\partial_\mathfrak{F}\textup{Fam})_\mathfrak{f} \coloneqq (D\textup{Fam})_{\mathfrak{f}}(1)$ for the derivative of $\textup{Fam}$. Since the connections in \autoref{sec:Example18} to which we apply \autoref{theo:perturbing_almost_instantons} are parametrised by an interval, this assumption does not restrict us in this article. Note however, that one can easily generalise the theorems of this section to apply to any manifold $\mathfrak{F}$  (possibly non-compact and/or with boundary).
\item\label{bul: perturbing for more general H-actions}  Similarly, one can weaken the condition that the $H$-action lifts to $P$. Indeed, it suffices that the action lifts to the adjoint bundle $\mathfrak{g}_P$ and $\Upsilon_A$ is equivariant with respect to this lift.
\end{enumerate}
\end{remark}
\begin{remark}
If $H$ and $\mathfrak{F}$ are both trivial and $\varepsilon=0$, then the above theorem becomes analogous to the construction used in \cite{Walpuski-InstantonsKummer} to produce $\G_2$-instantons over generalised Kummer constructions resolving orbifolds with codimension $4$ singularities. Note, however, that for non-trivial $H$-actions our second assumption (listed under \autoref{bul:linearEstimate}) differs from the analogous condition in~\cite[Section~8]{Walpuski-InstantonsKummer}. This is because the $H$-invariant kernel and cokernel of $L_A$ are in general not isomorphic (as can be see in the examples of \autoref{sec:Example18}). Note also that establishing \autoref{bul:linearEstimate} of \autoref{theo:perturbing_almost_instantons} (which is done in \autoref{prop:LinearEstimate}) becomes simpler when $\varepsilon>0$ and we will therefore make use of this freedom. 
\end{remark}
\begin{remark}
If $\mathfrak{F}$ is an interval, then each $\G_2$-instanton $A_\mathfrak{f}$ in the family $(A_\mathfrak{f})_{\mathfrak{f}\in \mathfrak{F}}$ constructed in the previous theorem is necessarily obstructed. Note, however, that under mild assumptions (e.g. the analogue of \autoref{bul:infinitesimally_irreducible} in \autoref{prop:infinitesimal_irreducibility} for the operator $L_{\tilde{A}_{\mathfrak{f}}}$) the instanton $A_\mathfrak{f}$ becomes $H$-unobstructed (i.e. there are no $H$-invariant elements in the cokernel of $L_{A_\mathfrak{f}}$) once $t$ is sufficiently small. This can be proven as \autoref{prop:infinitesimal_irreducibility}.
\end{remark}
\begin{proof}[Proof of \autoref{theo:perturbing_almost_instantons}]
For $k\in \mathbb{N}_0$ define the closed (and therefore complete) subspaces 
\begin{align*}
V^{k,\alpha}_{\beta,t} &\coloneqq \left\{ \underline{a} \in C^{k,\alpha}_{\beta,t}(T^*\hat{Y}_t\otimes\mathfrak{g}_P\oplus \Lambda^7T^*\hat{Y}_t \otimes \mathfrak{g}_P) \mid \textup{$\underline{a}$ is $H$-invariant}\right\} \\
W^{k,\alpha}_{\beta,t} &\coloneqq \left\{ \underline{b} \in C^{k,\alpha}_{\beta,t}(\Lambda^6T^*\hat{Y}_t\otimes \mathfrak{g}_P\oplus \mathfrak{g}_P) \mid \textup{$\underline{b}$ is $H$-invariant}\right\}. 
\end{align*} 
Since $L_{\tilde{A}_{\mathfrak{f}}}$ and $L_{\tilde{A}_{\mathfrak{f}}}^*$ are $H$-equivariant, the second condition implies that $L_{\tilde{A}_{\mathfrak{f}}}L_{\tilde{A}_{\mathfrak{f}}}^* \col W^{2,\alpha}_{-\varepsilon,t} \to W^{0,\alpha}_{-\varepsilon-2,t}$ is injective. Formal self-adjointness, $H$-equivariance, and elliptic regularity imply that $L_{\tilde{A}_{\mathfrak{f}}}L_{\tilde{A}_{\mathfrak{f}}}^* \col W^{2,\alpha}_{-\varepsilon,t} \to W^{0,\alpha}_{-\varepsilon-2,t}$ is also surjective and we can therefore define \[R_{\tilde{A}_{\mathfrak{f}}} \coloneqq L_{\tilde{A}_{\mathfrak{f}}}^*(L_{\tilde{A}_{\mathfrak{f}}}L_{\tilde{A}_{\mathfrak{f}}}^*)^{-1} \col W^{0,\alpha}_{-2-\varepsilon,t} \to V^{1,\alpha}_{-1-\varepsilon,t}. \] By construction, this is a right-inverse to $L_{\tilde{A}_{\mathfrak{f}}}$ and \autoref{bul:linearEstimate} and \autoref{bul: operator norm estimate on L_A^*} in the above statement and \autoref{theo:torsionfree_G2_structure} imply
\begin{align}\label{eq: estimate right-inverse}
\Vert R_{\tilde{A}_{\mathfrak{f}}} \underline{b} \VertWHt{1}{-1-\varepsilon}{} & \leq c \Vert \underline{b}\VertWHt{0}{-2-\varepsilon}{}
\end{align}
for a $t$ and $\mathfrak{f}$-independent constant $c>0$ (which we will assume to be larger than 2 in the following).

Since $F_{\tilde{A}_{\mathfrak{f}}}\wedge\psi_t \in W^{0,\alpha}_{-2-\varepsilon,t}$ and $Q_{\tilde{A}_{\mathfrak{f}}}\col V^{1,\alpha}_{-1-\varepsilon,t} \to  W^{0,\alpha}_{-2-\varepsilon,t}$, we can define the map 
\begin{align*}
E_{\mathfrak{f}} \col W^{0,\alpha}_{-2-\varepsilon,t} &\to W^{0,\alpha}_{-2-\varepsilon,t} \\
\underline{b} & \mapsto E_{\mathfrak{f}}(\underline{b}) \coloneqq - \big(Q_{\tilde{A}_{\mathfrak{f}}}(R_{\tilde{A}_{\mathfrak{f}}}\underline{b})+F_{\tilde{A}_{\mathfrak{f}}}\wedge\psi_t \big).
\end{align*}
Our assumptions, \autoref{prop:DecompositionNonlinearMap}, its following remark, and~\eqref{eq:weighted-t-Höldernorms-estimates} imply 
\begin{align*}
\Vert E_{\mathfrak{f}}(0) \VertWHt{0}{-2-\varepsilon}{} &\leq c t^{\gamma} \\
\Vert E_{\mathfrak{f}}(\underline{b}_1)-E_{\mathfrak{f}}(\underline{b}_2) \VertWHt{0}{-2-\varepsilon}{} &\leq  c \big( \Vert R_{\tilde{A}_{\mathfrak{f}}}\underline{b}_1 \VertWHt{0}{-1-\varepsilon/2}{} + \Vert R_{\tilde{A}_{\mathfrak{f}}}\underline{b}_2 \VertWHt{0}{-1-\varepsilon/2}{}\big) \Vert R_{\tilde{A}_{\mathfrak{f}}}(\underline{b}_1 - \underline{b}_2) \VertWHt{0}{-1-\varepsilon/2}{} \\
& \leq c t^{-\varepsilon} \big( \Vert R_{\tilde{A}_{\mathfrak{f}}}\underline{b}_1 \VertWHt{0}{-1-\varepsilon}{} + \Vert R_{\tilde{A}_{\mathfrak{f}}}\underline{b}_2 \VertWHt{0}{-1-\varepsilon}{}\big) \Vert R_{\tilde{A}_{\mathfrak{f}}}(\underline{b}_1 - \underline{b}_2) \VertWHt{0}{-1-\varepsilon}{}\\
& \leq ct^{-\varepsilon} \big( \Vert \underline{b}_1 \VertWHt{0}{-2-\varepsilon}{} + \Vert \underline{b}_2 \VertWHt{0}{-2-\varepsilon}{}\big) \Vert \underline{b}_1 - \underline{b}_2 \VertWHt{0}{-2-\varepsilon}{}
\end{align*}
where $c=c(\alpha,C)$ is independent of $t$ and of $\mathfrak{f}$.
Therefore, the map $E_{\mathfrak{f}}$ restricts to a contraction on $\overline{B_R(0)}\subset W^{0,\alpha}_{-2-\varepsilon,t}$ if \[ 2 ct^{-\varepsilon} R < 1 \quad \textup{and} \quad ct^\gamma + c t^{-\varepsilon} R^2<R.\] This holds, for example, whenever $t^\gamma<(2 c)^{-2}t^{\varepsilon}$ (which can be achieved as $0<\varepsilon< \gamma$) and $R\coloneqq 2 ct^\gamma$. Let $\underline{b}_{\mathfrak{f}} \in \overline{B_R(0)}$ be the unique fixpoint of $E_{\mathfrak{f}}$ and $\underline{a}_{\mathfrak{f}}\coloneqq R_{\tilde{A}_{\mathfrak{f}}}\underline{b}_{\mathfrak{f}}$. By construction, this solves \begin{align*}
\Upsilon_{A_{\mathfrak{f}}}(\underline{a}_{\mathfrak{f}}) & = L_{\tilde{A}_{\mathfrak{f}}}\underline{a}_{\mathfrak{f}} + Q_{\tilde{A}_{\mathfrak{f}}}(\underline{a}_{\mathfrak{f}})+F_{\tilde{A}_{\mathfrak{f}}}\wedge \psi_t = \underline{b}_{\mathfrak{f}} + Q_{\tilde{A}_{\mathfrak{f}}}(R_{\tilde{A}_{\mathfrak{f}}}\underline{b}_{\mathfrak{f}})+F_{\tilde{A}_{\mathfrak{f}}}\wedge \psi_t =  0
\end{align*} 
and is therefore smooth by elliptic regularity. Furthermore, \autoref{prop:AugmentedInstanton} implies that $A_{\mathfrak{f}} \coloneqq \tilde{A}_{\mathfrak{f}}+a_{\mathfrak{f}} \in \mathcal{A}(P)$ is a $\G_2$-instanton.

If $\mathfrak{f} \mapsto \tilde{A}_{\mathfrak{f}}$ lies in $C^L(\mathfrak{F},\mathcal{A}(P))$, where we assume from now on that $\mathfrak{F}$ is an interval, then $\mathfrak{f} \mapsto L_{\tilde{A}_{\mathfrak{f}}}$ lies in $C^L(\mathfrak{F},\textup{Lin}(C^{k,\alpha}_{\beta,t},C^{k-1,\alpha}_{\beta-1,t}))$ for all $k \in \mathbb{N}_0$ and $\beta \in \mathbb{R}$ (where $\textup{Lin}(C^{k,\alpha}_{\beta,t},C^{k-1,\alpha}_{\beta-1,t})$ denotes the Banach space of bounded linear maps $C^{k,\alpha}_{\beta,t} \to C^{k-1,\alpha}_{\beta-1,t}$). Since concatenation and inversion are smooth operations, this implies that \[ \mathfrak{F} \times W^{0,\alpha}_{-2-\varepsilon,t} \ni (\mathfrak{f},\underline{b}) \mapsto E_{\mathfrak{f}}(\underline{b})\in W^{0,\alpha}_{-2-\varepsilon,t} \] is a $C^L$-map. \autoref{lem: fixed-points of families} implies therefore that $\mathfrak{f} \mapsto a_{\mathfrak{f}}$ lies in $C^L(\mathfrak{F},V^{1,\alpha}_{-1-\varepsilon,t})$ (where $V^{1,\alpha}_{-1-\varepsilon,t}$ is equipped with its Banach space topology). 

Next, we will prove that $\mathfrak{f} \mapsto a_{\mathfrak{f}}$ is indeed of regularity $C^L$ when regarded as a map into the space $C^\infty(T^*\hat{Y}_t \otimes \mathfrak{g}_P \oplus \Lambda^7 T^*\hat{Y}_t \otimes \mathfrak{g}_P)$ equipped with its Fréchet topology. For this we will denote in the following by $(\partial_\mathfrak{F}\underline{a})_\mathfrak{f} \in V^{1,\alpha}_{-1-\varepsilon,t}$ and $(\partial_\mathfrak{F}L_{\tilde{A}})_\mathfrak{f} \in \textup{Lin}(V^{1,\alpha}_{\beta,t},W^{0,\alpha}_{\beta-1,t})$ the respective derivatives of $\mathfrak{f}\mapsto \underline{a}_\mathfrak{f}$ and $\mathfrak{f} \mapsto L_{\tilde{A}_\mathfrak{f}}$ taken with respect to the indicated Banach space topologies. Similarly, we will denote by $(\partial_\mathfrak{F}F_{\tilde{A}})_\mathfrak{f} \in \Omega^2(\hat{Y}_t,\mathfrak{g}_P)$ the derivative of $\mathfrak{f} \mapsto F_{\tilde{A}_\mathfrak{f}}$ taken with respect to the Fréchet topology. 

Taking the derivative of $\mathfrak{f}\mapsto \Upsilon_{\tilde{A}_\mathfrak{f}}(\underline{a}_\mathfrak{f})$ (as a map into the Banach space $W^{0,\alpha}_{-2-\varepsilon,t}$) at $\mathfrak{f}_0 \in \mathfrak{F}$ (and noting that $Q_{\tilde{A}_{\mathfrak{f}}} \equiv Q$ is independent of the connection) yields \[ L_{\tilde{A}_{\mathfrak{f}_0}} (\partial_\mathfrak{F}\underline{a})_{\mathfrak{f}_0} + (\partial_\mathfrak{F}L_{\tilde{A}})_{\mathfrak{f}_0} \underline{a}_{\mathfrak{f}_0} + (D_{\underline{a}_{\mathfrak{f}_0}}Q)((\partial_\mathfrak{F}\underline{a})_{\mathfrak{f}_0}) + (\partial_\mathfrak{F}F_{\tilde{A}})_{\mathfrak{f}_0} \wedge \psi_t = 0.\] Elliptic regularity of the operator $L_{\tilde{A}_{\mathfrak{f}_0}}$ gives via bootstrapping $(\partial_\mathfrak{F}\underline{a})_{\mathfrak{f}_0} \in C^{\infty}(T^*\hat{Y}_t\otimes \mathfrak{g}_P \oplus \Lambda T^*\hat{Y}_t \otimes \mathfrak{g}_P)$. 

Next, we prove that $(\partial_\mathfrak{F}\underline{a})_{\mathfrak{f}_0}$ is also the derivative of $\underline{a}$ at ${\mathfrak{f}_0}$ when regarded as a map $\mathfrak{F} \to C^\infty(T^*\hat{Y}_t \otimes \mathfrak{g}_P \oplus \Lambda^7 T^*\hat{Y}_t \otimes \mathfrak{g}_P)$. For this we assume that we have already proven that $(\partial_\mathfrak{F}\underline{a})_{\mathfrak{f}_0}$ is the derivative of $\underline{a}$ at ${\mathfrak{f}_0}$ seen as a map $\mathfrak{F} \to C^{k-1,\alpha}(T^*\hat{Y}_t \otimes \mathfrak{g}_P \oplus \Lambda^7 T^*\hat{Y}_t \otimes \mathfrak{g}_P)$ for some fixed $k \in \mathbb{N}$ (the case $k=2$ holds by construction and \autoref{lem: fixed-points of families}) and show the same statement holds when regarded as a map $\mathfrak{F} \to C^{k,\alpha}(T^*\hat{Y}_t \otimes \mathfrak{g}_P \oplus \Lambda^7 T^*\hat{Y}_t \otimes \mathfrak{g}_P)$. Since sequences converge in the $C^\infty$-topology if and only if they converge in all $C^{k,\alpha}$-norms, this implies the claim. 

The Schauder estimate\footnote{The constant $c$ in the following estimates will in general depend on $t$. This does not affect the argument.} for $L_{\tilde{A}_{\mathfrak{f}_0}}$ gives for any $\mathfrak{f}\in \mathbb{R}$ with $\vert \mathfrak{f} \vert $ sufficiently small
\begin{align*}
\Vert \underline{a}_{\mathfrak{f}_0+\mathfrak{f}} - \underline{a}_{\mathfrak{f}_0} - (\partial_\mathfrak{F}\underline{a})_{\mathfrak{f}_0} \cdot \mathfrak{f} \VertH{k}{} &\leq c \big( \Vert L_{\tilde{A}_{\mathfrak{f}_0}} (\underline{a}_{\mathfrak{f}_0+\mathfrak{f}} - \underline{a}_{\mathfrak{f}_0} - (\partial_\mathfrak{F}\underline{a})_{\mathfrak{f}_0} \cdot \mathfrak{f} ) \VertH{k-1}{} \\
& \qquad + \Vert \underline{a}_{\mathfrak{f}_0+\mathfrak{f}} - \underline{a}_{\mathfrak{f}_0} - (\partial_\mathfrak{F}\underline{a})_{\mathfrak{f}_0} \cdot \mathfrak{f}  \VertC{0}{} \big).
\end{align*}
The second term on the right-hand side is by assumption $o(\vert \mathfrak{f} \vert)$. In order to bound the first term we rewrite \[ L_{\tilde{A}_{\mathfrak{f}_0}} \underline{a}_{\mathfrak{f}_0+\mathfrak{f}} = \big(L_{\tilde{A}_{\mathfrak{f}_0+\mathfrak{f}}} - (\partial_\mathfrak{F}L_{\tilde{A}})_{\mathfrak{f}_0} \cdot \mathfrak{f} - \mathfrak{R}_{\mathfrak{f}_0}(\mathfrak{f})\big) \underline{a}_{\mathfrak{f}_0+\mathfrak{f}} \] where by the differentiability of $\mathfrak{f}\mapsto \tilde{A}_{\mathfrak{f}}$ and $\mathfrak{f}\mapsto \underline{a}_\mathfrak{f}$ we have $\Vert \mathfrak{R}_{\mathfrak{f}_0}(\mathfrak{f}) \underline{a}_{\mathfrak{f}_0+\mathfrak{f}} \VertH{k-1}{} = o(\vert \mathfrak{f} \vert)$. Inserting the equations for $L_{\tilde{A}_{\mathfrak{f}_0+\mathfrak{f}}}\underline{a}_{\mathfrak{f}_0+\mathfrak{f}}$, $L_{\tilde{A}_{\mathfrak{f}_0}}\underline{a}_{\mathfrak{f}_0}$, and $L_{\tilde{A}_{\mathfrak{f}_0}} (\partial_\mathfrak{F}\underline{a})_{\mathfrak{f}_0}$, noting that $(\partial_\mathfrak{F}L_{\tilde{A}})_{\mathfrak{f}_0}$ is an algebraic operator, and using the assumption that $(\partial_{\mathfrak{F}}\underline{a})_{\mathfrak{f}_0}$ is the derivative of the map $\mathfrak{f}\mapsto \underline{a}_\mathfrak{f}\in C^{k-1,\alpha}$, one can prove that the first term on the right-hand side is also $o(\vert \mathfrak{f} \vert)$. This shows that $(\partial_\mathfrak{F}\underline{a})_{\mathfrak{f}_0}$ is also the derivative of the map $\mathfrak{F} \to C^{k,\alpha}(T^*\hat{Y}_t \otimes \mathfrak{g}_P \oplus \Lambda^7 T^*\hat{Y}_t \otimes \mathfrak{g}_P)$ and by iteration also the derivative of $\mathfrak{f} \mapsto \underline{a}_\mathfrak{f}\in C^{\infty}(T^*\hat{Y}_t \otimes \mathfrak{g}_P \oplus \Lambda^7 T^*\hat{Y}_t \otimes \mathfrak{g}_P)$.

Finally, the continuity of the derivative $\mathfrak{F} \ni \mathfrak{f} \mapsto (\partial_\mathfrak{F}\underline{a})_{\mathfrak{f}} \in C^{\infty}(T^*\hat{Y}_t \otimes \mathfrak{g}_P \oplus \Lambda^7 T^*\hat{Y}_t \otimes \mathfrak{g}_P) $ can be proven analogously which implies that $\mathfrak{f} \mapsto A_{\mathfrak{f}}$ lies in $C^1(\mathfrak{F},\mathcal{A}(P))$.

For the higher derivatives, inductively take the $\ell$-th derivative of $\Upsilon_{\tilde{A}_{\mathfrak{f}}}(\underline{a}_{\mathfrak{f}})$ for $\ell \in \{1,\dots,L\}$ and proceed in the same manner.
\end{proof}

\begin{proposition}\label{prop: derivative estimates on deformed family of instantons}
Assume the situation of \autoref{theo:perturbing_almost_instantons} in which the family $\mathfrak{f}\mapsto \tilde{A}_{\mathfrak{f}}$ lies in $C^L(\mathfrak{F},\mathcal{A}(P))$ for some fixed $L \in \mathbb{N}$ and assume that we additionally have 
\begin{enumerate}\setcounter{enumi}{3}
\item \label{bul:pregluing-error-of-derivatives} $\Vert (\partial_\mathfrak{F}^\ell F_{\tilde{A}})_{\mathfrak{f}} \wedge \psi_t \VertWHt{0}{-2-\varepsilon}{} \leq C t^{\gamma}$ for every $\ell\in \{0,\dots,L\}$ and every $\mathfrak{f}\in \mathfrak{F}$,
\item \label{bul: operator norm estimate on derivatives of L_A} $\Vert (\partial^\ell_{\mathfrak{F}}L_{\tilde{A}})_{\mathfrak{f}} \underline{a} \VertWHt{0}{-2-\varepsilon}{} \leq C \Vert  \underline{a} \VertWHt{1}{-1-\varepsilon}{}$  for every $H$-invariant $\underline{a}\in \Omega^1(\hat{Y}_t,\mathfrak{g}_P) \oplus \Omega^7(\hat{Y}_t,\mathfrak{g}_P)$, every $\ell\in \{0,\dots,L\}$, and every $\mathfrak{f}\in \mathfrak{F}$,
\item \label{bul: operator norm estimate on derivatives of L_A^*} $\Vert (\partial^\ell_{\mathfrak{F}}L_{\tilde{A}}^*)_{\mathfrak{f}} \underline{b} \VertWHt{1}{-1-\varepsilon}{} \leq C \Vert  \underline{b} \VertWHt{2}{-\varepsilon}{}$  for every $H$-invariant $\underline{b}\in \Omega^6(\hat{Y}_t,\mathfrak{g}_P) \oplus \Omega^0(\hat{Y}_t,\mathfrak{g}_P)$, every $\ell\in \{0,\dots,L\}$, and every $\mathfrak{f}\in \mathfrak{F}$.
\end{enumerate}
Here and below $(\partial_\mathfrak{F}L_{\tilde{A}})_{\mathfrak{f}}$ and $(\partial_{\mathfrak{F}}F_{\tilde{A}})_{\mathfrak{f}}$ denote the derivatives of $\mathfrak{f}\mapsto L_{\tilde{A}_\mathfrak{f}}\in \textup{Lin}(C^{1,\alpha}_{-1-\varepsilon,t},C^{0,\alpha}_{-2-\varepsilon,t})$ and $\mathfrak{f} \mapsto F_{\tilde{A}_\mathfrak{f}} \in \Omega^2(\hat{Y},\mathfrak{g}_P)$, respectively. Under these assumptions $c_I$ in \autoref{theo:perturbing_almost_instantons} can be enlarged so that additionally \[\Vert (\partial_{\mathfrak{F}}^\ell \tilde{A})_\mathfrak{f} - (\partial_{\mathfrak{F}}^\ell A)_\mathfrak{f} \VertWHt{1}{-1-\varepsilon}{} \leq c_I t^\gamma \] holds for every $\ell\in \{0,\dots,L\}$.
\end{proposition}
\begin{proof}
In order to estimate the difference of the first derivatives of $\mathfrak{f} \mapsto \tilde{A}_{\mathfrak{f}}$ and $\mathfrak{f} \mapsto A_{\mathfrak{f}}$ note that by construction \[\Vert (\partial_{\mathfrak{F}} \tilde{A})_{\mathfrak{f}} - (\partial_{\mathfrak{F}} A)_{\mathfrak{f}} \VertWHt{1}{-1-\varepsilon}{} = \Vert (\partial_\mathfrak{F} \underline{a})_{\mathfrak{f}} \VertWHt{1}{-1-\varepsilon}{} = \Vert (\partial_\mathfrak{F} R_{\tilde{A}})_{\mathfrak{f}} \underline{b}_{\mathfrak{f}} + R_{\tilde{A}_{\mathfrak{f}}} (\partial_\mathfrak{F}\underline{b})_{\mathfrak{f}} \VertWHt{1}{-1-\varepsilon}{} \] where $\underline{b}_{\mathfrak{f}}$ and $R_{\tilde{A}_{\mathfrak{f}}}$ are as in the proof of \autoref{theo:perturbing_almost_instantons}. In the following we will prove that there exists a constant so that 
\begin{equation}\label{equ: auxiliary derivative estimates}
\Vert (\partial_\mathfrak{F}\underline{b})_f \VertWHt{0}{-2-\varepsilon}{} \leq c t^{\gamma} \qquad \textup{and} \qquad \Vert (\partial_\mathfrak{F} R_{\tilde{A}})_f \underline{b}\VertWHt{1}{-1-\varepsilon}{}\leq c \Vert \underline{b}\VertWHt{0}{-2-\varepsilon}{} 
\end{equation} 
for every ${\mathfrak{f}}\in \mathfrak{F}$ and $H$-invariant $\underline{b}\in \Omega^6(\hat{Y}_t,\mathfrak{g}_P)\oplus \Omega^0(\hat{Y}_t,\mathfrak{g}_P)$. Since by \eqref{eq: estimate right-inverse} and \autoref{theo:perturbing_almost_instantons} we additionally have \[\Vert R_{\tilde{A}_{\mathfrak{f}}} \underline{b} \VertWHt{1}{-1-\varepsilon}{}  \leq c \Vert \underline{b}\VertWHt{0}{-2-\varepsilon}{} \qquad \textup{and} \qquad \Vert \underline{b}_{\mathfrak{f}} \VertWHt{0}{-2-\varepsilon}{} \leq c t^\gamma \] for every ${\mathfrak{f}}\in \mathfrak{F}$ and $H$-invariant $\underline{b}\in \Omega^6(\hat{Y}_t,\mathfrak{g}_P)\oplus \Omega^0(\hat{Y}_t,\mathfrak{g}_P)$, this implies the estimate on the first derivative.

Note that the contraction constant of the contraction $E_{\mathfrak{f}}$ in the proof of \autoref{theo:perturbing_almost_instantons} can be bounded $\mathfrak{f}$-independently by $4c^2t^{\gamma-\varepsilon} <1$. \autoref{lem: fixed-points of families} implies therefore for the derivative of $\mathfrak{f} \mapsto \underline{b}_{\mathfrak{f}}$: \[ \Vert (\partial_\mathfrak{F}\underline{b})_{\mathfrak{f}} \VertWHt{0}{-2-\varepsilon}{} \leq c \Vert (\partial_\mathfrak{F}E)_{\mathfrak{f}}(\underline{b}_{\mathfrak{f}}) \VertWHt{0}{-2-\varepsilon}{} = c \Vert (D_{R_{\tilde{A}_{\mathfrak{f}}}\underline{b}_{\mathfrak{f}}}Q)((\partial_\mathfrak{F}R_{\tilde{A}})_{\mathfrak{f}}\underline{b}_{\mathfrak{f}}) - \partial_\mathfrak{F} (F_{\tilde{A}})_{\mathfrak{f}} \wedge \psi_t \VertWHt{0}{-2-\varepsilon}{}. \] 
Furthermore, the derivative of the quadratic map $Q$ takes the form \[(D_{R_{\tilde{A}_{\mathfrak{f}}}\underline{b}_{\mathfrak{f}}}Q)((\partial_\mathfrak{F}R_{\tilde{A}})_{\mathfrak{f}}\underline{b}_{\mathfrak{f}}) = \{R_{\tilde{A}_{\mathfrak{f}}}\underline{b}_{\mathfrak{f}},(\partial_\mathfrak{F}R_{\tilde{A}})_{\mathfrak{f}} \underline{b}_{\mathfrak{f}}\}\] where $\{ \cdot,\cdot\}$ is a bilinear algebraic map. This observation together with \autoref{bul:pregluing-error-of-derivatives} in the assumptions, $\Vert \underline{b}_\mathfrak{f} \VertWHt{0}{-2-\varepsilon}{} \leq ct^\gamma$, and $\Vert R_{\tilde{A}_{\mathfrak{f}}} \underline{b}_{\mathfrak{f}} \VertWHt{1}{-1-\varepsilon}{} \leq c t^\gamma$, implies that the first estimate in~\eqref{equ: auxiliary derivative estimates} follows from the second estimate.

We are therefore left to prove \[\Vert (\partial_\mathfrak{F} R_{\tilde{A}})_{\mathfrak{f}} \underline{b}\VertWHt{1}{-1-\varepsilon}{}\leq c \Vert \underline{b}\VertWHt{0}{-2-\varepsilon}{} \] for every ${\mathfrak{f}}\in \mathfrak{F}$ and $H$-invariant $\underline{b}\in \Omega^6(\hat{Y}_t,\mathfrak{g}_P)\oplus \Omega^0(\hat{Y}_t,\mathfrak{g}_P)$. For this note that taking the derivative of $(L_{\tilde{A}_{\mathfrak{f}}} L_{\tilde{A}_{\mathfrak{f}}}^*)(L_{\tilde{A}_{\mathfrak{f}}}L_{\tilde{A}_{\mathfrak{f}}}^*)^{-1} = \id$ implies that \[ \partial_\mathfrak{F} (L_{\tilde{A}} L_{\tilde{A}}^*)_{\mathfrak{f}}^{-1} = - (L_{\tilde{A}_{\mathfrak{f}}} L_{\tilde{A}_{\mathfrak{f}}}^*)^{-1} (\partial_\mathfrak{F}(L_{\tilde{A}} L_{\tilde{A}}^*)_{\mathfrak{f}})(L_{\tilde{A}_{\mathfrak{f}}} L_{\tilde{A}_{\mathfrak{f}}}^*)^{-1}. \] The assumption in \autoref{bul:linearEstimate} of \autoref{theo:perturbing_almost_instantons} and \autoref{bul: operator norm estimate on derivatives of L_A} and \autoref{bul: operator norm estimate on derivatives of L_A^*} in this proposition yield the $t$-independent bound \[ \Vert \partial_\mathfrak{F} (L_{\tilde{A}} L_{\tilde{A}}^*)_{\mathfrak{f}}^{-1} \underline{b} \VertWHt{2}{-\varepsilon}{} \leq c \Vert \underline{b} \VertWHt{0}{-2-\varepsilon}{} \] for every $H$-invariant $\underline{b}\in \Omega^6(\hat{Y}_t,\mathfrak{g}_P)\oplus \Omega^0(\hat{Y}_t,\mathfrak{g}_P)$. Similarly, we obtain \[ \Vert \partial_\mathfrak{F} (R_{\tilde{A}})_{\mathfrak{f}}\underline{b} \VertWHt{1}{-\varepsilon}{} = \Vert (\partial_\mathfrak{F} L_{\tilde{A}}^*)_{\mathfrak{f}}(L_{\tilde{A}_{\mathfrak{f}}} L_{\tilde{A}_{\mathfrak{f}}}^*)^{-1}\underline{b} + L_{\tilde{A}_{\mathfrak{f}}}^*(\partial_\mathfrak{F}(L_{\tilde{A}} L_{\tilde{A}}^*)_{\mathfrak{f}}^{-1})\underline{b} \VertWHt{1}{-\varepsilon}{} \leq c \Vert \underline{b} \VertWHt{0}{-2-\varepsilon}{} \] for every $H$-invariant $\underline{b}\in \Omega^6(\hat{Y}_t,\mathfrak{g}_P)\oplus \Omega^0(\hat{Y}_t,\mathfrak{g}_P)$ which is exactly the second estimate in \eqref{equ: auxiliary derivative estimates}. This completes the proof for $L=1$. The estimates for the difference of the higher derivatives of $f \mapsto \tilde{A}_{\mathfrak{f}}$ and $f\mapsto A_f$ are proven in a similar way by inductively showing that~\eqref{equ: auxiliary derivative estimates} also holds for higher  $\partial_\mathfrak{F}$-derivatives.
\end{proof}

Denote by $\mathfrak{z} \subset \mathfrak{g}$ the center of $\mathfrak{g}$, i.e. all elements of $\mathfrak{g}$ that are fixed under the Adjoint-action. There exists a canonical inclusion $\mathfrak{z} \subset \Omega^0(\hat{Y}_t,\mathfrak{g}_P)$ which associates to each $\zeta \in \mathfrak{z}$ the section \[ \hat{Y}_t \ni y \mapsto [p_y,\zeta] \in \mathfrak{g}_P \] where $p_y \in P$ is any element in the fiber over $y$.

\begin{definition}\label{def:infinitesimally-irreducible}
A connection $A \in \mathcal{A}(P)$ is called infinitesimally irreducible if the kernel of $\diff_A \col \Omega^0(\hat{Y}_t,\mathfrak{g}_P) \to \Omega^1(\hat{Y}_t,\mathfrak{g}_P)$ equals $\mathfrak{z}$.
\end{definition}

\begin{proposition}\label{prop:infinitesimal_irreducibility}
Assume the set-up of \autoref{theo:perturbing_almost_instantons}. There exists a $0<T_{iI}<T_I$ depending only on $\alpha,\gamma,C,$ and $\varepsilon$ with the following significance: If $t<T_{iI}$ and $\tilde{A}_{\mathfrak{f}}$ for ${\mathfrak{f}}\in \mathfrak{F}$ satisfies the additional hypothesis
\begin{enumerate}
\setcounter{enumi}{3}
\item \label{bul:infinitesimally_irreducible} $\Vert \xi \VertWHt{2}{-\varepsilon}{} \leq C \Vert \diff_{\tilde{A}_{\mathfrak{f}}}^*\diff_{\tilde{A}_{\mathfrak{f}}} \xi \VertWHt{0}{-2-\varepsilon}{}$ for every $\xi \in \mathfrak{z}^{\perp,L^2} \subset \Omega^0(\hat{Y}_t,\mathfrak{g}_P)$,
\item \label{bul: estimate on diff_A} $\Vert \diff_{\tilde{A}_{\mathfrak{f}}} \xi \VertWHt{0}{\beta-1}{} \leq C \Vert \xi \VertWHt{1}{\beta}{}$ for any $\xi \in \Omega^0(\hat{Y}_t,\mathfrak{g}_P)$ and $\beta \in \{0, -1-\varepsilon\}$,
\end{enumerate}
then the $\G_2$-instanton $A_{\mathfrak{f}}$ constructed in \autoref{theo:perturbing_almost_instantons} is infinitesimally irreducible.
\end{proposition}
\begin{proof}
By integration by parts, the connections $A_{\mathfrak{f}}$ and $\tilde{A}_{\mathfrak{f}}$ are infinitesimally irreducible if and only if the associated Laplacians $\diff_{A_{\mathfrak{f}}}^*\diff_{A_{\mathfrak{f}}}$ and $\diff_{\tilde{A}_{\mathfrak{f}}}^*\diff_{\tilde{A}_{\mathfrak{f}}}$ mapping $ C^{2,\alpha}_{-\varepsilon,t}(\mathfrak{g}_P) \to C^{0,\alpha}_{-2-\varepsilon,t}(\mathfrak{g}_P)$ have kernels equal to $\mathfrak{z}$. For $\xi \in C^{2,\alpha}_{-\varepsilon,t}(\mathfrak{g}_P)$ we have by~\eqref{eq:weighted-t-Höldernorms-multiplication},~\eqref{eq:weighted-t-Höldernorms-estimates}, \autoref{theo:perturbing_almost_instantons}, and \autoref{bul: estimate on diff_A}
\begin{align*}
\Vert (\diff_{\tilde{A}_{\mathfrak{f}}}^*\diff_{\tilde{A}_{\mathfrak{f}}} - \diff_{A_{\mathfrak{f}}}^*\diff_{A_{\mathfrak{f}}}) \xi \VertWHt{0}{-2-\varepsilon}{} & \leq \Vert \diff_{\tilde{A}_{\mathfrak{f}}}^*(\diff_{\tilde{A}_{\mathfrak{f}}} - \diff_{A_{\mathfrak{f}}}) \xi \VertWHt{0}{-2-\varepsilon}{} + \Vert (\diff_{\tilde{A}_{\mathfrak{f}}}^* - \diff_{A_{\mathfrak{f}}}^*)\diff_{A_{\mathfrak{f}}} \xi \VertWHt{0}{-2-\varepsilon}{} \\
& \leq c \Vert (\diff_{\tilde{A}_{\mathfrak{f}}} - \diff_{A_{\mathfrak{f}}}) \xi \VertWHt{1}{-1-\varepsilon}{} + \Vert \tilde{A}_{\mathfrak{f}}-A_{\mathfrak{f}} \VertWHt{0}{-1-\varepsilon}{} \Vert \diff_{A_{\mathfrak{f}}} \xi \VertWHt{0}{-1}{} \\
& \leq \Vert \tilde{A}_{\mathfrak{f}}-A_{\mathfrak{f}} \VertWHt{1}{-1-\varepsilon}{} \big(c \Vert \xi \VertWHt{1}{0}{}+ \Vert A_{\mathfrak{f}} - \tilde{A}_{\mathfrak{f}} \VertWHt{0}{-1-\varepsilon}{} \Vert  \xi \VertWHt{0}{\varepsilon}{} \\
&\qquad  +\Vert \diff_{\tilde{A}_{\mathfrak{f}}} \xi \VertWHt{0}{-1}{} \big)\\
&\leq c t^{\gamma-\varepsilon} \big(\Vert \xi \VertWHt{1}{-\varepsilon}{}+  t^{\gamma-\varepsilon} \Vert \xi \VertWHt{0}{-\varepsilon}{} + \Vert \xi \VertWHt{1}{-\varepsilon}{} \big) \\
& \leq c t^{\gamma-\varepsilon} (2+t^{\gamma-\varepsilon}) \Vert \xi \VertWHt{2}{-\varepsilon}{},
\end{align*}
where all constants are independent of $t$ and $\mathfrak{f}$. This together with \autoref{bul:infinitesimally_irreducible} implies that $A_{\mathfrak{f}}$ is infinitesimally irreducible once $c t^{\gamma-\varepsilon} (2+ t^{\gamma-\varepsilon})<\frac{1}{C}$. 
\end{proof}

\section{Approximate $\G_2$-instantons from gluing data}\label{sec:approxG2instantons_from_gluingdata}

Let $(Y_0,\phi_0)$ be a flat $\G_2$-orbifold satisfying \autoref{Ass:Codim6Singularities}. That is, for any connected component of the singular set $S\in \mathcal{S}$ there is a finite subgroup $\Gamma_S < \SU(3)$, a lattice $\Lambda_S < \mathbb{R}$, and an open embedding $\mathtt{J}_S \col \mathcal{V}_S \to Y_0$ for $\mathcal{V}_S \coloneqq (\mathbb{R}\times B_\kappa(0)/\Gamma_S)/\Lambda_S$ that satisfies $\mathtt{J}_S^*\phi_0 = \phi_S$.

\begin{lemma}
Let $(\pi_0\col P_0 \to Y_0,A_0)$ be a flat orbifold principal $G$-bundle. The pullback $\mathtt{J}_S^*\pi_0 \col \mathtt{J}_S^*P_0 \to \mathcal{V}_S$ is isomorphic to \[ \big(\mathbb{R} \times (B_{\kappa}(0) \times_{\Gamma_S} G)\big)/\Lambda_S\] where $\Gamma_S$ and $\Lambda_S$ act on $G$ via the monodromy representation associated to $\mathtt{J}_S^*A_0$. This isomorphism identifies $\mathtt{J}^*_SA_0$ with the flat connection induced from the canonical connection on $(\mathbb{R}\times B_\kappa)\times G$.
\end{lemma}
\begin{proof}
This follows from the general theory of flat bundles (cf. \cite[Proposition~2.2.3]{DonaldsonKronheimer} and \cite[Theorem~2.36]{ALR-Orbifolds_and_stringy_topology}) because $\mathbb{R}\times B_\kappa(0)$ is the universal orbifold covering space of $\mathcal{V}_S$ (cf.~\cite[Example~2.20]{ALR-Orbifolds_and_stringy_topology}) and the orbifold fundamental group fits into the short exact sequence
\begin{equation*}
\begin{tikzcd}
1 \arrow[r] & \Gamma_S \arrow[r] & \piorb \arrow[r] & \Lambda_S \arrow[r] & 1.
\end{tikzcd}\qedhere
\end{equation*}
\end{proof}

In the following we will specify data from which we will subsequently construct the connections over $\hat{Y}_t$ which are our candidates for applying \autoref{theo:perturbing_almost_instantons}.

\begin{definition}[{cf.~\cite[Definition~6.3]{Walpuski-InstantonsKummer}}]\label{def:gluing_data}
Let $(Y_0,\phi_0)$ be a flat $\G_2$-orbifold satisfying ~\autoref{Ass:Codim6Singularities} together with a set $\mathfrak{R}$ of resolution data. A set of gluing data compatible with $\mathfrak{R}$ consists of
\begin{itemize}
\item a flat orbifold principal $G$-bundle $(\pi_0\col P_0 \to Y_0,A_0)$,
\end{itemize} 
and for every $S \in \mathcal{S}$ of 
\begin{itemize}
\item a choice of isomorphism $\tilde{\mathtt{J}}_S \col \big(\mathbb{R}\times (B_\kappa(0)\times_{\Gamma_S} G)\big)/\Lambda_S \to \mathtt{J}^*_SP_0$ as in the previous lemma,
\item a $G$-bundle $\pi_S\col \hat{P}_S \to \hat{Z}_S$ together with a locally framed HYM connection $\hat{A}_S$ asymptotic to the flat connection over the bundle $(\mathbb{C}^3\setminus\{0\}) \times_{\Gamma_S} G$ (with rate $-5$), and
\item a lift $\hat{\tilde{\rho}}_S \col \Lambda_S \to \textup{Isom}(\hat{P}_S)$ of the action $\hat{\rho}_S \col \Lambda_S \to \textup{Isom}(\hat{Z}_S)$ that 
\begin{enumerate}
\item satisfies $\tilde{\tau}_S^{-1}\circ \tilde{\rho}_S \circ \tilde{\tau}_S = \hat{\tilde{\rho}}_S$ where $\tilde{\rho} \col \Lambda_S \to \textup{Isom}((\mathbb{C}^3\setminus\{0\}) \times_{\Gamma_S} G)$ is induced by the monodromy representation of $A_0$. 
\item preserves $\hat{A}_S$.
\end{enumerate}
\end{itemize}
\end{definition}

\begin{remark}
The last point in the previous definition means that for every $g\in \Lambda_S$ the following diagram commutes (where the dashed arrows signify that $\tau_S$ and $\tilde{\tau}_S$ are only defined on (or restricted to) the complement of the exceptional divisor $\tau_S^{-1}(0)$ and map $\hat{A}_S$ only asymptotically to $A_0$):
\begin{equation*}
\begin{tikzcd}
& \big((\mathbb{C}^3\setminus\{0\})\times_{\Gamma_S} G,A_0\big) \arrow[rr, "\tilde{\rho}_S(g)"] \arrow[dd] &  & \big((\mathbb{C}^3\setminus\{0\})\times_{\Gamma_S} G,A_0\big) \arrow[dd] \\
\big(\hat{P}_S,\hat{A}_S\big)\arrow{rr}[near end] {\hat{\tilde{\rho}}_S(g)} \arrow[ur, dashed, "\tilde{\tau}_S"] \arrow[dd]&  & \big(\hat{P}_S,\hat{A}_S\big)\arrow[ur, dashed, "\tilde{\tau}_S"] \arrow[dd] &  \\
 & (\mathbb{C}^3\setminus\{0\})/\Gamma_S \arrow{rr}[near start]{\rho_S(g)}& & (\mathbb{C}^3\setminus\{0\})/\Gamma_S \\
\hat{Z}_S\arrow{rr}{\hat{\rho}_S(g)} \arrow[ur, dashed, "\tau_S"] & & \hat{Z}_S \arrow[ur, dashed, "\tau_S"] &  
\end{tikzcd}
\end{equation*}
\end{remark}

In \autoref{rem:EquivariantGeneralisedKummer} we have seen that some generalised Kummer constructions carry a non-trivial group action preserving the coassociative 4-form. The following extension of \autoref{def:gluing_data} helps us to lift this action to the principal bundle constructed out of the data specified in \autoref{def:gluing_data} and to perform the pre-gluing construction of the connection equivariantly.

\begin{remark}[Equivariant gluing data]\label{rem:equiv_gluing_data}
Let $\lambda_0 \col H \to \textup{Isom}(Y_0,g_{\phi_0})$ be a group action as in \autoref{rem:EquivariantGeneralisedKummer} and let the chosen resolution data be $H$-equivariant. An $H$-equivariant gluing data is a set of gluing data together with a lift $\tilde{\lambda}_0 \col H \to \textup{Isom}(P_0)$ preserving $A_0$ and, moreover:
\begin{enumerate}
\item $\tilde{\mathrm{J}}_S=\tilde{\mathrm{J}}_{hS}$, $\hat{P}_S = \hat{P}_{hS}$, $\tilde{\tau}_S=\tilde{\tau}_{hS}$, $\hat{A}_{S}=\hat{A}_{hS}$, and $\hat{\tilde{\rho}}_S=\hat{\tilde{\rho}}_{hS}$ for every $h\in H$ and $S\in \mathcal{S}$.
\item For every $[S]\in \mathcal{S}/H$ there exists $\tilde{\lambda}_{[S]}\col H \to N_{\textup{Isom}((\mathbb{R}\times \mathbb{C}^3)\times_{\Gamma_S} G)}(\Lambda_{[S]})/\Lambda_{[S]}$ covering $\lambda_{[S]}$ that satisfies \[ \tilde{\mathtt{J}}_{[S]} \circ [\tilde{\lambda}_{[S]}(h)] = \tilde{\lambda}(h) \circ \tilde{\mathtt{J}}_{[S]}\] for every $h\in H$ (where $[\tilde{\lambda}_{[S]}(h)]$ denotes the induced map on $((\mathbb{R}\times \mathbb{C}^3)\times_{\Gamma_{[S]}} G)/\Lambda_{[S]}$).
\item For every $[S]\in \mathcal{S}/H$ we have a lift $\hat{\tilde{\lambda}}_{[S]} \col H \to N_{\textup{Isom}(\mathbb{R}\times\hat{P}_{[S]})}(\Lambda_{[S]})/\Lambda_{[S]}$ of $\hat{\lambda}_{[S]}$ that preserves $\hat{A}_{[S]}$ and satisfies \[ \tilde{\tau}_{[S]} \circ [\hat{\tilde{\lambda}}_{[S]}(h)] = [\tilde{\lambda}_{[S]}(h)] \circ \tilde{\tau}_{[S]}\] for every $h\in H$.
\end{enumerate}
\end{remark}

\begin{proposition}[{cf.~\cite[Proposition~6.7]{Walpuski-InstantonsKummer}}]\label{prop:pregluing}
Let $(Y_0,\phi_0)$ be a flat $\G_2$-orbifold satisfying \autoref{Ass:Codim6Singularities} together with a chosen set of resolution data $\mathfrak{R}$. Furthermore, assume that $\mathfrak{G}$ is a set of gluing data compatible with $\mathfrak{R}$. Then there exists for each $t\in (0,T_K)$ a bundle $\pi_t\col \hat{P}_t \to \hat{Y}_t$ together with a connection $\tilde{A}_t$. Moreover, for each $k\in \mathbb{N}_0$, $\alpha\in(0,1)$, and $\vartheta\geq 0$ there exists a constant $c_{pI}$ (independent of $t<T_K \equiv T_K(k,\alpha)$) with \[ \Vert F_{\tilde{A}_t} \wedge \psi_t \VertWHt{k}{-6+\vartheta}{} \leq c_{pI} t^{\textup{min}\{9/2-\vartheta,4\}}. \] 

In addition: if $\lambda_0 \col H \to \textup{Isom}(Y_0)$ is a group action and the resolution and gluing data are chosen $H$-equivariantly (in the sense of \autoref{rem:EquivariantGeneralisedKummer} and \autoref{rem:equiv_gluing_data}), then there exists a lift $\hat{\tilde{\lambda}}\col H \to \textup{Isom}(\hat{P}_t)$ (covering the action $\hat{\lambda}$ of \autoref{rem:EquivariantGeneralisedKummer}) that preserves $\tilde{A}_t$.
\end{proposition}
The proof is almost identical to \cite[Proof of Proposition~6.7]{Walpuski-InstantonsKummer}, where Kummer constructions resolving orbifolds with codimension 4 singularities and trivial group actions were considered. We have included a proof, nevertheless, for the sake of completeness and to verify the $H$-equivariance in the construction.
\begin{proof}
For the construction of $\hat{P}_t$ we first set $P_0^\circ \coloneqq {P_{0}}_{\vert Y_0 \setminus \mathrm{J}(\mathcal{V}_{\kappa/8})}$ and for each $S\in \mathcal{S}$ we define $\hat{P}^\circ_S \coloneqq (\mathbb{R} \times \hat{P}_S)/\Lambda_S$. Note that the latter naturally carries the structure of a $G$-bundle $\pi^\circ_S \col \hat{P}^\circ_S \to (\mathbb{R}\times \hat{Z}_S)/\Lambda_S$. Furthermore, let in the following \[ \widetilde{t \tau} \col \hat{P}_S^\circ \to ((\mathbb{R}\times B_{\kappa}(0))\times_{\Gamma_S} G)/\Lambda_S \] denote the concatenation of first applying $[\textup{Id}_{\mathbb{R}}\times\tilde{\tau}]$ and then parallel translation (with respect to the flat connection) over the (shortest) straight line connecting any $[s,z]\in (\mathbb{R}\times B_{t^{-1}\kappa}(0)/\Gamma_S)/\Lambda_S$ with $[s,tz]\in \mathcal{V}_{S,\kappa}$. As a manifold \[ \hat{P}_t \coloneqq P_0^\circ \cup \big(\cup_{S\in \mathcal{S}} \hat{P}^\circ_{S\vert \hat{\mathcal{V}}_{S,\kappa}^t} \big)/\sim \] where $\hat{P}^\circ_{S\vert \hat{\mathcal{V}}_{S,\kappa}^t\setminus \hat{\mathcal{V}}^t_{S,\kappa/8}} \ni p \sim \tilde{\mathtt{J}}_S(\widetilde{t\tau}(p)) \in P^\circ_{0 \vert \mathtt{J}(\mathcal{V}_{\kappa}\setminus \mathcal{V}_{\kappa/8})}$. Since $\tilde{\mathtt{J}}_S \circ \widetilde{t\tau_S}$ is a bundle map, $\hat{P}_t$ carries a canonical $G$-bundle structure $\pi_t \col \hat{P}_t \to \hat{Y}_t$.

For the construction of $\tilde{A}_t$ we first note that the connections $(\hat{A}_S)_{S\in \mathcal{S}}$ over $(\hat{P}_S)_{S\in \mathcal{S}}$ induce connections over $(\hat{P}^\circ_S)_{S\in\mathcal{S}}$ which we again denote by $\hat{A}_S$. Furthermore, over $\hat{\mathcal{V}}_{S,\kappa}^t$ these satisfy: \[ \hat{A}_S =  \widetilde{t \tau}^*\hat{\mathtt{J}}_S^* A_0 + a_S  \] where $a_S \in \Omega^1(\hat{\mathcal{V}}^t_{S,\kappa}, \mathfrak{g}_{\hat{P}^\circ_S})$ satisfies $\vert \nabla^k a_S \vert_{g_\mathbb{R} \oplus t^2 g_{\hat{Z}_S}} = t^4 \mathcal{O}(r_t^{-5-k}).$

The connection over $\hat{P}_t$ is now defined as
\begin{align*}
\tilde{A}_t \coloneqq 
\begin{cases}
A_0 &\textup{ over $(P_0^\circ)_{\vert r_t^{-1}[\kappa/4,\kappa]}$} \\
(\widetilde{t \tau})^*\tilde{\mathtt{J}}_S^* A_0 + \widetilde{\chi}_t a_S &\textup{ over each $\hat{P}^\circ_S$}
\end{cases}
\end{align*} 
where $\widetilde{\chi}_t \coloneqq \widetilde{\chi} \circ r_t$ for a smooth non-increasing function $\widetilde{\chi} \col [0,\kappa] \to [0,1]$ with 
\begin{align*}
\widetilde{\chi}(s) = \begin{cases}
0 \textup{ for $s \geq \kappa/4$}\\
1 \textup{ for $s \leq \kappa/8$.}
\end{cases}
\end{align*}

To estimate the error we first observe that the connection $\tilde{A}_t$ agrees over $(P_0^\circ)_{\vert r_t^{-1}[\kappa/4,\kappa]}$ with $A_0$ and is therefore flat in this area. Over each $\hat{P}_S^\circ$ we have \[F_{\tilde{A}_t} = \widetilde{\chi}_t F_{\hat{A}_S} + \diff \widetilde{\chi}_t \wedge a_S + \tfrac{\widetilde{\chi}_t^2-\widetilde{\chi}_t}{2} [a_S\wedge a_S]. \] The curvature of $\hat{A}_S$ can be estimated by $\vert \nabla^k F_{\hat{A}_S} \vert_{g_\mathbb{R} \oplus t^2 g_{\hat{Z}_S}} = t^4 \mathcal{O}(r_t^{-6-k}) $ which together with~\eqref{eq:weighted-t-Höldernorms-multiplication}, \autoref{ex:HYM_is_G2-instanton}, and \autoref{theo:torsionfree_G2_structure} leads to 
\begin{align*}
\Vert \widetilde{\chi}_t F_{\hat{A}_S} \wedge \psi_t \VertWHt{k}{-6}{(\hat{\mathcal{V}}^t_{S,\kappa})} &\leq \Vert \widetilde{\chi}_t F_{\hat{A}_S} \wedge (\psi_t-\tilde{\psi}_t) \VertWHt{k}{-6}{(\hat{\mathcal{V}}^t_{S,\kappa})} \leq c t^{1/2} \Vert F_{\hat{A}_S} \VertWHt{k}{-6}{(\hat{\mathcal{V}}^t_{S,\kappa})} \leq c t^{9/2}.
\end{align*}
Finally, \eqref{eq:weighted-t-Höldernorms-estimates} implies $\Vert F_{\hat{A}_t} \wedge \psi_t \VertWHt{k}{-6+\vartheta}{} \leq c t^{9/2-\vartheta}$. 

The rest in the expression of $F_{\tilde{A}_t}$ is supported in the region $\hat{\mathcal{V}}_{S,\kappa}^t\setminus \hat{\mathcal{V}}_{S,\kappa/8}^t$ where the $C^{k,\alpha}_{\beta,t}$-norms for all values of $\beta$ are equivalent (uniformly in $t$). This leads to 
\begin{align*}
\left\Vert \left(\diff \widetilde{\chi}_t \wedge a_S + \tfrac{\widetilde{\chi}_t^2 - \widetilde{\chi}_t}{2}[a_S\wedge a_S]\right) \wedge \psi_t  \right\VertWHt{k}{-6+\vartheta}{(\hat{\mathcal{V}}^t_{S,\kappa})}   & \leq c \big( \Vert a_S \VertWHt{k}{-5}{(\hat{\mathcal{V}}^t_{S,\kappa})} + \Vert a_S \VertWHt{k}{-5}{(\hat{\mathcal{V}}^t_{S,\kappa})}^2 \big) \\
& \leq c \big(t^4+t^8\big)
\end{align*}
and in total \[ \Vert F_{\tilde{A}_t} \wedge \psi_t \VertWHt{k}{-6+\vartheta}{} \leq c t^{\textup{min}\{9/2-\vartheta,4\}}. \] 

Assume now that $\mathfrak{R}$ and $\mathfrak{G}$ are $H$-equivariant for some $H$-action as defined in \autoref{rem:EquivariantGeneralisedKummer} and \autoref{rem:equiv_gluing_data}. We then obtain induced $H$-actions on $P_0^\circ$ and $\hat{P}_S^\circ$ (via $\tilde{\lambda}$ and $\hat{\tilde{\lambda}}_{S}$). Since $\tilde{\tau}_S$ and parallel transport are equivariant, these combine to $\hat{\tilde{\lambda}}\col H \to \textup{Isom}(\hat{P}_t)$. Since $A_0$, $A_S$, and $\tilde{\chi}_t$ are invariant under $H$, the same holds true for $\tilde{A}_t$.
\end{proof}
\begin{remark}
An improved estimate on the $\G_2$-structure $\phi_t$ (cf. \autoref{rem:improvedEstimates}) of the form \[ \Vert \phi_t - \tilde{\phi}_t \VertWHt{k}{0}{} \leq c t^\gamma \] would improve the error estimate in the previous proposition by \[ \Vert F_{\tilde{A}_t} \wedge \psi_t \VertWHt{k}{-6+\vartheta}{} \leq c t^{\textup{min}\{4+\gamma-\vartheta,4\}}. \]  
\end{remark}

\begin{proposition}[Pregluing for a simple family of gluing data]\label{prop:pregluing-family}
Assume that we have a set of gluing data $\mathfrak{G}\coloneqq \{(\pi \col P_0\to Y_0,A_0), (\tilde{\mathtt{J}}_S,\pi_S \col \hat{P}_S \to \hat{Z}_S, \tilde{\tau}_S, \hat{A}_S, \hat{\tilde{\rho}}_S)_{S\in \mathcal{S}}\} $ compatible with a chosen set of resolution data. Assume further that $\mathfrak{F}\subset \mathbb{R}$ is a compact interval and that we have
\begin{itemize}
\item a smooth family $(A_0 + a_{0,\mathfrak{f}})_{\mathfrak{f} \in \mathfrak{F}} \subset \mathcal{A}(P_0)$ of flat connections on $P_0$,
\item for every $S\in \mathcal{S}$ a smooth family $(\hat{A}_S + \hat{a}_{S,\mathfrak{f}})_{\mathfrak{f} \in \mathfrak{F}} \subset \mathcal{A}((\mathbb{R}\times \hat{P}_S)/\Lambda_S)$ of $\G_2$-instantons,
\end{itemize}
and that these families of connections satisfy for every $S\in \mathcal{S}$ and $\mathfrak{f}\in \mathfrak{F}$:
\begin{itemize}
\item $\tilde{\mathtt{J}}_S^*a_{0,\mathfrak{f}} = (\tilde{\tau}_S)_*\hat{a}_{S,\mathfrak{f}}$ over $\mathcal{V}_{S,\kappa}\setminus \mathcal{V}_{S,\kappa/8}$,
\item the restriction of $(\tilde{\tau}_S)_*\hat{a}_{S,\mathfrak{f}}$ to $(\mathbb{R}\times (\mathbb{C}^3 \setminus B_{\kappa/8}(0))/\Gamma_S)/\Lambda_S$ satisfies \[\delta_{\lambda}^*((\tilde{\tau}_S)_*\hat{a}_{S,\mathfrak{f}}) = (\tilde{\tau}_S)_*\hat{a}_{S,\mathfrak{f}} \] for every $\lambda\geq 1$, where $\delta_\lambda \col (\mathbb{R}\times (\mathbb{C}^3 \setminus B_{\kappa/8}(0))/\Gamma_S)/\Lambda_S \to (\mathbb{R}\times (\mathbb{C}^3 \setminus B_{\lambda\kappa/8}(0))/\Gamma_S)/\Lambda_S$ denotes the dilation by $\lambda$ in $\mathbb{C}^3/\Gamma_S$-direction (i.e. $\delta_\lambda([s,z]) \coloneqq [s,\lambda z]$). 
\end{itemize}
Then there exists for each $t \in (0,T_K)$ a bundle $\pi_t \col \hat{P}_t \to \hat{Y}_t$ together with a smooth family of connections $(\tilde{A}_{t} + \tilde{a}_{t,\mathfrak{f}})_{\mathfrak{f} \in \mathfrak{F}} \subset \mathcal{A}(\hat{P}_t)$ such that for each $L,k\in \mathbb{N}_0$, $\alpha \in (0,1)$, and $\vartheta\geq 0$ there exists a ($t$-independent) constant $c_{pI}>0$ with \[ \Vert (\partial_\mathfrak{F}^\ell F_{\tilde{A}_t + \tilde{a}_{t,\mathfrak{f}}}) \wedge \psi_t \VertWHt{k}{-6+\vartheta}{} \leq c_{pI} t^{\min\{9/2-\vartheta,4\}} \] and 
\begin{align*}
\Vert (\partial_\mathfrak{F}^\ell L_{\tilde{A}_t+\tilde{a}_{t,\mathfrak{f}}}) \underline{a} \VertWHt{k}{\beta-1}{} \leq c_{pI} \Vert \underline{a} \VertWHt{k+1}{\beta}{} &\textup{ for any $\underline{a} \in \Omega^1(\hat{Y}_t,\mathfrak{g}_P)\oplus \Omega^7(\hat{Y}_t,\mathfrak{g}_P)$ and $\beta \in \mathbb{R}$,}\\
\Vert (\partial_\mathfrak{F}^\ell L^*_{\tilde{A}_t+\tilde{a}_{t,\mathfrak{f}}}) \underline{b} \VertWHt{k}{\beta-1}{} \leq c_{pI} \Vert \underline{b} \VertWHt{k+1}{\beta}{} &\textup{ for any $\underline{b} \in \Omega^6(\hat{Y}_t,\mathfrak{g}_P)\oplus \Omega^0(\hat{Y}_t,\mathfrak{g}_P)$ and $\beta \in \mathbb{R}$}
\end{align*} 
for any $\ell \in \{0,\dots, L\}$. Here we again denote by $\partial_{\mathfrak{F}}^\ell F_{\tilde{A}_t + \tilde{a}_{t,\mathfrak{f}}}\in \Omega^2(\hat{Y}_t,\mathfrak{g}_P)$ and $\partial_{\mathfrak{F}}^\ell L_{\tilde{A}_{t}+\tilde{a}_{f,t}} \in \textup{Lin}(C^{k+1,\alpha}_{\beta,t},C^{k,\alpha}_{\beta-1,t})$ the $\ell$-th derivative of the family $\mathfrak{f} \mapsto F_{\tilde{A}_t + \tilde{a}_{t,\mathfrak{f}}} \in \Omega^2(\hat{Y}_t,\mathfrak{g}_P)$ and $\mathfrak{f} \mapsto L_{\tilde{A}_{t}+\tilde{a}_{f,t}} \in \textup{Lin}(C^{k+1,\alpha}_{\beta,t},C^{k,\alpha}_{\beta-1,t})$. Furthermore, the covariant derivatives in the $\Vert \cdot \VertWHt{k}{\beta}{}$-norms in the estimates above are taken with respect to $\tilde{A}_t$ and the Levi-Civita connection with respect to $\tilde{g}_t$. 

In addition: if $\mathfrak{G}$ is a set of $H$-equivariant gluing data and both $a_{0,\mathfrak{f}}$ and $\hat{a}_{S,\mathfrak{f}}$ are $H$-invariant for every $\mathfrak{f} \in \mathfrak{F}$, then so is $\tilde{A}_t+\tilde{a}_{t,\mathfrak{f}}$.
\end{proposition}
\begin{remark}
The first condition on the families of connections in the previous proposition ensures that the respective connections match over the gluing region. The second condition says that on the complement of $\mathcal{V}_{S,\kappa/8} \subset (\mathbb{R}\times \mathbb{C}^3/\Gamma_S)/\Lambda_S$, the (asymptotic) 1-form $(\tau_S)_* \hat{a}_{S,\mathfrak{f}}$ is dilation-invariant. This is in particular satisfies whenever the 1-form $(\tilde{\tau}_S)_*\hat{a}_{S,\mathfrak{f}}$ is on the complement of $\mathcal{V}_{S,\kappa/8}$ pulled back from $(\mathbb{R}\times S^5/\Gamma_S)/\Lambda_S \subset(\mathbb{R}\times (\mathbb{C}^3\setminus \{0\})/\Gamma_S)/\Lambda_S$.
\end{remark}
\begin{remark}\label{rem: pregluing families weakening assumptions}
As in the formulation of \autoref{theo:perturbing_almost_instantons} we restricted ourself in the previous proposition to families parametrised by an interval $\mathfrak{F} \subset \mathbb{R}$, because this simplifies the notation and is the relevant case for our examples in \autoref{sec:Example18}. However, \autoref{prop:pregluing-family} can be generalised to any \ita{compact} manifold $\mathfrak{F}$ (possibly with boundary). Similarly, one can weaken the condition that $(\tilde{\tau}_S)_* \hat{a}_{S,\mathfrak{f}}$ is dilation-invariant on the complement of $\mathcal{V}_{S,\kappa/8}$ and agrees with $\tilde{J}_S^*a_{0,\mathfrak{f}}$ in the gluing region to respective asymptotic conditions.
\end{remark}
\begin{remark}\label{rem: estimate on d_A^*}
If $\underline{b} \in \Omega^6(\hat{Y}_t,\mathfrak{g}_P)\oplus \Omega^0(\hat{Y}_t,\mathfrak{g}_P)$ is of the form $b=(0,\xi)$ for some $\xi \in \Omega^0(\hat{Y}_t,\mathfrak{g}_P)$, then $L_{\tilde{A}_t+\tilde{a}_{t,\mathfrak{f}}}^*$ acts as $L_{\tilde{A}_t+\tilde{a}_{f,t}}^* \underline{b} = \diff_{\tilde{A}_t + \tilde{a}_{t,\mathfrak{f}}} \xi$. The previous proposition implies therefore the estimate \[\Vert \diff_{\tilde{A}_t + \tilde{a}_{t,\mathfrak{f}}} \xi \VertWHt{k}{\beta-1}{} \leq c_{pI} \Vert \xi \VertWHt{k+1}{\beta}{} \textup{ for any $\xi \in \Omega^0(\hat{Y}_t,\mathfrak{g}_P)$ and $\beta \in \mathbb{R}$.} \]
\end{remark}
\begin{proof}[Proof of \autoref{prop:pregluing-family}]
First, we construct $\hat{\pi}_t \col \hat{P}_t \to \hat{Y}_t$ and $\tilde{A}_t \in \mathcal{A}(\hat{P}_t)$ from $\mathfrak{G}$ as in \autoref{prop:pregluing}. Next, we define $\tilde{a}_{t,\mathfrak{f}} \in \Omega^1(\hat{Y}_t,\mathfrak{g}_{\hat{P}_t})$ as follows:
\begin{align*}
\tilde{a}_{t,\mathfrak{f}} \coloneqq \begin{cases}
a_{0,\mathfrak{f}} \textup{ over $Y_0\setminus \mathtt{J}(\mathcal{V}_{\kappa/8})$,} \\
\hat{a}_{S,\mathfrak{f}} \textup{ over any $\hat{\mathcal{V}}_{S,\kappa}^t$.}
\end{cases}
\end{align*}
Our conditions ensure that $\tilde{\mathtt{J}}^*_Sa_{0,\mathfrak{f}}  =  (\widetilde{t\tau_S})_*\hat{a}_{S,\mathfrak{f}}$ over $\mathcal{V}_{S,\kappa}\setminus\mathcal{V}_{S,\kappa/8}$ for any $t>0$. Thus, $\tilde{a}_{t,\mathfrak{f}}$ is well-defined and smooth. Furthermore, the $H$-invariance (in the presence of an $H$-action) of $\tilde{a}_{t,\mathfrak{f}}$ follows immediately from the $H$-invariance of $a_{0,\mathfrak{f}}$ and $\hat{a}_{S,\mathfrak{f}}$.

We will now prove all pregluing-estimates for $L=0$. The estimates for higher $L$ can be proven analogously.

For the error estimate we note as in the proof of \autoref{prop:pregluing} that over $P_0^\circ$, the connection $\tilde{A}_t + \tilde{a}_{t,\mathfrak{f}}$ agrees with the flat connection $A_{0}+a_{0,\mathfrak{f}}$. Over each $\hat{P}_S^\circ$ we have \[F_{\tilde{A}_t+\tilde{a}_{t,\mathfrak{f}}} = \widetilde{\chi}_t F_{\hat{A}_S+\hat{a}_{S,\mathfrak{f}}} + \diff \widetilde{\chi}_t \wedge a_S + \tfrac{\widetilde{\chi}_t^2-\widetilde{\chi}_t}{2} [a_S\wedge a_S]+ (1-\widetilde{\chi}_t)(\widetilde{t\tau_S})^*\tilde{\mathtt{J}}_S^*\big(\diff_{A_0}a_{0,\mathfrak{f}} + \tfrac{1}{2}[a_{0,\mathfrak{f}} \wedge a_{0,\mathfrak{f}}]\big) \] (where $a_S\in \Omega^1(\hat{Z}_S,\mathfrak{g}_{\hat{P}_S})$ is as in the previous proof defined by $\hat{A}_S-(\widetilde{t\tau_S})^*\tilde{\mathtt{J}}_S^*A_0$). Note that the last term vanishes because $\diff_{A_0}a_{0,\mathfrak{f}} + \tfrac{1}{2}[a_{0,\mathfrak{f}}\wedge a_{0,\mathfrak{f}}] = F_{A_0+a_{0,\mathfrak{f}}} = 0$ and that the second and the third term can be estimated as in the proof of \autoref{prop:pregluing}. Furthermore, over the complement of $\tau_S^{-1}(\mathcal{V}_{S,\kappa/8}) \subset (\mathbb{R}\times \hat{Z}_S)/\Lambda_S$ we have
\begin{align}
F_{\hat{A}_S+\hat{a}_{S,\mathfrak{f}}} &= F_{\hat{A}_S} + (\widetilde{t\tau_S})^*\tilde{\mathtt{J}}_S^*\big(\diff_{A_0}a_{0,\mathfrak{f}} + \tfrac{1}{2}[a_{0,\mathfrak{f}} \wedge a_{0,\mathfrak{f}}]\big) + [a_S,\hat{a}_{S,\mathfrak{f}}] \notag\\
&= F_{\hat{A}_S} + [a_S,\hat{a}_{S,\mathfrak{f}}]. \label{equ: pregluing family curvature outside compactum}
\end{align}

Since $\delta_{\lambda}^*(\tilde{\tau}_S)_*(\hat{a}_{S,\mathfrak{f}}) = (\tilde{\tau}_S)_*(\hat{a}_{S,\mathfrak{f}})$ for any $\lambda\geq 1$ over \[(\mathbb{R}_s\times (\mathbb{R}_{r>\kappa/8} \times S^5/\Gamma_S))/\Lambda_S = (\mathbb{R}\times (\mathbb{C}^3\setminus B_{\kappa/8}(0))/\Gamma_S)/\Lambda_S,\] we may write $(\tilde{\tau}_S)_*\hat{a}_{S,\mathfrak{f}}$ over this domain as \[(\tilde{\tau}_S)_*\hat{a}_{S,\mathfrak{f}} = \pr_{S^5/\Gamma_S}^*(a_{r,\mathfrak{f}}) \tfrac{\diff r}{r} + \pr_{S^5/\Gamma_S}^*(a_{S^5/\Gamma_S,\mathfrak{f}}) + \pr_{S^5/\Gamma_S}^*(a_{s,\mathfrak{f}}) \diff s \] where $a_{r,\mathfrak{f}},a_{s,\mathfrak{f}} \in \Omega^0((\mathbb{R}\times S^5/\Gamma_S)/\Lambda_S,\mathfrak{g}_{P})$ and $a_{S^5/\Gamma_S,\mathfrak{f}}\in \Omega^1((\mathbb{R}\times S^5/\Gamma_S)/\Lambda_S,\mathfrak{g}_{P})$ satisfies $i_{\partial_s}(a_{S^5/\Gamma_S,\mathfrak{f}})=0$ and where $\pr_{S^5/\Gamma_S}\col (\mathbb{R}\times (\mathbb{C}^3\setminus \{0\})/\Gamma_S)/\Lambda_S \to (\mathbb{R}\times S^5/\Gamma_S)/\Lambda_S$ denotes the radial projection. In the following we will drop the pullback by $\pr_{S^5/\Gamma_S}$ from our notation. 

Since $\mathfrak{F}$ and $(\mathbb{R}\times S^5/\Gamma_S)/\Lambda_S$ are compact, the explicit form of $(\tilde{\tau}_S)_*\hat{a}_{S,\mathfrak{f}}$ over the region $(\mathbb{R}\times B_{t^{-1}\kappa}(0)/\Gamma_S)/\Lambda_S$ implies that \[ \vert \nabla^k (a_{r,\mathfrak{f}} \tfrac{\diff r}{r} + a_{S^5/\Gamma_S,\mathfrak{f}}) \vert_{g_{\mathbb{R}}\oplus t^2 g_{\mathbb{C}^3}} \leq c w_t^{-k-1}  \] and \[ \vert \nabla^k (a_{s,\mathfrak{f}} \diff s) \vert_{g_{\mathbb{R}}\oplus t^2 g_{\mathbb{C}^3}} \leq c w_t^{-k} \] and by~\eqref{eq:weighted-t-Höldernorms-estimates}
\begin{equation}\label{eq:weighted_estimate_family_1-form}
\Vert \hat{a}_{S,\mathfrak{f}} \VertWHt{k}{-1}{(\hat{\mathcal{V}}_{S,\kappa}^t)} \leq \Vert a_{r,\mathfrak{f}} \tfrac{\diff r}{r} + a_{S^5/\Gamma_S,\mathfrak{f}} \VertWHt{k}{-1}{(\hat{\mathcal{V}}_{S,\kappa}^t)} + \Vert a_{s,\mathfrak{f}} \diff s \VertWHt{k}{0}{(\hat{\mathcal{V}}_{S,\kappa}^t)} \leq c. 
\end{equation} 

Using \eqref{equ: pregluing family curvature outside compactum} and \eqref{eq:weighted_estimate_family_1-form} we can bound as in the proof of \autoref{prop:pregluing}
\begin{align*}
\Vert \widetilde{\chi}_t F_{\hat{A}_S+\hat{a}_{S,\mathfrak{f}}} \wedge \psi_t \VertWHt{k}{-6}{(\hat{\mathcal{V}}_{S,\kappa}^t)} & \leq c t^{1/2} \Vert F_{\hat{A}_S+\hat{a}_{S,\mathfrak{f}}} \VertWHt{k}{-6}{(\hat{\mathcal{V}}_{S,\kappa}^t)} \\
& \leq ct^{1/2} \big( \Vert F_{\hat{A}_S} \VertWHt{k}{-6}{(\hat{\mathcal{V}}_{S,\kappa}^t)} + \Vert a_S \VertWHt{k}{-5}{(\hat{\mathcal{V}}_{S,\kappa}^t)} \cdot \Vert \hat{a}_{S,\mathfrak{f}} \VertWHt{k}{-1}{(\hat{\mathcal{V}}_{S,\kappa}^t)} \big) \\
&\leq c t^{9/2}
\end{align*}
which together with~\eqref{eq:weighted-t-Höldernorms-estimates} implies the pregluing estimate.

In order to prove $\Vert L_{\tilde{A}_t+\tilde{a}_{t,\mathfrak{f}}} \underline{a} \VertWHt{k}{\beta-1}{} \leq c \Vert \underline{a} \VertWHt{k+1}{\beta}{}$ for every $\underline{a}\in \Omega^1(\hat{Y}_t,\mathfrak{g}_P)\oplus \Omega^7(\hat{Y}_t,\mathfrak{g}_P)$ we first note that \[L_{\tilde{A}_t+\tilde{a}_{t,\mathfrak{f}}}\underline{a} = L_{\tilde{A}_t} \underline{a} + \{\tilde{a}_{t,\mathfrak{f}},\underline{a}\} \] where $\{ \cdot, \cdot \}$ denotes a bilinear algebraic operation. Since we consider $C^{k,\alpha}_{\beta,t}$-norms on $\hat{Y}_t$ in which the derivative is taken with respect to $\tilde{A}_t$, it suffices to estimate \[ \Vert \{\tilde{a}_{t,\mathfrak{f}} , \underline{a} \} \VertWHt{k}{\beta-1}{} \leq c \Vert \tilde{a}_{t,\mathfrak{f}} \VertWHt{k}{-1}{} \cdot \Vert \underline{a} \VertWHt{k}{\beta}{}.\] Since $\mathfrak{F}$ is compact and the restriction of the weighted norm $\Vert \cdot \VertWHt{k}{\beta}{}$ to $\hat{Y}_t \setminus \hat{\mathcal{V}}_{\kappa}^t$ is uniformly (in $t\in (0,T)$) equivalent to the ordinary Hölder norm on $Y_0 \setminus \mathtt{J}(\mathcal{V}_{\kappa})$, we have $\Vert \tilde{a}_{t,\mathfrak{f}} \VertWHt{k}{-1}{(\hat{Y}_t \setminus \hat{\mathcal{V}}_{\kappa}^t)} \leq c$. Together with~\eqref{eq:weighted_estimate_family_1-form} this implies $\Vert \tilde{a}_{t,\mathfrak{f}} \VertWHt{k}{-1}{} \leq c$ and finishes the estimate on $L_{\tilde{A}_t + \tilde{a}_{t,\mathfrak{f}}}$. Since $L_{\tilde{A}_t + \tilde{a}_{t,\mathfrak{f}}}^* = *L_{\tilde{A}_t + \tilde{a}_{t,\mathfrak{f}}}*$ and $\Vert * \VertWHt{k}{0}{} < c$, this also implies the estimate on $L_{\tilde{A}_t + \tilde{a}_{t,\mathfrak{f}}}^*$.
\end{proof}

\section{Linear analysis} \label{sec: linear analysis}

This section establishes \autoref{bul:linearEstimate} of \autoref{theo:perturbing_almost_instantons} and \autoref{bul:infinitesimally_irreducible} of \autoref{prop:infinitesimal_irreducibility} for the approximate instantons constructed in \autoref{prop:pregluing} and \autoref{prop:pregluing-family}.

\subsection{A model operator on $\mathbb{R}\times\text{ALE}$}

Let $(\hat{Z},\tau,\hat{\omega},\hat{\Omega})$ be a Calabi--Yau ALE $3$-fold, $\pi \col \hat{P} \to \hat{Z}$ be a locally framed principal $G$-bundle together with a Hermitian Yang--Mills connection $A \in \mathcal{A}_{-5}^{\textup{HYM}}(\hat{P},A_\infty)$. 

\begin{definition}\label{def:anticomplex_HodgeStar}
For any $q\in \{0,1,2,3\}$ define $*_{\hat{\Omega}} \col \Lambda^{0,q}T^*_\mathbb{C} \hat{Z}\to \Lambda^{0,3-q}T^*_\mathbb{C} \hat{Z}$ by demanding that \[ \eta_1 \wedge *_{\hat{\Omega}} \eta_2 = \tfrac{1}{\sqrt{8}} \langle \eta_1, \eta_2 \rangle_{\mathbb{C}} \hat{\bar{\Omega}} \] for any $\eta_1,\eta_2 \in \Lambda^{0,q}T^*_\mathbb{C} \hat{Z}$.
\end{definition}

\begin{remark}
The map $*_{\hat{\Omega}}$ is an anticomplex version of the Hodge star operator and appeared, for example, in \cite[Section~2]{DonaldsonThomas-HigherDimGaugeTheory}.
\end{remark}
The following identities can be proven either by wedging with a $(0,q)$-form and inserting $\vol = \frac{i}{8} \hat{\Omega}\wedge\hat{\bar{\Omega}}$ or as for the ordinary Hodge star operator.
\begin{lemma}\label{lem:anticomplex-HodgeStar}
This operator satisfies the following identities:
\begin{align*}
*\bar{\eta}&= \tfrac{(-1)^{(3-q)+1}i}{\sqrt{8}} \hat{\Omega} \wedge *_{\hat{\Omega}} \eta\\
*_{\hat{\Omega}}(i\eta) &= -i *_{\hat{\Omega}}\eta \\
*_{\hat{\Omega}}^2 \eta &= (-1)^{q(3-q)} \eta\\
\langle *_{\hat{\Omega}} \eta_1,*_{\hat{\Omega}} \eta_2 \rangle&= \overline{\langle  \eta_1,\eta_2 \rangle}\\
({*_{\hat{\Omega}} \bar{\partial} *_{\hat{\Omega}}}) \eta &= (-1)^q \bar{\partial}^* \eta
\end{align*}
where $\eta,\eta_1,\eta_2\in\Omega^{0,q}(\hat{Z},\mathbb{C})$.
\end{lemma}

We now consider the product $\hat{Y}\coloneqq \mathbb{R}\times \hat{Z}$ equipped with its $\G_2$-structure~\eqref{eq:modelG2Structure} and with the pullback bundle (again denoted by) $\pi \col \hat{P} \to \hat{Y}$. By \autoref{ex:HYM_is_G2-instanton}, the pullback of $A$ to this bundle is a $\G_2$-instanton which we will again denote by $A$. 

\begin{lemma}
Consider the isomorphisms of vector bundles over $\hat{Y}$ (where we suppress any pullbacks in our notation) given by
\begin{align*}
T^*\mathbb{R}\otimes \mathfrak{g}_{\hat{P}} \oplus T^*\hat{Z}\otimes \mathfrak{g}_{\hat{P}} \oplus \Lambda^7T^* \hat{Y}\otimes \mathfrak{g}_{\hat{P}} &\to \Lambda^{0,1}T^*_\mathbb{C}\hat{Z}\otimes \mathfrak{g}_{\hat{P}} \oplus \Lambda^{0,3}T^*_\mathbb{C}\hat{Z}\otimes \mathfrak{g}_{\hat{P}} \\
\big(\diff s \otimes \xi_1, a, \vol_{\hat{Y}} \otimes \xi_2 \big) & \mapsto \big(\sqrt{2} a^{0,1},*_{\hat{\Omega}}(\xi_1+i\xi_2) \big)
\end{align*}
and
\begin{align*}
T^*\mathbb{R}\otimes \Lambda^5T^*\hat{Z}\otimes \mathfrak{g}_{\hat{P}} \oplus \Lambda^6
T^*\hat{Z}\otimes \mathfrak{g}_{\hat{P}} \oplus \mathfrak{g}_{\hat{P}} &\to  \mathfrak{g}_{\hat{P}}^\mathbb{C} \oplus \Lambda^{0,2}T^*_\mathbb{C}\hat{Z}\otimes \mathfrak{g}_{\hat{P}} \\
\big(\diff s \otimes b, \vol_{\hat{Z}} \otimes \xi_1, \xi_2  \big) & \mapsto \big(\xi_2-i\xi_1, \tfrac{1}{2} i_{\hat{\Omega}^\sharp} b \big),
\end{align*}
where $i_{\hat{\Omega}^\sharp}$ denotes the insertion of $\hat{\Omega}^\sharp \in \Lambda^{3,0}T_\mathbb{C}\hat{Z}$.

Under the induced isomorphisms 
\begin{align*}
\Omega^1(\hat{Y},\mathfrak{g}_{\hat{P}}) \oplus \Omega^7(\hat{Y},\mathfrak{g}_{\hat{P}}) &\cong \Omega^0(\hat{Y},\Lambda^{0,\textup{odd}}T^*_\mathbb{C}\hat{Z} \otimes \mathfrak{g}_{\hat{P}}) \\
\Omega^6(\hat{Y},\mathfrak{g}_{\hat{P}}) \oplus \Omega^0(\hat{Y},\mathfrak{g}_{\hat{P}}) &\cong \Omega^0(\hat{Y},\Lambda^{0,\textup{even}}T^*_\mathbb{C}\hat{Z} \otimes \mathfrak{g}_{\hat{P}})
\end{align*}
the linearised instanton operator and its adjoint become \begin{align*}
L_A &= 
\begin{pmatrix}
0 &  - *_{\hat{\Omega}}\partial_s  \\
*_{\hat{\Omega}} \partial_s  & 0
\end{pmatrix}
+ \sqrt{2} 
\begin{pmatrix}
\bar{\partial}_A^* & 0 \\
\bar{\partial}_A &  \bar{\partial}_A^*
\end{pmatrix}\\
L_A^* & = 
\begin{pmatrix}
0 & - *_{\hat{\Omega}}\partial_s \\
 *_{\hat{\Omega}} \partial_s & 0 
\end{pmatrix}
+ \sqrt{2} 
\begin{pmatrix}
\bar{\partial}_A & \bar{\partial}_A^* \\
0 &  \bar{\partial}_A
\end{pmatrix}.
\end{align*}
These satisfy \[ L_A L_A^* = - \partial_s^2 +2 {\Delta_{\bar{\partial}_A}}_{\vert \Omega^{0,\textup{even}}}.\]

\end{lemma}

\begin{proof}
As a first step, we identifiy
\begin{align*}
T^*\mathbb{R}\otimes \mathfrak{g}_{\hat{P}} \oplus T^*\hat{Z}\otimes \mathfrak{g}_{\hat{P}} \oplus \Lambda^7T^* \hat{Y}\otimes \mathfrak{g}_{\hat{P}} &\to \mathfrak{g}_{\hat{P}} \oplus \mathfrak{g}_{\hat{P}} \oplus T^*\hat{Z} \otimes \mathfrak{g}_{\hat{P}} \\
\big(\diff s \otimes \xi_1, a,  \vol_{\hat{Y}} \otimes \xi_2 \big) & \mapsto \big(\xi_1,\xi_2, a \big)
\end{align*}
and
\begin{align*}
T^*\mathbb{R}\otimes \Lambda^5T^*\hat{Z}\otimes \mathfrak{g}_{\hat{P}} \oplus \Lambda^6
T^*\hat{Z}\otimes \mathfrak{g}_{\hat{P}} \oplus \mathfrak{g}_{\hat{P}} &\to \mathfrak{g}_{\hat{P}} \oplus \mathfrak{g}_{\hat{P}} \oplus \Lambda^5 T^*\hat{Z} \otimes \mathfrak{g}_{\hat{P}} \\
\big(\diff s \otimes b,  \vol_{\hat{Z}} \otimes \xi_1, \xi_2  \big) & \mapsto \big(\xi_1,\xi_2, b \big).
\end{align*}
The operator then becomes
\begin{align*}
L_A = \begin{pmatrix}
0 & \partial_s & \Lambda_{\hat{\omega}}\diff_A \\
-\partial_s & 0 & \diff_A^*  \\
- *_{\hat{Z}}\hat{I}^*\diff_A & -*_{\hat{Z}} \diff_A & *_{\hat{Z}} \hat{I}^* \partial_s + \re \hat{\Omega} \wedge \diff_A
\end{pmatrix}
\end{align*}
where both $\diff_A$ and $\diff_A^*$ are taken over $\hat{Z}$. This makes use of the identities $\frac{1}{2} *_{\hat{Z}} (\hat{\omega}\wedge\hat{\omega}\wedge b) = \Lambda_{\hat{\omega}} b$ for $b \in \Omega^2(\hat{Z})$
and $\frac{1}{2} *_{\hat{Z}}(\hat{\omega}\wedge\hat{\omega}\wedge a)= -I^*a$ for $a\in \Omega^1(\hat{Z})$ (cf. \cite[Chapter~1.2]{Huybrechts-Complex}). Going through the other identifications is straight forward and uses the identities given in the previous lemma.
\end{proof}

The following lemma immediately follows from \cite[Lemma~7.10]{Walpuski-InstantonsKummer}:
\begin{lemma}
Identify \begin{align*}
\Omega^6(\hat{Y},\mathfrak{g}_{\hat{P}}) \oplus \Omega^0(\hat{Y},\mathfrak{g}_{\hat{P}}) &\cong \Omega^0(\hat{Y},\Lambda^{0,\textup{even}}T^*_\mathbb{C}\hat{Z} \otimes \mathfrak{g}_{\hat{P}})
\end{align*} 
as in the previous lemma and assume that $\underline{b} \in \Omega^0(\hat{Y},\Lambda^{0,\textup{even}}T^*_\mathbb{C}\hat{Z} \otimes \mathfrak{g}_{\hat{P}})$ satisfies $L_AL_A^* \underline{b}=0$ and $\Vert \underline{b}\VertC{0}{} < \infty$. Then $\underline{b}$ is constant in the $\mathbb{R}$ direction.
\end{lemma}

Next, we introduce the following weight functions on $\hat{Y} = \mathbb{R} \times \hat{Z}$:  \[w(s,z) \coloneqq 1 + r(z) \quad \text{and} \quad w((s_1,z_1),(s_2,z_2))\coloneqq \min\{w(s_1,z_1),w(s_2,z_2)\},\] and for $k\in \mathbb{N}_0$ and $\alpha\in(0,1)$ the following weighted norms:
\begin{align*}
[f\WsH{0}{\beta}{(U)} &\coloneqq \sup_{d(x,y)\leq w(x,y)} w(x,y)^{\alpha-\beta} \frac{\vert f(x)-f(y) \vert}{d(x,y)^\alpha}  \\
\Vert f \VertWC{0}{\beta}{(U)} & \coloneqq \big\Vert w^{-\beta} f \big{\VertC{0}{(U)}} \\
\Vert f \VertWH{k}{\beta}{(U)} & \coloneqq \sum_{j=0}^k \left\Vert \nabla^j f \right\VertWC{0}{\beta-j}{(U)} + \left[\nabla^j f\right\WsH{0}{\beta-j}{(U)}.
\end{align*} 

\begin{proposition}\label{prop:instanton_deformations_come_from_HYM-deformations}
Identify \begin{align*}
\Omega^6(\hat{Y},\mathfrak{g}_{\hat{P}}) \oplus \Omega^0(\hat{Y},\mathfrak{g}_{\hat{P}}) &\cong \Omega^0(\hat{Y},\Lambda^{0,\textup{even}}T^*_\mathbb{C}\hat{Z} \otimes \mathfrak{g}_{\hat{P}})
\end{align*}
as in the previous lemmas and assume that $\underline{b}\coloneqq (\xi,\eta) \in C^{k,\alpha}_{\ell oc}(\Lambda^{0,\textup{even}}T^*_\mathbb{C}\hat{Z} \otimes \mathfrak{g}_{\hat{P}})$ satisfies $L_AL_A^* \underline{b}=0$ and $\Vert \underline{b}\VertWC{0}{\beta}{} < \infty$ for any $\beta <0$. Then $\xi=0$ and $\eta$ is the pullback of an element in $*_{\hat{\Omega}}\mathcal{H}^1_{A,-5}$ (as in \autoref{def:HYM_infinitesimal_deformations}).
\end{proposition}

\begin{proof}
First note that $\underline{b}$ is smooth by elliptic regularity. The previous lemma implies that $\underline{b}$ is constant in $\mathbb{R}$ direction and therefore $\Delta_{\bar{\partial}_A}\underline{b}=0$. This means \[\Delta_{\bar{\partial}_A} \xi = 0 \quad \textup{and} \quad \Delta_{\bar{\partial}_A} \eta = 0\] and therefore $\xi=0$ by \autoref{lem:delbar_harmonic_with_decay}.

One can now argue as in the proof of \autoref{prop:HYMConnection improve decay}: Since there are no indicial roots of $\Delta_{\bar{\partial}_A}$ contained in the interval $(-4,0)$, we have $\vert \nabla^k \eta \vert = \mathcal{O}(r^{-4-k})$ (cf.~\cite[Proposition~1.14]{Bartnik-MassofALF} 
or~\cite{LockhartMcOwen-ellipticOperators_nonCompactMfds}). Integration by parts implies $(\bar{\partial}_A+\bar{\partial}_A^*)\eta=0$ and by the last identity in \autoref{lem:anticomplex-HodgeStar} \[ (\bar{\partial}_A+\bar{\partial}_A^*)(*_{\hat{\Omega}}\eta) = *_{\hat{\Omega}} \bar{\partial}_A^* \eta - *_{\hat{\Omega}} \bar{\partial}_A \eta = 0. \] By \autoref{prop:kernel-HYM-decay}, $*_{\hat{\Omega}} \eta \in \mathcal{H}^1_{A,-5}$ or, equivalently, $\eta \in *_{\hat{\Omega}}\mathcal{H}^1_{A,-5}$.

Somewhat alternatively, if one already has $\Vert \underline{b} \VertWC{1}{\beta}{}<\infty$ (which holds in the proof of \autoref{prop:LinearEstimate} where we apply this proposition), one can also argue as follows: First note that by $\Delta_{\bar{\partial}_A} \eta = 0$, $\bar{\partial}_A^2=0$ and \autoref{lem:delbar_harmonic_with_decay} we have $\bar{\partial}_A \eta =0$. This implies that $\bar{\partial}_A^*\eta$ satisfies $(\bar{\partial}_A + \bar{\partial}_A^*) \bar{\partial}_A^*\eta =0$ and therefore (by \autoref{prop:kernel-HYM-decay}) $\bar{\partial}_A^*\eta \in \mathcal{H}^1_{A,-5}$. In particular, $\vert \bar{\partial}_A^*\eta \vert$ decays like $r^{-5}$ and since $\vert \eta\vert$ decays like $r^{\beta}$ with $\beta<0$, we have \[ 0 = \int_{\hat{Z}} \langle \bar{\partial}_A\bar{\partial}_A^*\eta ,\eta\rangle = \int_{\hat{Z}} \vert \bar{\partial}_A^* \eta \vert^2. \] Thus, $(\bar{\partial}_A + \bar{\partial}_A^*) \eta = 0$ and as above one shows that $\eta \in *_{\hat{\Omega}} \mathcal{H}^1_{A,-5}$.
\end{proof}

Assume now that the ALE $3$-fold $\hat{Z}\equiv \hat{Z}_S$, the bundle ${\hat{P}}\equiv\hat{P}_S$, and the HYM connection $A\equiv\hat{A}_S$ are part of resolution and gluing data for a flat $\G_2$-orbifold. As before we denote by $\hat{Y}_t$ the manifold defined in \autoref{def:OrbifoldResolution} and by $(\hat{P}_t,\tilde{A}_t)$ the bundle and connection constructed in \autoref{prop:pregluing}.

Next we want to compare the linearised instanton operator $L_{\tilde{A}_t}$ over $\hat{\mathcal{V}}_{S,\kappa}^t$ with the model operator $L_{\hat{A}_S}$. For this we define for $t\in (0,T_K)$ and $s_0\in \mathbb{R}$ the following dilation map
\begin{align*}
\delta_{t,s_0} \col \mathbb{R} \times \hat{Z}_S &\to  (\mathbb{R}\times \hat{Z}_S)/\Lambda_S \\
(s,x) &\mapsto [ts+s_0,x]
\end{align*}
and for $\beta \in \mathbb{R}$ and any $a \in \Omega^{\ell}(\hat{Y}_t,\mathfrak{g}_{\hat{P}_t})$ the following rescaling:
\begin{align}\label{eq:rescaling_ALE}
\mathrm{s}_{\beta,t,s_0}(a) \coloneqq t^{-\beta-\ell} \delta_{t,s_0}^*a_{\vert \hat{\mathcal{V}}_{S,\kappa}^t}.
\end{align}
We will denote the induced map on any sum of the form $\Omega^k(\hat{Y}_t,\mathfrak{g}_{\hat{P}_t})\oplus \Omega^{\ell}(\hat{Y}_t,\mathfrak{g}_{\hat{P}_t})$ by $\mathrm{s}_{\beta,t,s_0}$ as well.

\begin{lemma}[{cf.~\cite[Proposition~7.7]{Walpuski-InstantonsKummer}}]\label{lem:ALE-rescaling-comparison}
For fixed $\beta\in \mathbb{R}$, $k\in \mathbb{N}_0$, and $\alpha\in(0,1)$ there exists a constant $C$ such that for every $t \in (0,T_K)$, $a \in \Omega^{\ell}(\hat{Y}_t,\mathfrak{g}_{\hat{P}_t})$, $\underline{a} \in \Omega^{1}(\hat{Y}_t,\mathfrak{g}_{\hat{P}_t}) \oplus \Omega^{7}(\hat{Y}_t,\mathfrak{g}_{\hat{P}_t})$, and $\underline{b} \in \Omega^{6}(\hat{Y}_t,\mathfrak{g}_{\hat{P}_t}) \oplus \Omega^{0}(\hat{Y}_t,\mathfrak{g}_{\hat{P}_t})$ the following holds:
\[C^{-1} \Vert a \VertWHt{k}{\beta}{(\hat{\mathcal{V}}_{S,\kappa}^t)} \leq \Vert \mathrm{s}_{\beta,t,s_0} (a) \VertWH{k}{\beta}{(\mathbb{R}\times (t \tau_S)^{-1}(B_\kappa(0)/\Gamma_S))} \leq C\Vert a \VertWHt{k}{\beta}{(\hat{\mathcal{V}}_{S,\kappa}^t)} 
\]
\[
\Vert (L_{\tilde{A}_t} - \mathrm{s}_{\beta-1,t,s_0}^{-1} L_{\hat{A}_S} \mathrm{s}_{\beta,t,s_0})(\underline{a}) \VertWHt{k}{\beta-1}{(\hat{\mathcal{V}}_{S,\kappa}^t)}\leq C t^{1/2} \Vert \underline{a} \VertWHt{k+1}{\beta}{(\hat{\mathcal{V}}_{S,\kappa}^t)}
\]
and 
\[
\Vert (L_{\tilde{A}_t}^* - \mathrm{s}_{\beta-1,t,s_0}^{-1} L_{\hat{A}_S}^* \mathrm{s}_{\beta,t,s_0})(\underline{b}) \VertWHt{k}{\beta-1}{(\hat{\mathcal{V}}_{S,\kappa}^t)}\leq C t^{1/2} \Vert \underline{b} \VertWHt{k+1}{\beta}{(\hat{\mathcal{V}}_{S,\kappa}^t)}.
\]
\end{lemma}
\begin{proof}
The proof is the same as \cite[Proposition~7.7]{Walpuski-InstantonsKummer}: $\delta_{t,s_0}$ pulls back the metric $g_\mathbb{R}\oplus t^2 g_{\hat{Z}_S}$ to $t^2(g_\mathbb{R} \oplus g_{\hat{Z}_S})$. Thus, if we would calculate the norm on $\hat{\mathcal{V}}_{S,\kappa}^t$ with respect to the product metric $g_\mathbb{R}\oplus t^{2}g_{\hat{Z}_S}$, then the first chain of inequalities would actually hold as an equality. However, the norm on $\hat{\mathcal{V}}_{S,\kappa}^t$ is calculated with respect to the metric $\tilde{g}_t$ which only agrees with $g_\mathbb{R}\oplus t^2 g_{\hat{Z}_S}$ up to $\mathcal{O}(t^{6})$; thus the inequality.

For the second, we calculate 
\begin{align*}
L_{\hat{A}_S} s_{\beta,t,s_0} (a,\xi) = t^{-\beta-7} \begin{pmatrix}
t^6\hat{\psi}_{S} \wedge \diff_{\hat{A}_S} \delta_{t,s_0}^* a & -\diff_{\hat{A}_S}^* \delta_{t,s_0}^*\xi \\
t^6 \diff_{\hat{A}_S}^* \delta_{t,s_0}^* a & 0
\end{pmatrix}
\end{align*}
(where $\hat{\psi}_S$ denotes the product $\G_2$-structure on $\mathbb{R}\times \hat{Z}_S$ and $*_S$ its associated Hodge-star operator) and 
\begin{align*}
s_{\beta-1,t,s_0} L_{\tilde{A}_t} (a,\xi) = t^{-\beta-5} \begin{pmatrix}
\delta_{t,s_0}^* \big(\psi_t \wedge \diff_{\tilde{A}_t} a\big) & - \delta_{t,s_0}^* \diff_{\tilde{A}_t}^* \xi \\
t^6\delta_{t,s_0}^* \diff_{\tilde{A}_t}^*  a & 0
\end{pmatrix}.
\end{align*}
Again, if everything in the definition of $L_{\tilde{A}_t}$ came from a product structure (i.e. on $\hat{\mathcal{V}}_{S,\kappa}^t$: $\tilde{A}_t =\hat{A}_S$ and $\tilde{\psi}_t = \hat{\psi}_{S,t}$ where $\hat{\psi}_{S,t}$ is associated to $g_\mathbb{R}\oplus t^2 g_{\hat{Z}_S}$), then both operators would agree. As in the previous situation one needs to estimate how far all entries in $L_{\tilde{A}_t}$ are from those defined from product structures. This uses \autoref{theo:torsionfree_G2_structure} which explains the factor $t^{1/2}$ in the estimate. The same holds true for $L_{\tilde{A}_t}^*$. 
\end{proof}

For the rest of this section we make the following assumption:

\begin{assumption}\label{ass: very simple family of gluing data}
Let $\mathfrak{F} \subset \mathbb{R}$ be a compact interval and $(\tilde{A}_t+\tilde{a}_{t,\mathfrak{f}})_{\mathfrak{f} \in \mathfrak{F}}$ be the family of connections over $\hat{P}_t$ that arises from the construction in \autoref{prop:pregluing-family}. Assume that over $S\in \mathcal{S}$ the family of 1-forms $\hat{a}_{S,\mathfrak{f}} \in \Omega^1((\mathbb{R}\times \hat{Z}_S)/\Lambda_S,\mathfrak{g}_{\hat{P}_S})$ in the construction of \autoref{prop:pregluing-family} is of the form $\hat{a}_{S,\mathfrak{f}} = \hat{\xi}_{S,\mathfrak{f}} \diff s$ where both $\hat{\xi}_{S,\mathfrak{f}} \in \Omega^0((\mathbb{R}\times\hat{Z}_S)/\Lambda_S,\mathfrak{g}_{\hat{P}_S})$ and $(\tilde{\tau}_S)_*\hat{\xi}_{S,\mathfrak{f}}$ are constant.
\end{assumption}
\begin{remark}
As in \autoref{rem:perturbing almost instantons weakening assumptions} and \autoref{rem: pregluing families weakening assumptions}, these assumptions can be weakened at the cost of a heavier notation.
\end{remark}

 For $t>0$ and $\mathfrak{f} \in \mathfrak{F}$ we define $L_{\hat{A}_S+t\hat{a}_{S,\mathfrak{f}}}$ to be the linearised instanton operator associated to the connection $\hat{A}_S+t \hat{a}_{S,\mathfrak{f}} \in \mathcal{A}(\hat{P}_S)$ over $\mathbb{R}\times \hat{Z}_S$.
\begin{proposition}\label{prop:ALE-operator-comparison-family} 
For fixed $\beta\in \mathbb{R}$, $k\in \mathbb{N}_0$, and $\alpha\in(0,1)$ there exists a constant $C$ such that for every $t \in (0,T_K)$, $\mathfrak{f} \in \mathfrak{F}$, $\underline{a} \in \Omega^{1}(\hat{Y}_t,\mathfrak{g}_{\hat{P}_t}) \oplus \Omega^{7}(\hat{Y}_t,\mathfrak{g}_{\hat{P}_t})$, and $\underline{b} \in \Omega^{6}(\hat{Y}_t,\mathfrak{g}_{\hat{P}_t}) \oplus \Omega^{0}(\hat{Y}_t,\mathfrak{g}_{\hat{P}_t})$ the following estimates hold:
\begin{align*}
\Vert (L_{\tilde{A}_t+\tilde{a}_{t,\mathfrak{f}}} - \mathrm{s}_{\beta-1,t,s_0}^{-1} L_{\hat{A}_S+t \hat{a}_{S,\mathfrak{f}}} \mathrm{s}_{\beta,t,s_0})(\underline{a}) \VertWHt{k}{\beta-1}{(\hat{\mathcal{V}}_{S,\kappa}^t)}&\leq C t^{1/2} \Vert \underline{a} \VertWHt{k+1}{\beta}{(\hat{\mathcal{V}}_{S,\kappa}^t)}\\
\Vert (L_{\tilde{A}_t+\tilde{a}_{t,\mathfrak{f}}}^* - \mathrm{s}_{\beta-1,t,s_0}^{-1} L_{\hat{A}_S+t \hat{a}_{S,\mathfrak{f}}}^* \mathrm{s}_{\beta,t,s_0})(\underline{b}) \VertWHt{k}{\beta-1}{(\hat{\mathcal{V}}_{S,\kappa}^t)}&\leq C t^{1/2} \Vert \underline{b} \VertWHt{k+1}{\beta}{(\hat{\mathcal{V}}_{S,\kappa}^t)}.
\end{align*}
\end{proposition}
\begin{proof}
This is proven as in \autoref{lem:ALE-rescaling-comparison}. The $t$-dependence of the model connection $\hat{A}_S +t \hat{a}_{S,\mathfrak{f}}$ over $\mathbb{R}\times\hat{Z}_S$ is a consequence of $\delta_{t,s_0}^* (\hat{\xi}_{S,f}\diff s) = t \hat{\xi}_{S,f}\diff s$.
\end{proof}
\begin{proposition}\label{prop:Schauder-estimate-families-over-ALE}
For fixed $\beta\in \mathbb{R}$, $k\in \mathbb{N}_0$, and $\alpha \in (0,1)$, there exists a $c_{SI}>0$ and $0<T_{SI}<T_K$ such that for every $t \in (0,T_{SI})$, $\mathfrak{f} \in \mathfrak{F}$, and $\underline{b}\in \Omega^6(\mathbb{R}\times\hat{Z}_S,\mathfrak{g}_{\hat{P}_S}) \oplus \Omega^0(\mathbb{R}\times\hat{Z}_S,\mathfrak{g}_{\hat{P}_S})$ the estimate 
\begin{align*}
\Vert \underline{b} \VertWH{k+2}{\beta}{(\mathbb{R}\times\tau^{-1}(B_{t^{-1}\kappa/2}(0)))} &\leq c_{SI} \big(\Vert L_{\hat{A}_S+t \hat{a}_{S,\mathfrak{f}}} L_{\hat{A}_S+t \hat{a}_{S,\mathfrak{f}}}^* \underline{b} \VertWH{k}{\beta-2}{(\mathbb{R}\times\tau^{-1}(B_{t^{-1}\kappa}(0)))} \\
& \qquad \qquad \qquad \qquad\qquad\qquad \qquad \quad + \Vert \underline{b} \VertWC{0}{\beta}{(\mathbb{R}\times\tau^{-1}(B_{t^{-1}\kappa}(0)))} \big)  
\end{align*} 
holds over $\mathbb{R}\times\tau_S^{-1}(B_{t^{-1}\kappa}(0)) \subset \mathbb{R}\times \hat{Z}_S$.
\end{proposition}
\begin{proof}[Proof sketch]
The Schauder estimate for the fixed translation- and asymptotically dilation-invariant operator $L_{\hat{A}_S}L_{\hat{A}_S}^*$ is standard and proven in \cite[Proposition~7.11]{Walpuski-InstantonsKummer} (see also \cite[Theorem~3.1]{NirenbergWalker-Null_space_of_elliptic_operator} and \cite[Proposition~1.6]{Bartnik-MassofALF}). The proof for $L_{\hat{A}_S+t \hat{a}_{S,\mathfrak{f}}} L_{\hat{A}_S+t \hat{a}_{S,\mathfrak{f}}}^*$ (which is not asymptotically dilation-invariant anymore due to the $t \hat{a}_{S,\mathfrak{f}}$-term) and the observation that $c_{SI}$ can be chosen independently of $\mathfrak{f}$ is similar and we only sketch the argument:

When restricted to $\mathbb{R}\times K$ for a fixed compact set $K\subset \hat{Z}_S$, the operator $L_{\hat{A}_S+t \hat{a}_{S,\mathfrak{f}}} L_{\hat{A}_S+t \hat{a}_{S,\mathfrak{f}}}^*\in \textup{Lin}(C^{k+2,\alpha}_\beta,C^{k,\alpha}_{\beta-2})$ is for small $t< T_{SI}$ an arbitrarily small perturbation of $L_{\hat{A}_S} L_{\hat{A}_S}^*$. The estimate for $L_{\hat{A}_S+t \hat{a}_{S,\mathfrak{f}}} L_{\hat{A}_S+t \hat{a}_{S,\mathfrak{f}}}^*$ over $\mathbb{R}\times K$ follows therefore from the $\mathfrak{f}$-independent Schauder estimate of $L_{\hat{A}_S}L^*_{\hat{A}_S}$.

Choosing $K \subset \hat{Z}_S$ in the previous step sufficiently large, we may assume that over $(\mathbb{R}\times \hat{Z}_S)\setminus (\mathbb{R}\times K)$ the connections $\hat{A}_S + t\hat{a}_{S,\mathfrak{f}}$ and $(\tilde{\tau}_S)^*A_\infty + t \hat{a}_{S,\mathfrak{f}}$ are as close as we wish. For any $p\coloneqq (s,z) \in (\mathbb{R}\times \hat{Z}_S)\setminus (\mathbb{R}\times K)$ and $R\coloneqq \tfrac{1}{8}(1+r(z))$ we can now prove a local estimate \[ \Vert \underline{b} \VertWH{k+2}{\beta}{(B_R(p))} \leq c \big(\Vert L_{\hat{A}_S+t \hat{a}_{S,\mathfrak{f}}} L_{\hat{A}_S+t \hat{a}_{S,\mathfrak{f}}}^* \underline{b} \VertWH{k}{\beta-2}{(B_{2R}(p))} + \Vert \underline{b} \VertWC{0}{\beta}{(B_{2R}(p))} \big) \] over the ball $B_R(p)$ as follows: After rescaling (and identifying $\hat{Z}_S\setminus K \cong (\mathbb{C}^3/\Gamma_S) \setminus \tau_S(K)$) this estimate becomes equivalent to  
\begin{align}\label{equ: auxiliary rescaled Schauder estimate}
\Vert \underline{b} \VertH{k+2}{(B_1(0))} \leq c \big( \Vert  L_{A_\infty +tR(\tau_S)_* \hat{a}_{S,\mathfrak{f}}}L_{A_\infty +tR(\tau_S)_* \hat{a}_{S,\mathfrak{f}}}^* \underline{b} \VertH{k}{(B_2(0)} + \Vert \underline{b} \VertC{0}{(B_2(0))} \big)
\end{align}
over the fixed balls $B_1(0),B_2(0) \subset \mathbb{R}\times \mathbb{C}^3$ (see also \autoref{lem:C^3-rescaling-comparison} below). If now $p \in (\mathbb{R}\times \tau_S^{-1}(B_{t^{-1}\kappa}(0)))$, then $tR\leq T_{SI}+\kappa$ and $L_{A_\infty +tR(\tau_S)_* \hat{a}_{S,\mathfrak{f}}}L_{A_\infty +tR(\tau_S)_* \hat{a}_{S,\mathfrak{f}}}^*$ differs from $L_{A_\infty}L_{A_\infty}^*$ only by lower order terms which are uniformly bounded. The estimate \eqref{equ: auxiliary rescaled Schauder estimate} follows therefore from the standard interior Schauder estimate.

The estimate over $\mathbb{R}\times K$ and over each $B_R(p)$ for $p \in (\mathbb{R}\times \tau_S^{-1}(B_{t^{-1}\kappa}(0)))\setminus (\mathbb{R}\times K)$ imply the estimate over $(\mathbb{R}\times \tau_S^{-1}(B_{t^{-1}\kappa}(0)))$.
\end{proof}

We end this section by comparing the family of operators $L_{\hat{A}_S+t\hat{a}_{S,\mathfrak{f}}}$ to its asymptotic limit $L_{A_\infty+t (\tau_S)_*\hat{a}_{S,\mathfrak{f}}}$ on $\mathbb{R}\times \mathbb{C}^3/\Gamma_S$. For this recall that $\hat{Z}_S$ comes equipped with a resolution map $\tau_S \col \hat{Z}_S \to \mathbb{C}^3/\Gamma_S$ which is covered by the framing $\tilde{\tau}_S \col \hat{P}_{\vert \hat{Z}_S\setminus \tau_S^{-1}(0)} \to (P_\infty)_{\vert \mathbb{C}^3\setminus\{0\}/\Gamma_S}$ such that $\vert \nabla^k ((\tilde{\tau}_S)_*A - A_\infty)\vert = \mathcal{O}(r^{-5-k})$. 

\noindent For $R>0$ let 
\begin{align*}
\tilde{\delta}_R \col \mathbb{R}\times \mathbb{C}^3/\Gamma &\to \mathbb{R}\times \mathbb{C}^3/\Gamma  \\
(s,z) & \mapsto (Rs,Rz)
\end{align*}
and with this dilation at hand we define the following rescaling:
\begin{definition}\label{def:scaling_on_C^3}
Let $\beta,R_1,R_2\in \mathbb{R}$ be constants with $R_1, R_2>0$. For $a \in \Omega^{\ell}(\mathbb{R}\times \hat{Z},\mathfrak{g}_{\hat{P}})$ we define \[\tilde{\mathrm{s}}_{\beta,R_1,R_2}(a) \in \Omega^{\ell}\big(\mathbb{R}\times (\mathbb{C}^3\setminus B_{R_1/R_2}(0))/\Gamma),\mathfrak{g}_{P_\infty}\big)  \]
as 
\begin{align*}
\tilde{\mathrm{s}}_{\beta,R_1,R_2}(a) \coloneqq R_2^{-\beta-\ell} \tilde{\delta}_{R_2}^*\tilde{\tau}(a_{\vert \hat{Z}\setminus \tau^{-1}(B_{R_1})})
\end{align*}
where we use parallel transport with respect to $A_\infty$ in radial direction to identify different fibers of $\mathfrak{g}_{P_\infty}$. As above, we will denote the induced map on any sum of the form $\Omega^k(\mathbb{R}\times\hat{Z},\mathfrak{g}_{\hat{P}})\oplus \Omega^{\ell}(\mathbb{R}\times\hat{Z},\mathfrak{g}_{\hat{P}})$ by $\tilde{\mathrm{s}}_{\beta,R_1,R_2}$ as well.
\end{definition}
\begin{lemma}\label{lem:C^3-rescaling-comparison} 
For fixed $\beta\in \mathbb{R}$, $k\in \mathbb{N}_0$, and $\alpha\in(0,1)$ there exists a constant $C$ such that for every $R_1, R_2 \geq 1$: 
\[C^{-1} \Vert \underline{a} \VertWH{k}{\beta}{(\mathbb{R}\times (\hat{Z}\setminus B_{R_1}))} \leq \Vert \tilde{\mathrm{s}}_{\beta,R_1,R_2} \underline{a} \VertWH{k}{\beta}{(\mathbb{R}\times (\mathbb{C}^3\setminus B_{R_1/R_2})/\Gamma)} \leq C\Vert \underline{a} \VertWH{k}{\beta}{(\mathbb{R}\times (\hat{Z}\setminus B_{R_1}))} 
\]
\[
\Vert (L_{\hat{A}_S+t\hat{a}_{S,\mathfrak{f}}} - \tilde{\mathrm{s}}_{\beta-1,R_1,R_2}^{-1} L_{A_\infty+tR_2 \tau_*\hat{a}_{S,\mathfrak{f}}} \tilde{\mathrm{s}}_{\beta,R_1,R_2})\underline{a} \VertWH{k}{\beta-1}{}\leq C R_1^{-4} \Vert \underline{a} \VertWH{k+1}{\beta}{}
\]
\[
\Vert (L_{\hat{A}_S+t\hat{a}_{S,\mathfrak{f}}}^* - \tilde{\mathrm{s}}_{\beta-1,R_1,R_2}^{-1} L_{A_\infty+tR_2 \tau_*\hat{a}_{S,\mathfrak{f}}}^* \tilde{\mathrm{s}}_{\beta,R_1,R_2})\underline{a} \VertWH{k}{\beta-1}{}\leq C R_1^{-4} \Vert \underline{a} \VertWH{k+1}{\beta}{}
\]
where the family of connection $(\hat{A}_S+\hat{a}_{S,\mathfrak{f}})_{\mathfrak{f} \in \mathfrak{F}}$ is as in \autoref{prop:ALE-operator-comparison-family}.
Furthermore, the weighted $C^{k,\alpha}_\beta$-norm on $\mathbb{R}\times (\mathbb{C}^3\setminus B_{R_1}(0))/\Gamma_S$ is defined as on $\mathbb{R}\times \hat{Z}$ but with the weight function $w(s,z) \coloneqq \vert z \vert$. 
\end{lemma}
This is proven as \autoref{lem:ALE-rescaling-comparison} and \autoref{prop:ALE-operator-comparison-family}. The factor of $R_1^{-4}$ on the right-hand side of the second and third estimate comes from $\Vert A - A_\infty \VertWH{k}{-1}{(\mathbb{R}\times(\hat{Z}\setminus B_{R_1}))} \leq c R_1^{-4}$.

\begin{proposition}\label{prop:instanton_deformations_vanish_over_euklidean_space}
Let $\beta < 0$ and $\underline{b} \in C^{k,\alpha}_{\ell oc} (\mathfrak{g}_{P_\infty}\oplus\Lambda^6T^*(\mathbb{R}\times \mathbb{C}^3/\Gamma)\otimes \mathfrak{g}_{P_\infty})$ satisfy $\Vert \underline{b} \VertWC{0}{\beta}{} < \infty$ and $L_{A_\infty}L_{A_\infty}^* \underline{b}=0$. Then $\underline{b} = 0.$
\end{proposition}

\begin{proof}
As in the proof of \autoref{prop:instanton_deformations_come_from_HYM-deformations} we obtain that (under an analogous identification) $\underline{b}$ is constant in $\mathbb{R}$-direction and satisfies $\Delta_{\bar{\partial}_{A_\infty}} \underline{b}=0$ over $\mathbb{C}^3/\Gamma$. By the maximum principle $\underline{b}=0$.
\end{proof}

\subsection{The linear estimate}

This section verifies the second point of \autoref{theo:perturbing_almost_instantons} and \autoref{bul:infinitesimally_irreducible} in \autoref{prop:infinitesimal_irreducibility} for a simple case of connections constructed in \autoref{prop:pregluing-family}. 
\begin{assumption}\label{ass:invertible_linearisations}
Assume that we are in the set-up of \autoref{prop:pregluing-family} where for every $S\in \mathcal{S}$ and $\mathfrak{f} \in \mathfrak{F}$, the 1-form $\hat{a}_{S,\mathfrak{f}}\in \Omega^1((\mathbb{R}\times\hat{Z}_S)/\Lambda_S,\mathfrak{g}_{\hat{P}_S})$ is as in \autoref{ass: very simple family of gluing data}. Furthermore, assume that
\begin{enumerate}
\item  for every $\mathfrak{f} \in \mathfrak{F}$ the only $H$-invariant element in $\ker(L_{A_0+a_{0,\mathfrak{f}}}^*)$ is $0$,
\item  for each Hermitian Yang--Mills connection $\hat{A}_S\in \mathfrak{G}$ one of the following holds: 
\begin{itemize}
\item either $*_{\hat{\Omega}_S}\mathcal{H}^1_{\hat{A}_S}$ is $0$, where $\mathcal{H}^1_{\hat{A}_S}$ was defined in \autoref{def:HYM_infinitesimal_deformations} and $*_{\hat{\Omega}_S}$ in \autoref{def:anticomplex_HodgeStar},
\item alternatively, let $H_S<H$ consist of all elements $h \in H$ with $\lambda_0(h)_{\vert S}=\id$ and assume that the action $\hat{\tilde{\lambda}}_{[S]} \col H_S \to N_{\textup{Isom}(\mathbb{R}\times\hat{P}_{[S]})}(\Lambda_{[S]})/\Lambda_{[S]}$ (where $\hat{\tilde{\lambda}}_{[S]}$ is as in \autoref{rem:equiv_gluing_data}) lifts to $\hat{\tilde{\lambda}}_{[S]} \coloneqq  \id \times \hat{\tilde{\vartheta}} \col H_S \to \textup{Isom}(\mathbb{R}\times\hat{P}_{[S]})$. We assume that the only $H_S$-invariant element in $*_{\hat{\Omega}_S}\mathcal{H}^1_{\hat{A}_S}$ is $0$.
\end{itemize}
\end{enumerate}
\end{assumption}

\begin{proposition}\label{prop:LinearEstimate}
For $t\in (0,T_K)$ let $\pi_t \col \hat{P}_t \to \hat{Y}_t$ and $(\tilde{A}_t+\tilde{a}_{t,\mathfrak{f}})_{\mathfrak{f} \in \mathfrak{F}} \subset \mathcal{A}(\hat{P}_t)$ be the bundle and connections constructed in \autoref{prop:pregluing-family}. If \autoref{ass:invertible_linearisations} holds, then there exists for any fixed $k \in \mathbb{N}_0,$ $\alpha\in(0,1)$, and $\beta\in(-4,0)$ two constants $0<T_{le}<T_K$ and $c_{le}>0$ such that any $H$-invariant $\underline{b} \in \Omega^1(\hat{Y}_t,\mathfrak{g}_{P_t}) \oplus \Omega^7(\hat{Y}_t, \mathfrak{g}_{P_t})$ satisfies \[ \Vert \underline{b} \VertWHt{k+2}{\beta}{} \leq c_{le} \Vert L_{\tilde{A}_t+\tilde{a}_{t,\mathfrak{f}}}L_{\tilde{A}_t+\tilde{a}_{t,\mathfrak{f}}}^* \underline{b} \VertWHt{k}{\beta-2}{} \]  for any $t\in (0,T_{le})$ and $\mathfrak{f} \in \mathfrak{F}$.
\end{proposition} 

The proof of this proposition is very similar to \cite[Proposition~8.1]{Walpuski-InstantonsKummer}:

\begin{proposition}
For $t\in (0,T_K)$ let $\pi_t \col \hat{P}_t \to \hat{Y}_t$ and $(\tilde{A}_t+\tilde{a}_{t,\mathfrak{f}})_{\mathfrak{f}\in\mathfrak{F}} \subset \mathcal{A}(\hat{P}_t)$ be the bundle and connections constructed in \autoref{prop:pregluing-family}. Assume further, that the 1-forms $\hat{a}_{S,\mathfrak{f}}$ in the construction of $\tilde{a}_{t,\mathfrak{f}}$ are of the form described in \autoref{ass: very simple family of gluing data}. Then there exists for any $\beta \in \mathbb{R}$, $k\in \mathbb{N}_0$, and $\alpha\in(0,1)$ a $c_{SI}^\prime>0$ and $0<T_{SI}^\prime<T_K$ such that for all $t \in (0,T^\prime_{SI})$, $\mathfrak{f} \in \mathfrak{F}$, and $\underline{b} \in \Omega^6(\hat{Y}_t,\mathfrak{g}_{\hat{P}_t}) \oplus \Omega^0(\hat{Y}_t,\mathfrak{g}_{\hat{P}_t})$ we have \[ \Vert \underline{b} \VertWHt{k+2}{\beta}{} \leq c_{SI}^\prime \big(\Vert L_{\tilde{A}_t+\tilde{a}_{t,\mathfrak{f}}}L_{\tilde{A}_t+\tilde{a}_{t,\mathfrak{f}}}^* \underline{b} \VertWHt{k}{\beta-2}{} + \Vert \underline{b} \VertWCt{0}{\beta}{} \big). \]
\end{proposition}
\begin{proof}
As in the proof of \autoref{prop:Schauder-estimate-families-over-ALE} it suffices to prove this over the two regions $r_t^{-1}([0,\kappa/2])$ and $r_t^{-1}[\kappa/2,\kappa]$ separately. Over $r_t^{-1}([0,\kappa/2])$ the estimate follows for sufficiently small $t$ from \autoref{lem:ALE-rescaling-comparison}, \autoref{prop:ALE-operator-comparison-family}, and \autoref{prop:Schauder-estimate-families-over-ALE}. 

The metric $\tilde{g}_t$ and the family of connections $(\tilde{A}_t+\tilde{a}_{t,\mathfrak{f}})_{\mathfrak{f} \in \mathfrak{F}}$ converges in $C^{\infty}(r_t^{-1}[\varepsilon,\kappa])$ for any fixed $\varepsilon>0$ to $g_0$ and $(A_0 + a_{0,\mathfrak{f}})_{\mathfrak{f} \in \mathfrak{F}}$. Since $\mathfrak{F}$ is compact, the coefficients in $L_{A_0+a_{0,\mathfrak{f}}}L_{A_0+a_{0,\mathfrak{f}}}^*$ are uniformly bounded and the estimate over $r_t^{-1}([\kappa/2,\kappa])$ follows therefore from the ordinary interior Schauder estimate on balls in $Y_0$ of a fixed size.
\end{proof}

The previous proposition reduces the proof of \autoref{prop:LinearEstimate} to the following:
\begin{proposition}
Assume that we are in the setup of \autoref{prop:LinearEstimate}. For any $\alpha\in (0,1)$ and $\beta \in (-4,0)$ there are constants $0<T_{le}^\prime <T_K$ and $c_{le}^\prime>0$ such that for all $t \in (0,T^\prime_{le})$, $\mathfrak{f} \in \mathfrak{F}$ and $H$-invariant $\underline{b} \in \Omega^6(\hat{Y}_t,\mathfrak{g}_{\hat{P}_t}) \oplus \Omega^0(\hat{Y}_t,\mathfrak{g}_{\hat{P}_t})$ we have \[ \Vert \underline{b} \VertWCt{0}{\beta}{} \leq c \Vert L_{\tilde{A}_t+\tilde{a}_{t,\mathfrak{f}}}L_{\tilde{A}_t+\tilde{a}_{t,\mathfrak{f}}}^* \underline{b} \VertWHt{0}{\beta-2}{} \]
\end{proposition}
\begin{proof}
In various places throughout the proof it will be necessary to take a subsequence of a certain sequence. By abuse of notation we will tacitly denote this subsequence by the same letter as the original sequence. 

Assume that such a constant does not exist to produce a contradiction. Then there exist sequences $(\underline{b}_n)_{n\in \mathbb{N}}$ of $H$-invariant sections, $(t_n)_{n\in \mathbb{N}} \subset (0,T_K)$, and $(\mathfrak{f}_n)_{n\in\mathbb{N}}\subset \mathfrak{F}$ such that \[ t_n \to 0, \quad \Vert \underline{b}_n \VertWCt{0}{\beta}{} = 1, \quad \textup{and} \quad \Vert L_{\tilde{A}_{t_n}+\tilde{a}_{t_n,\mathfrak{f}_n}}L_{\tilde{A}_{t_n}+\tilde{a}_{t_n,\mathfrak{f}_n}}^* \underline{b}_n \VertWHt{0}{\beta-2}{} \to 0. \] First, note that because $\mathfrak{F}$ is compact, $\mathfrak{f}_n \to \mathfrak{f} \in \mathfrak{F}$. Furthermore, the previous proposition implies 
\begin{equation}\label{eq:linearEstimate_HölderBound}
\Vert \underline{b}_n \VertWHt{2}{\beta}{} \leq 2c.
\end{equation}
For each $n\in \mathbb{N}$ pick a $y_n \in \hat{Y}_{t_n}$ such that $w_{t_n}(y_n)^{-\beta}\vert \underline{b}_n (y_n)\vert = 1$.

\noindent \textbf{Case 1:} There exists an $\varepsilon>0$ such that $\limsup r_{t_n}(y_n)= \varepsilon$.

For each $t\in (0,T_K)$ we use $t\tau \col \hat{\mathcal{V}}_{\kappa}^t \to \mathcal{V}_{\kappa}$ (as defined prior to \autoref{def:OrbifoldResolution}) to identify $\hat{Y}_t \setminus (t\tau)^{-1}(\cup_{S\in \mathcal{S}} S )$ with $Y_0 \setminus \cup_{S \in \mathcal{S}} S$ and analogously \[\hat{P}_{t\vert \hat{Y}_t \setminus (t\tau)^{-1}(\cup_{S\in \mathcal{S}} S)} \cong {P_0}_{\vert Y_0 \setminus \cup_{S \in \mathcal{S}} S}.\] Note that both identifications are $H$-equivariant. 

We can therefore regard the restrictions ${\underline{b}_n}_{\vert Y_t \setminus (t\tau)^{-1}(\cup_{S\in \mathcal{S}} S)}$ as a sequence on $Y_0\setminus \cup_{S \in \mathcal{S}} S$. Pick an exhaustion \[ \overline{Y_0\setminus U} \coloneqq K_1 \subset K_2 \subset K_3 \subset \dots \subset Y_0 \setminus \cup_{S \in \mathcal{S}} S \] by $H$-invariant compact subsets. Since on every $K_i$ the metric $\tilde{g}_t$ converges in $C^\infty$ to $g_0$ (by the construction of $\tilde{\phi}_t$), \eqref{eq:linearEstimate_HölderBound} implies that $\Vert {\underline{b}_n} \VertH{2}{(K_i)}$ is uniformly bounded. The Arzelà--Ascoli Theorem gives therefore rise to a converging subsequence on $K_i$. Using a diagonal sequence, we obtain an $H$-invariant $\underline{b} \in \Omega^{6}(Y_0 \setminus \cup_{S\in \mathcal{S}} S,\mathfrak{g}_{P_0})\oplus \Omega^{0}(Y_0 \setminus \cup_{S\in \mathcal{S}} S,\mathfrak{g}_{P_0})$ such that $\underline{b}_n \to \underline{b}$ in $C^{2}_{\textup{loc}}$. This limit satisfies 
\begin{equation}\label{eq:linearEstimate_GrowthControl}
\vert \underline{b} \vert \leq c \cdot d(\cdot,S)^{\beta} \quad \textup{and} \quad \vert \nabla \underline{b} \vert \leq c \cdot d(\cdot,S)^{\beta-1}
\end{equation}
and since $(\tilde{A}_{t_n}+\tilde{a}_{t_n,\mathfrak{f}_n})\vert_{K_i} \to ({A_0}+a_{0,\mathfrak{f}})\vert_{K_i}$, \[L_{A_0+a_{0,\mathfrak{f}}}L_{A_0+a_{0,\mathfrak{f}}}^* \underline{b} = 0.\] The assumption $\beta >-4$ and~\eqref{eq:linearEstimate_GrowthControl} imply that $L_{A_0+a_{0,\mathfrak{f}}}L_{A_0+a_{0,\mathfrak{f}}}^* \underline{b} =0$ in the sense of distributions over the entire $Y_0$. By elliptic regularity $\underline{b}$ is smooth. Furthermore, \autoref{ass:invertible_linearisations} and integration by parts implies that the only $H$-invariant element in the kernel of $L_{A_0+a_{0,\mathfrak{f}}}L_{A_0+a_{0,\mathfrak{f}}}^*$ is the trivial section, therefore $\underline{b}=0$. However, since $\limsup r_{t_n}(y_n) = \varepsilon$, there exists a converging subsequence $y_n \to y \in Y_0\setminus \cup_{S\in \mathcal{S}} S$ with $\textup{d}(x,S)=\varepsilon$. At this point \[ \vert\underline{b}\vert(x) = \varepsilon^{\beta}>0\] which is a contradiction.

\noindent \textbf{Case 2:} $\limsup r_{t_n}(y_n)=0$.

We can assume that $(y_n)_{n\in \mathbb{N}}$ is only contained inside one $\hat{\mathcal{V}}^{t}_{S,\kappa}$. We write each $y_n$ as $[s_n,z_n]\in (\mathbb{R}\times \hat{Z}_S)/\Lambda_S$ and distinguish the following two scenarios:

\noindent \textbf{Case 2a:} $\liminf r(z_n) < \infty$.

First, pick a subsequence $(y_n)_{n\in \mathbb{N}}$ such that $z_n \to z \in \hat{Z}_S$. For each $n\in \mathbb{N}$ we define $\tilde{\underline{b}}_n \coloneqq \mathrm{s}_n(\underline{b}_n)$ where $\mathrm{s}_n \coloneqq \mathrm{s}_{\beta,t_n,s_n}$ is the rescaling map defined in~\eqref{eq:rescaling_ALE}. By~\eqref{eq:linearEstimate_HölderBound} and \autoref{lem:ALE-rescaling-comparison} these satisfy \[  \Vert \tilde{\underline{b}}_n \VertWH{2}{\beta}{(\mathbb{R}\times (t\tau_S)^{-1}(B_\kappa(0)/\Gamma_S))} \leq 2c^2  \quad \textup{and} \quad \vert w^{-\beta} \tilde{\underline{b}}_n \vert (0,z_n)  \geq c^{-1}.\] As in the first case we find an element $\underline{b} \in C^{2}(\Lambda^6T^*(\mathbb{R} \times \hat{Z}_S)\otimes \mathfrak{g}_{\hat{P}_S} \oplus \mathfrak{g}_{\hat{P}_S})$ with $\Vert \underline{b}\VertWH{1}{\beta}{}< \infty$ such that $\tilde{\underline{b}}_n \to \underline{b}$ in $C^{2}_{\textup{loc}}$. Because $t_n \to 0$, \autoref{prop:ALE-operator-comparison-family} implies that $L_{\hat{A}_S}L_{\hat{A}_S}^* \underline{b} = 0$ and by \autoref{prop:instanton_deformations_come_from_HYM-deformations}, $\underline{b}$ is pulled back from an element in $*_{\hat{\Omega}_S}\mathcal{H}^1_{\hat{A}_S}$. 

By \autoref{ass:invertible_linearisations}, either $*_{\hat{\Omega}_S} \mathcal{H}^1_{\hat{A}_S}= 0$, or $H_S$ acts on $\mathbb{R}\times \hat{Z}$ via $\textup{Id} \times \hat{\vartheta}_S$ and the maps $\delta_{t,s_n} \col \mathbb{R}\times \hat{Z} \to \hat{U}^t_S$ in the definition of the rescaling map~\eqref{eq:rescaling_ALE} are all $H_S$-equivariant. Each $\tilde{\underline{b}}_n$ is therefore $H_S$-invariant and the same holds therefore true for $\underline{b}$. In either case, \autoref{ass:invertible_linearisations} implies that $\underline{b}=0$ which again contradicts $\vert \underline{b}\vert (0,z) \geq c^{-1}w^\beta(z) >0$.

\noindent \textbf{Case 2b:} $\liminf r(z_n) = \infty$.

We define $\tilde{\underline{b}}_n$ as in the previous case and rescale a second time: \[\tilde{\underline{b}}^\prime_n \coloneqq \tilde{\mathrm{s}}_{\beta,\sqrt{r_n},r_n}(\tilde{\underline{b}}_n) \] where $r_n \coloneqq r(z_n)$ and $\tilde{\mathrm{s}}_{\beta,\sqrt{r_n},r_n}$ is the map defined in \autoref{def:scaling_on_C^3}. By \autoref{lem:C^3-rescaling-comparison}, this sequence satisfies \[  \Vert \tilde{\underline{b}}^\prime_n \VertWH{2}{\beta}{} \leq 2c^2  \quad \textup{and} \quad \vert \tilde{\underline{b}}_n^\prime \vert (0, z^\prime_n)  \geq c^{-1}\] where $z^\prime_n \coloneqq \frac{\tau(z_n)}{\vert \tau(z_n) \vert}$. As before, we obtain a subsequence $(\tilde{\underline{b}}^\prime_n)_{n\in \mathbb{N}}$ and $\underline{b} \in C^{2}(\Lambda^6T^*(\mathbb{R}\times (\mathbb{C}^3\setminus \{0\})/\Gamma_S)\otimes \mathfrak{g}_{P_\infty}\oplus\mathfrak{g}_{P_\infty})$ with $\Vert \underline{b} \VertWH{1}{\beta}{}< \infty$ such that $\tilde{\underline{b}}^\prime_n \to \underline{b}$ in $C^2_\textup{loc}$. Since $r_nt_n = r_{t_n}(y_n) \to 0$, \autoref{lem:C^3-rescaling-comparison} implies that this limit satisfies $L_{A_\infty}L_{A_\infty}^* \underline{b} =0$ (strongly) over $\mathbb{R}\times (\mathbb{C}^3\setminus \{0\})/\Gamma_S$. Because of the assumption $\beta>-4$ we have, in fact, $L_{A_\infty}L_{A_\infty}^* \underline{b} =0$ over the entire $\mathbb{R}\times \mathbb{C}^3/\Gamma_S$ in the sense of distributions. Elliptic regularity implies that $\underline{b}$ is smooth and \autoref{prop:instanton_deformations_vanish_over_euklidean_space} that it vanishes everywhere. However, the sequence $z_n^\prime$ has an accumulation point $z$ over which $\vert \underline{b} \vert (0,z) \geq c^{-1} >0$ which is a contradiction.
\end{proof}

The following establishes \autoref{bul:infinitesimally_irreducible} of \autoref{prop:infinitesimal_irreducibility}:
\begin{proposition}\label{prop:linearEstimate-for-rigidity}
Assume that we are in the setup of \autoref{prop:pregluing-family} and assume that $\ker(\diff_{A_{0}+a_{0,\mathfrak{f}}}\col \Omega^0(Y_0,\mathfrak{g}_{P_0}) \to \Omega^1(Y_0,\mathfrak{g}_{P_0})) = 0$ for every $\mathfrak{f}\in \mathfrak{F}$. Then there exists for $\beta\in(-4,0)$ and $\alpha\in(0,1)$ two constants $0<T_{iL}<T_K$ and $c_{iL}>0$ such that for $t\in(0,T_{iL})$ and $\mathfrak{f} \in \mathfrak{F}$ any $\xi \in \Omega^0(\hat{Y}_t,\mathfrak{g}_{\hat{P}_t})$ satisfies \[\Vert \xi \VertWHt{2}{\beta}{} \leq c_{iL} \Vert \diff_{\tilde{A}_t+\tilde{a}_{t,\mathfrak{f}}}^*\diff_{\tilde{A}_t+\tilde{a}_{t,\mathfrak{f}}} \xi \VertWHt{0}{\beta-2}{}.\]
\end{proposition}

\begin{proof}
This can be proven directly as \autoref{prop:LinearEstimate}. Somewhat alternatively, note that $L_{\tilde{A}_t+\tilde{a}_{t,\mathfrak{f}}}L_{\tilde{A}_t+\tilde{a}_{t,\mathfrak{f}}}^*$ acts on $(0,\xi)\in\Omega^6(\hat{Y}_t,\mathfrak{g}_{\hat{P}_t})\oplus \Omega^0(\hat{Y}_t,\mathfrak{g}_{\hat{P}_t})$ as $(\psi_t \wedge [F_{\tilde{A}_t+\tilde{a}_{t,\mathfrak{f}}},\xi],\diff_{\tilde{A}_t+\tilde{a}_{t,\mathfrak{f}}}^*\diff_{\tilde{A}_t+\tilde{a}_{t,\mathfrak{f}}} \xi)$. Going through the proof of \autoref{prop:LinearEstimate} using only (not necessarily $H$-invariant) $\underline{b}_n \in \Omega^6(\hat{Y}_t,\mathfrak{g}_{\hat{P}_t})\oplus \Omega^0(\hat{Y}_t,\mathfrak{g}_{\hat{P}_t})$ of the form $(0,\xi_n)$ and the new condition, one can check that there exist $T,c>0$ such that for every $\underline{b}=(0,\xi)$, $\mathfrak{f} \in \mathfrak{F}$ and $t\in(0,T)$ \begin{align*}
\Vert \xi \VertWHt{2}{\beta}{}&=\Vert \underline{b} \VertWHt{2}{\beta}{} \leq c \Vert L_{\tilde{A}_t+\tilde{a}_{t,\mathfrak{f}}}L_{\tilde{A}_t+\tilde{a}_{t,\mathfrak{f}}}^*\underline{b} \VertWHt{0}{\beta-2}{} \\
& \leq c \big(\Vert \diff^*_{\tilde{A}_t+\tilde{a}_{t,\mathfrak{f}}}\diff_{\tilde{A}_t+\tilde{a}_{t,\mathfrak{f}}}\xi \VertWHt{0}{\beta-2}{} + \Vert F_{\tilde{A}_t+\tilde{a}_{t,\mathfrak{f}}}\wedge\psi_t \VertWHt{0}{-2}{}\Vert \xi \VertWHt{0}{\beta}{}\big). 
\end{align*} 
Since $\Vert F_{\tilde{A}_t+\tilde{a}_{t,\mathfrak{f}}}\wedge\psi_t \VertWHt{0}{-2}{} \leq c t^{1/2}$ by \autoref{prop:pregluing-family}, the second term can be absorbed to the left-hand side once $t$ is sufficiently small.
\end{proof}

\section{Examples}\label{sec:Example18}

This sections finds examples of $\G_2$-instantons over a generalised Kummer construction resolving the $\G_2$-orbifold discovered in~\cite[Example~18]{Joyce-GeneralisedKummer2}. For this, consider the orbifold $Y_0 \coloneqq \mathbb{R}^7/\Gamma$ where $\Gamma< \G_2 \ltimes \mathbb{R}^7$ is the crystallographic group generated by 
\begin{align*}
\alpha(y_1,\dots,y_7) &= (y_2,y_3,y_7,-y_6,-y_4,y_1,y_5) \\
\beta(y_1,\dots,y_7) &= (\tfrac{1}{2} - y_1, \tfrac{1}{2} - y_2, -y_3, -y_4, \tfrac{1}{2} + y_5, \tfrac{1}{2} + y_6, y_7)\\
\tau_1 (y_1,\dots,y_7) &= (y_1+1,y_2,y_3,y_4,y_5,y_6,y_7).
\end{align*}
It is not difficult to verify that the flat $\G_2$-structure $\phi_0\in \Lambda^3(\mathbb{R}^7)^*$ (associated to $(\mathbb{C}^3,\omega_0,\Omega_0)$ via \autoref{ex_G2Model}) is preserved by $\Gamma$ and induces therefore a $\G_2$-structure on $Y_0$.

\begin{proposition}\label{prop:presentation__of_piorb}
$\Gamma$ is isomorphic to the finitely presented group $\langle\, a,b,t \,\vert\, R\, \rangle$ with relations $R$ listed below. This isomorphism is defined by mapping $a,b,t$ to $\alpha,\beta,\tau_1$, respectively. In order to state the relations $R$, we define $b_i \coloneqq a^iba^{-i}$ and $t_i \coloneqq a^i t a^{-i}$ (note that $t_i$ does in general not correspond to the map on $\mathbb{R}^7$ which translates by $e_i$). Then $R$ is given by:
\begin{enumerate}
\item\label{rel:ord7} $a^7=1$
\item\label{rel:b2t}$b^2 = t_1t_3$
\item\label{rel:[b,bi]} $[b,b_{1}] = t_6t_2^{-1}t_3t_1$, $[b,b_{2}] = tt_5^{-1}t_2^{-1}t_1$, $[b,b_{3}] = t_6^{-1}t_2^{-1}t_3t_1$, $[b,b_{4}] = t^{-1}t_6t_5^{-1}t_3$, $[b,b_{5}] = tt_6^{-1}t_5^{-1}t_3$, $[b,b_{6}] = t^{-1}t_5^{-1}t_2^{-1}t_1$
\item\label{rel:b0b1b3} $b_0b_1=t_1t_2^{-1} b_3$
\item\label{rel:[b,ti]} $[b,t]=t^{-2}$, $[b,t_1]=1$, $[b,t_2]=t_2^{-2}$, $[b,t_3]=1$, $[b,t_4]=1$, $[b,t_5]=t_5^{-2}$, $[b,t_6]=t_6^{-2}$
\item\label{rel:[ti,tj]} $[t_i,t_j]=1$ for all $i,j$
\end{enumerate}
\end{proposition}
\begin{proof}
First, we define a map on the free group $F \col \langle \, a,b,t \, \rangle \to \Gamma$ by mapping $a,b,t$ to $\alpha, \beta$, $\tau_1$, respectively. Since $\alpha$, $\beta$ and $\tau_1$ generate $\Gamma$, this map is surjective. The relations are chosen such that $F$ descends to the quotient $\langle \, a,b,t \,\vert\, R\, \rangle$ and we only have to prove that this map is also injective. Due to the relations, each element in $\langle\, a,b,t \,\vert\, R\, \rangle$ is equivalent to one of the form \[t_0^{i^t_0}t_1^{i^t_1}t_2^{i^t_2}t_3^{i^t_3}t_4^{i^t_4}t_5^{i^t_5}t_6^{i^t_6}b_0^{i^b_0}b_1^{i^b_1}b_2^{i^b_2}a^{i^a} \] with $i^t_0,\dots,i^t_6 \in \mathbb{Z}$, $i^b_0,i^b_1,i^b_2 \in \{0,1\},$ and $i^a \in \{0,\dots,6\}$. This is also in accordance with~\cite[Example~18]{Joyce-GeneralisedKummer2} which states that the quotient of $\Gamma$ by the lattice $\mathbb{Z}^7<\Gamma$ is of order 56. Direct inspection of the maps involved in the definition of $\Gamma$ reveals that such an element is only mapped to $1$ under $F$ if and only if all exponents equal zero. This is again in accordance with the quotient of $\Gamma$ by $\mathbb{Z}^7$ being of order 56 as stated in~\cite[Example~18]{Joyce-GeneralisedKummer2}.
\end{proof}

\subsection{A family of flat bundles}

\begin{proposition}\label{prop:example18-definition-monodromy-representation}
For each $\theta \in S^1 \subset \mathbb{C}$, define the representation \[f_\theta \col \langle\, a,b,t \, \rangle \to \SO(14)\cong \SO_{\mathbb{R}}(\mathbb{C}^7)\] by 
\begin{align*}
f_\theta(a) \col (z_1,\dots,z_7) &\mapsto (z_2,z_3,z_7,\overline{z}_6,\overline{z}_4,z_1,z_5) \\
f_\theta(b) \col (z_1,\dots,z_7) &\mapsto (\theta \overline{z}_1,\theta \overline{z}_2, \overline{z}_3, \overline{z}_4, \theta z_5, \theta z_6, z_7) \\
f_\theta(t) \col (z_1,\dots,z_7) &\mapsto (\theta^2 z_1,z_2, \dots, z_7).
\end{align*}
Then $f_\theta$ descends to $\Gamma \cong \langle a,b,t \, \vert \, R \rangle$.
\end{proposition}
\begin{proof}
The proof reduces to checking that $f_\theta(a)$, $f_\theta(b)$, and $f_\theta(t)$ satisfy the relations in the previously given presentation of $\Gamma$. Relations~\ref{rel:ord7},~\ref{rel:b2t},~\ref{rel:b0b1b3}, and~\ref{rel:[ti,tj]} are easily verified. Here is a heuristic as to why the Relations~\ref{rel:[b,bi]} and~\ref{rel:[b,ti]} hold: Note that $f_\theta(b)$ and $f_\theta(t)$ originate from $\beta$ and $\tau_1$ as follows: A reflection of a coordinate is replaced by complex conjugation of the same coordinate and translation by $\frac{k}{2}$ is replaced by multiplication with $\theta^k$. The relations given in~\ref{rel:[b,bi]} and~\ref{rel:[b,ti]} all originate from commuting reflections with translations. Commuting complex conjugation and multiplication by an element in $S^1$ behaves (multiplicatively) analogous. 
\end{proof}

\begin{proposition}\label{prop:example18-monodromies-not-conjugated}
If $\theta_1,\theta_2 \in S^1$ are distinct, then there is no $R \in \SO_\mathbb{R}(\mathbb{C}^7)$ such that $R f_{\theta_1}(\gamma) R^{-1} = f_{\theta_2}(\gamma)$ for all $\gamma\in \Gamma$. Moreover, if $\theta \in S^1 \setminus \{\pm 1 \} $, then the only elements $R\in \SO_\mathbb{R}(\mathbb{C}^7)$ with $Rf_{\theta}(\gamma) R^{-1} =f_{\theta} (\gamma)$ for every $\gamma \in \Gamma$ are $R = \pm \id$.
\end{proposition}
\begin{proof}
Note that under the inclusion $\SO_\mathbb{R}(\mathbb{C}^7) \subset \U(\mathbb{C}^7\otimes_\mathbb{R}\mathbb{C})$, all $f_{\theta_i}(\gamma)$ for $\gamma\in \Gamma$ are diagonalisable. Furthermore, if $f_{\theta_2}(\gamma) = R f_{\theta_1}(\gamma) R^{-1}$ for some $R\in \SO_\mathbb{R}(\mathbb{C}^7)$, then all eigenvalues and the dimensions of their respective eigenspaces must agree. One can see that the eigenvalues of $f_{\theta_i}(b)$ are $1,-1,\theta_i$ and $\overline{\theta_i}$ and that the complex dimension of their respective eigenspaces are 6, 4, 2, 2. Thus, we can only have $f_{\theta_2}(b) = R f_{\theta_1}(b) R^{-1}$ if $\theta_2 = \theta_1$ or $\theta_2 = \overline{\theta_1}$. This proves the claim for $\theta_1 \in \{\pm1\}$.

We therefore assume $\theta_1, \theta_1 \notin \{\pm 1\}$. Direct inspection of the involved (block-) matrices reveals that we can only have $Rf_{\theta_1}(t) = f_{\theta_2}(t)R$ if 
\begin{align*}
R = \begin{pmatrix}
A_1 & 0 \\
0 & \tilde{A}
\end{pmatrix}
\end{align*}
where $A_1 \in \O_{\mathbb{R}}(\mathbb{C})$ and $\tilde{A}\in \O_{\mathbb{R}}(\mathbb{C}^6)$. Similarly, $R$ can only satisfy $R f_{\theta_1}(a^ita^{-i}) = f_{\theta_2}(a^ita^{-i}) R$ for all $i=0,\dots,6$ if it is of the form $R = \textup{diag}(A_1,\dots,A_7)$ with $A_i \in \O_{\mathbb{R}}(\mathbb{C})$.

The condition $R f_{\theta_1}(a) R^{-1} = f_{
\theta_2}(a)$ implies that all $A_i$ have the same determinant. Furthermore, recall from the first step of the proof that either $\theta_2=\theta_1$ or $\theta_2= \overline{\theta_1}$. If $\theta_2= \overline{\theta_1}$, then the sixth equation of $R f_{\theta_1}(b) R^{-1}  =f_{\overline{\theta_1}}(b)$ implies $\det(A_7)=-1$. However, this would lead to $\det(R) = \det(A_7)^7=-1$ which contradicts the fact that $R$ lies in $\SO_\mathbb{R}(\mathbb{C}^7)$. This shows that $f_{\theta_1}$ and $f_{\theta_2}$ are for $\theta_1 \neq \theta_2$ not conjugated in $\SO_\mathbb{R}(\mathbb{C}^7)$.

In order to prove the second claim we note that for $\theta \notin\{\pm 1\}$, $Rf_{\theta}(a^{i}ta^{-i}) R^{-1} =f_\theta (a^{i}ta^{-i})$ for every $i=0,\dots,6$ again implies that $R= \diag(A_1,\dots,A_7)$ with $A_i \in \O_{\mathbb{R}}(\mathbb{C}^7)$. The third equation in $R f_\theta(b) R^{-1} =f_{\theta}(b)$ implies that $A_3$ is diagonal. Finally, $Rf_\theta(a) R^{-1} =f_{\theta}(a)$ implies that either all $A_i = \id$ or all $A_i = -\id$.
\end{proof}

\begin{remark}
We will see later that $f_{\theta}$ and $f_{\overline{\theta}}$ are conjugated in $\O_\mathbb{R}(\mathbb{C}^7)$ (cf. \autoref{rem:example18_bundles_are_isomorphic}). Again, since the eigenvalues of any two conjugated $f_{\theta_1}(b)$ and $f_{\theta_2}(b)$ must agree, these are the only two possibilities of conjugated homomorphisms.
\end{remark}

\begin{lemma}\label{lem:Example18_fixpoint_set_flambda}
The representation $f_\theta \col \Gamma \to \SO_\mathbb{R}(\mathbb{C}^7)$ given in the previous proposition induces an action on $\mathfrak{so}_\mathbb{R}(\mathbb{C}^7)$ (via the adjoint representation $\Ad \circ f_\theta$) and on $\mathbb{R}^7 \otimes \mathfrak{so}_\mathbb{R}(\mathbb{C}^7)$ (via $\pr_1 \otimes (\Ad \circ f_\theta)$ where $\pr_1 \col \Gamma < \G_2\ltimes\mathbb{R}^7 \to \G_2$ is the projection to the first entry). For $\theta \notin \{1,-1\}$, the fixpoint-set of these actions are $\mathfrak{so}_\mathbb{R}(\mathbb{C}^7)^{\Gamma} = 0 $ and \[(\mathbb{R}^7 \otimes \mathfrak{so}_\mathbb{R}(\mathbb{C}^7))^{\Gamma} = \left\langle e_1 \otimes e_1 \wedge f_1 + \dots + e_7 \otimes e_7 \wedge f_7 \right\rangle_\mathbb{R}\] where we denote by $e_1,\dots,e_7$ and $f_1,\dots,f_7$ the canonical basis of $\mathbb{R}^7$ and $i\mathbb{R}^7$, respectively, and write $\mathbb{C}^7 = \mathbb{R}^7+i\mathbb{R}^7$.
\end{lemma}
\begin{proof}
It is not difficult to verify that once $\theta \notin \{1,-1\}$, the only elements invariant under all $f_\theta(t_i)$ for $i=0,\dots,6$ are \[\left\langle e_1\wedge f_1, \dots, e_7\wedge f_7 \right\rangle_\mathbb{R} \subset \mathfrak{so}_\mathbb{R}(\mathbb{C}^7) \] and \[\mathbb{R}^7 \otimes \left\langle e_1\wedge f_1, \dots, e_7\wedge f_7 \right\rangle_\mathbb{R} \subset \mathbb{R}^7\otimes \mathfrak{so}_\mathbb{R}(\mathbb{C}^7).\]
Of these elements, none in $\mathfrak{so}_\mathbb{R}(\mathbb{C}^7)$ is additionally invariant under all $f_\theta(b_i)$ and the only ones in $\mathbb{R}^7 \otimes \mathfrak{so}_\mathbb{R}(\mathbb{C}^7)$ are \[ \left\langle e_1 \otimes e_1 \wedge f_1, \dots, e_7 \otimes e_7 \wedge f_7 \right\rangle_\mathbb{R}. \] Ultimately, $f_\theta(a)$ reduces this further to \[ (\mathbb{R}^7 \otimes \mathfrak{so}_\mathbb{R}(\mathbb{C}^7))^{\Gamma} = \left\langle e_1 \otimes e_1 \wedge f_1 + \dots + e_7 \otimes e_7 \wedge f_7 \right\rangle_\mathbb{R}. \qedhere \]
\end{proof}

\begin{proposition}\label{prop:example18_no_fixpoints_quantitative_version}
Denote by $a,b,t$ the three generators of $\Gamma$ as in \autoref{prop:presentation__of_piorb}. For every $\theta \in S^1 \setminus \{\pm 1\}$ there exists a $c_\theta >0$ depending continuously on $\theta$, such that for every $\xi \in \mathfrak{so}_{\mathbb{R}}(\mathbb{C}^7)$ there exists a $\gamma \in \{a,b,t \} \subset \Gamma$ such that \[ \vert \xi \vert \leq c_\theta \vert (\id - \Ad_{f_\theta(\gamma)}) (\xi) \vert. \]
\end{proposition}
\begin{proof}
The previous proposition implies that the map 
\begin{align*}
\mathfrak{so}_{\mathbb{R}}(\mathbb{C}^7) &\to (\mathfrak{so}_{\mathbb{R}}(\mathbb{C}^7))^3 \\
\xi &\mapsto ((\id -\Ad_{f_{\theta}(a)}) (\xi),(\id -\Ad_{f_{\theta}(b)}) (\xi),(\id -\Ad_{f_{\theta}(t)}) (\xi)))
\end{align*}
is injective. Since $\mathfrak{so}_{\mathbb{R}}(\mathbb{C}^7)$ is finite dimensional, this implies the claim.
\end{proof}

\begin{definition}\label{def:example18_flat_bundles}
For each $\theta \in S^1$, let $E_{0,\theta} \coloneqq \mathbb{R}^7 \times_{\Gamma, f_\theta} \mathbb{C}^7$ be the (flat) oriented Euclidean vector bundle associated to $f_\theta$. Furthermore, let $\SO(E_{0,\theta})$ and $\O(E_{0,\theta})$ be the principal bundles of positively oriented and non-oriented frames, respectively.
\end{definition}

\begin{remark}\label{rem:example18_bundles_are_isomorphic}
The bundles $E_{0,\theta}$ for $\theta\in S^1$ are all isometric (because they lie in a smooth 1-parameter family of bundles). However, \autoref{prop:example18-monodromies-not-conjugated} proves that such an isomorphism cannot identify the flat connections on $\SO(E_{\theta_1,0})$ and $\SO(E_{\theta_2,0})$ whenever $\theta_1 \neq \theta_2$ (cf.~\cite[Proposition~2.2.3]{DonaldsonKronheimer}). We therefore obtain an $S^1$-family of flat $\SO(14)$-bundles. 

The situation is slightly different when going to the structure group $\O(14)$: Here, the isomorphism 
\begin{align*}
E_{0,\theta} &\to E_{\bar{\theta},0} \\
[y,w]&\mapsto[y,R_0w]
\end{align*}
where $R_0 \in \O_\mathbb{R}(\mathbb{C}^7)$ maps $(z_1,\dots,z_7) \mapsto (\overline{z}_1,\dots,\overline{z}_7)$ identifies the flat connections on $\O(E_{0,\theta})$ and $\O(E_{\bar{\theta},0})$.
\end{remark}

\begin{proposition}\label{prop:Example18_Z2-action}
There exists a $\mathbb{Z}_2$ action $\lambda_0\col \mathbb{Z}_2 \to \textup{Isom}(Y_0,g_0)$ which is induced by the map $\mathbb{R}^7 \to \mathbb{R}^7, y \mapsto -y$. This map is orientation reversing and preserves $\psi_0$. This action lifts to $E_{0,\theta}$ via $\tilde{\lambda}_0(-1) [y,w] \coloneqq [-y, R_0w]$ where $R_0 \in \O_{\mathbb{R}}(\mathbb{C}^7)$ is given by $R_0(z_1,\dots,z_7) \coloneqq (\overline{z}_1,\dots,\overline{z}_7)$. This action is isometric, orientation reversing and fixes the flat connection on $E_{0,\theta}$. We therefore obtain a connection preserving action (also denoted by $\tilde{\lambda}_0$) on $\O(E_{0,\theta})$.
\end{proposition}
\begin{proof}
First note, that $(-1,0) \in N_{\O(7)\ltimes \mathbb{R}^7}(\Gamma)$ because \[(-1,0)\tau_1 (-1,0) = \tau_1^{-1} \quad (-1,0) \alpha (-1,0) = \alpha \quad (-1,0) \beta (-1,0) = \tau_1^{-1}\tau_2^{-1}\tau_5^{-1}\tau_6^{-1} \beta\] (where $\tau_i$ for $i=1,\dots,7$ denote the map that translates be $e_i$). This induces the action on $Y_0$. Next, we note that $R_0, f_\theta(\alpha), f_\theta(\beta),$ and $f_\theta(\tau_1)$ satisfy the analogous relations \[R_0 f_\theta(\tau_1) R_0 = f_\theta(\tau_1^{-1}) \quad R_0 f_\theta(\alpha) R_0 = f_\theta(\alpha) \quad R_0 f_\theta(\beta) R_0 = f_\theta(\tau_1^{-1}\tau_2^{-1}\tau_5^{-1}\tau_6^{-1} \beta) \] which implies that the $\mathbb{Z}_2$ action lifts to $E_{0,\theta}$. The rest of the statement is easily verified.
\end{proof}

\begin{proposition}\label{prop:example18_ker(L_A)}
Let $\theta\in S^1$ and denote by $A_{0,\theta}$ the canonical (flat) $\G_2$-instanton on $\O(E_{0,\theta})$. If $\theta \notin\{1,-1\}$, then $A_{0,\theta}$ is infinitesimally irreducible and $\ker L_{A_{0,\theta}}^*$ is 1-dimensional and consists of elements that are anti-invariant under the $\mathbb{Z}_2$-action described in \autoref{prop:Example18_Z2-action}. 
\end{proposition}
\begin{remark}\label{rem:example18-explanation O(14)-bundle I}
In the following we work with the bundle $\O(E_{0,\theta})$ because the $\mathbb{Z}_2$-action is orientation reversing. The $\G_2$-instantons that we will construct using \autoref{theo:perturbing_almost_instantons}, will therefore a-priori also lie on (a glued) $\O_{\mathbb{R}}(\mathbb{C}^7)$-bundle $\O(\hat{E}_{\theta})$. However, this $\O_{\mathbb{R}}(\mathbb{C}^7)$-bundle will admit an $\SO_{\mathbb{R}}(\mathbb{C}^7)$-reduction $\SO(E_{0,\theta})\subset \O(E_{0,\theta})$ and since $\mathfrak{o}_{\mathbb{R}}(\mathbb{C}^7)=\mathfrak{so}_{\mathbb{R}}(\mathbb{C}^7)$, these $\G_2$-instantons will restrict to $\SO(E_{0,\theta})$. Note that although $\mathbb{Z}_2$ does not act on $\SO(E_{0,\theta})$, it still acts on the adjoint bundle $\mathfrak{so}(E_{0,\theta})$. In the light of \autoref{bul: perturbing for more general H-actions} of \autoref{rem:perturbing almost instantons weakening assumptions}, one could therefore deform the $\SO_{\mathbb{R}}(\mathbb{C}^7)$-connection directly, without the diversion to $\O(E_{0,\theta})$.
\end{remark}
\begin{proof}
Since $Y_0$ and $A_{0,\theta}$ are flat, the Weitzenböck-formula implies that any element in the kernel of $L_{A_{0,\theta}}$ is parallel and therefore corresponds to a fixed-point of the representation of $\Gamma$ on $(\mathbb{R}^7\otimes \mathfrak{so}_\mathbb{R}(\mathbb{C}^7) \oplus \Lambda^7\mathbb{R}^7\otimes \mathfrak{so}_\mathbb{R}(\mathbb{C}^7))$ described in \autoref{lem:Example18_fixpoint_set_flambda} (cf.~\cite[Proposition~9.2]{Walpuski-InstantonsKummer}). Furthermore, the connection is infinitesimally irreducible if and only if the representation of $\Gamma$ on $\mathfrak{o}_\mathbb{R}(\mathbb{C}^7) = \mathfrak{so}_\mathbb{R}(\mathbb{C}^7)$ only fixes the center $\mathfrak{z} = 0$ of $\mathfrak{so}_\mathbb{R}(\mathbb{C}^7)$. \autoref{lem:Example18_fixpoint_set_flambda} implies therefore, that for $\theta\notin \{1,-1\}$ $A_{0,\theta}$ is infinitesimally irreducible and, furthermore, $\ker(L_A) = \mathbb{R}a$ where $0 \neq a \in \Omega^1(Y_0,\mathfrak{so}({E_{0,\theta}}))$. The explicit form of $a$ given in~\autoref{lem:Example18_fixpoint_set_flambda} (or the fact that the entire $1$-parameter family of connections  determined by $f_\theta$ for $\theta\in S^1$ is $\mathbb{Z}_2$-invariant by \autoref{prop:Example18_Z2-action}) implies that $a$ is invariant under the $\mathbb{Z}_2$-action. Since $L_{A_{0,\theta}}^* = * L_{A_{0,\theta}} *$ we have $\ker(L_{A_{0,\theta}}^*)=\mathbb{R}(*a)$ and since the $\mathbb{Z}_2$-action is orientation reversing, we obtain $\tilde{\lambda}_0(-1)^* (*a) = -a$.
\end{proof}

\begin{proposition}\label{prop:example18_Poincare-type-estimate}
For every $\theta\in S^1 \setminus \{\pm 1\}$ there exists a constant $c_\theta>0$, which depends continuously on $\theta$, with the following property: Let $U\subset Y_0$ be an open set such that its preimage $\check{U} \subset \mathbb{R}^7$ under the quotient map $\mathbb{R} \to \mathbb{R}^7/\Gamma=Y_0$ is connected and for any $y_1,y_2 \in \check{U}\subset \mathbb{R}^7$ we have $\textup{dist}_{\check{U}}(y_1,y_2)< 10 (\vert y_1-y_2 \vert + 1)$ (where $\textup{dist}_{\check{U}}(y_1,y_2)$ measures the distance between $y_1$ and $y_2$ with respect to paths contained in $\check{U}$). Then every $\xi \in \Omega^0(U,\mathfrak{so}(E_{0,\theta}))$ satisfies \[ \Vert \xi \VertC{0}{(U)} \leq c_\theta \Vert \diff_{A_{0,\theta}} \xi \VertC{0}{(U)}. \]
\end{proposition}
\begin{proof}
As in the statement we denote by $\check{U}\subset \mathbb{R}^7$ the preimage of $U$ under the quotient map $\mathbb{R}^7 \to \mathbb{R}^7/\Gamma =Y_0$. Furthermore, let $\check{\xi}\in \Omega^0(\check{U},\mathbb{R}^7\times \mathfrak{so}_{\mathbb{R}}(\mathbb{C}^7))^{\Gamma,f_\theta}$ be the $\Gamma$-equivariant lift of $\xi$. We will prove in the following that \[ \Vert \check{\xi} \VertC{0}{(\check{U})} \leq c_\theta \Vert \diff \check{\xi} \VertC{0}{(\check{U})} \] for a constant $c_\theta>0$ that depends continuously on $\theta$. Since $A_{0,\theta}$ is induced by the trivial connection on $\mathbb{R}^7 \times \mathfrak{so}_{\mathbb{R}}(\mathbb{C}^7)$ and the quotient map $\mathbb{R}^7\to Y_0$ is isometric, this implies the statement.

For this let $\alpha,\beta,\tau_1 \in \Gamma$ be the set of generators described in the beginning of this section. For every $y\in \check{U}$ there exists by \autoref{prop:example18_no_fixpoints_quantitative_version} a $\gamma \in \{\alpha,\beta,\tau_1\}$ such that \[ \big\vert \check{\xi}(y) \big\vert \leq c_\theta \big\vert(\id -\Ad_{f_\theta(\gamma)})(\check{\xi}(y)) \big\vert = c_\theta \big\vert \check{\xi}(y) - \check{\xi}(\gamma(y)) \big\vert\] where $c_\theta>0$ only depends on $\theta$ (and this dependence is continuous). Our assumptions on $\check{U}$ and direct inspection of the maps $\alpha,\beta,\tau_1$ show that whenever $y \in \check{U} \cap [0,1]^7$, then we can always find a path connecting $y$ and $\gamma(y)$ for $\gamma \in \{\alpha,\beta,\tau_1\}$ of length less than $10 \cdot (2\sqrt{7}+1)$. The Fundamental Theorem of Calculus implies therefore that whenever $y \in \check{U}\cap [0,1]^7$ \[ \vert \check{\xi} \vert(y) \leq c_{\theta} \vert \check{\xi}(y) - \check{\xi}(\gamma(y)) \vert \leq 70 c_{\theta} \Vert \diff \check{\xi} \VertC{0}{(\check{U})}. \] Note that a general $y \in \check{U}$ can always be translated via elements in $\Gamma$ to an element in $\check{U}\cap [0,1]^7$. Since $\check{\xi}$ is $\Gamma$-equivariant and the action of $\Gamma$ on $\mathfrak{so}_{\mathbb{R}}(\mathbb{C}^7)$ preserves the norm, the previous inequality actually holds for any $y \in \check{U}$ and the claim follows.
\end{proof}

\subsection{A family of flat connections over a fixed bundle}\label{subsec:family of flat connections}

Next, we fix the bundle $E_0 \coloneqq E_{0,1}$ and connection $A_0 \coloneqq A_{0,1}$ and express the family of flat bundles $(E_{0,\theta},A_{0,\theta})_{\theta \in S^1}$ discovered in the previous section as a family of flat connections $(A_0 +a_{0,\mathfrak{f}})_{\mathfrak{f}\in \mathbb{R}} $ on the fixed bundle $\O(E_0)$ (and the subbundle $\SO(E_0)$), where $A_0+a_{0,\mathfrak{f}}$, $A_0+a_{0,\mathfrak{f}+2\pi}$, and $A_0+a_{0,-\mathfrak{f}}$ will turn out to be gauge equivalent (and $A_0+a_{0,\mathfrak{f}}$, $A_0+a_{0,\mathfrak{f}+2\pi}$ in the case of $\SO(E_0)$):

\begin{proposition}\label{prop:example18-identifying_flat_bundles-I}
Fix $(E_0,A_0) \coloneqq (E_{0,1}, A_{0,1})$ and let $\theta \coloneqq \e^{i \mathfrak{f}} \in S^1$ for $\mathfrak{f} \in \mathbb{R}$. Define 
\begin{align*}
F_{\mathfrak{f}}^\prime \col \mathbb{R}^7 \times \mathbb{C}^7 &\to \mathbb{R}^7 \times \mathbb{C}^7 \\
(y,w) &\mapsto (y , \textup{diag}(\e^{i y_1 \mathfrak{f}}, \dots, \e^{i y_7 \mathfrak{f}} ) \cdot w)
\end{align*} 
where $y=(y_1,\dots,y_7) \in \mathbb{R}^7$. Then $F_{\mathfrak{f}}^\prime$ descends to an (isometric and orientation preserving) $\mathbb{Z}_2$-equivariant isomorphism \[\overline{F_{\mathfrak{f}}^\prime} \col E_{0}\equiv E_{0,1} \to E_{0,\theta}\] that pulls back the flat connection $A_{0,\theta} \in \mathcal{A}(\SO(E_{0,\theta})) = \mathcal{A}(\O(E_{0,\theta}))$ to \[(\overline{F_{\mathfrak{f}}^\prime})^*A_{0,\theta} - A_{0} = \textup{diag} (i \mathfrak{f} \diff y^1, \dots , i \mathfrak{f} \diff y^7) \in \Omega^1(\mathbb{R}^7, \mathfrak{so}_{\mathbb{R}}(\mathbb{C}^7))^{\Gamma,f_1} \cong \Omega^1(Y_0,\mathfrak{so}(E_{0,1})).\] 
\end{proposition}
\begin{proof}
For the first point, one chooses a set of generators $\{a,b,t\}$ as in \autoref{prop:presentation__of_piorb} and checks directly that $F_{\mathfrak{f}}^\prime \circ (\gamma, f_1(\gamma)) = (\gamma, f_\theta(\gamma)) \circ F_{\mathfrak{f}}^\prime$ for every generator $\gamma \in \{a,b,t \}$. It is not difficult to see that the same equivariance holds true for the generator of the $\mathbb{Z}_2$-action. Furthermore, note that $F_{\mathfrak{f}}^\prime$ pulls back the trivial connection $A_{\textup{trivial}}$ on $\mathbb{R}^7\times \O_{\mathbb{R}}(\mathbb{C}^7)$ to \[ (F_{\mathfrak{f}}^\prime)^* A_{\textup{trivial}} - A_{\textup{trivial}} = \textup{diag}(i\mathfrak{f} \diff y^1 ,\dots , i\mathfrak{f} \diff y^7) \in \Omega^1(\mathbb{R}^7,\mathfrak{so}_{\mathbb{R}}(\mathbb{C}^7)). \] Since 
\begin{equation*}
\begin{tikzcd}
\big(\mathbb{R}^7 \times \O_{\mathbb{R}}(\mathbb{C}^7),A_{\textup{trivial}}\big) \arrow[r, "F_{\mathfrak{f}}^\prime"] \arrow[d] & \big(\mathbb{R}^7 \times \O_{\mathbb{R}}(\mathbb{C}^7),A_{\textup{trivial}}\big) \arrow[d] \\
\big(\O(E_{0,1}),A_{0,1}\big) \arrow[r, "\overline{F_{\mathfrak{f}}^\prime}"] & \big(\O(E_{0,\theta}),A_{0,\theta}\big)
\end{tikzcd}
\end{equation*} commutes, this implies the second statement.
\end{proof}

Next, we will modify $\overline{F_{\mathfrak{f}}^\prime}$ so that the family of connections $((\overline{F_{\mathfrak{f}}^\prime})^*A_{0,\e^{i\mathfrak{f}}})_{\mathfrak{f}\in \mathbb{R}}$ satisfies the assumptions in \autoref{prop:pregluing-family} and \autoref{ass: very simple family of gluing data}. For this we first need the following local description of $Y_0$ near its singular set:

\begin{proposition}[{\cite[Example~18]{Joyce-GeneralisedKummer2}}]\label{prop:example18_description_of_singularset}
The set $\mathcal{S}$ of connected components of $\textup{sing}(Y_0)$ consist of one element $S \cong S^1$ and (in the notation of \autoref{ass:invertible_linearisations}) there exists an open embedding $\mathtt{J}\col (\mathbb{R} \times B_\kappa(0)/\mathbb{Z}_7)/\mathbb{Z} \to Y_0$ with $\mathtt{J}^*\phi_0 = \phi_{S}$. The group $\mathbb{Z}_7$ acts hereby on $B_\kappa(0)\subset \mathbb{C}^3$ via 
\begin{equation}\label{eq:example18_action_of_Z7}
\big(z_1,z_2,z_3\big) \mapsto \big(e^{2\pi i/7}z_1,e^{4\pi i/7}z_2,e^{8\pi i/7}z_3\big) 
\end{equation} 
and $\mathbb{Z}$ only acts on $\mathbb{R}$ via translations by multiples of $\sqrt{7}$.

In fact, $\mathtt{J}$ can be constructed as follows: There exists an orthonormal basis $\{b_1,\dots,b_7\}$ of $\mathbb{R}^7$ with $b_1 = \tfrac{1}{\sqrt{7}} (1,1,1,-1,1,1,1)$ such that (under the identification $\mathbb{R}^7 \cong \mathbb{R}\oplus \mathbb{C}^3$ via $(b_1,b_2+ib_3,b_4+ib_5,b_6+ib_7)$) $\alpha$ attains the form described in \eqref{eq:example18_action_of_Z7} and $\phi_0\in \Lambda^3(\mathbb{R}^7)^*$ is given by \eqref{eq:modelG2Structure} (cf. \cite[Theorem~8.1]{SalamonWalpuski-NotesOnOctonions}). Then $\mathtt{J}$ is induced by the map 
\begin{align*}
\mathbb{R} \times B_{\kappa}(0) &\to \mathbb{R}^7 \\
(s,x_1+iy_1,\dots,x_3+iy_3) &\mapsto s b_1 +x_1 b_2 + y_1 b_3+ \dots + y_3 b_7. \qedhere
\end{align*}
\end{proposition}

By shrinking the positive constant $\kappa>0$ in the previous proposition if necessary, we may assume that $\mathtt{J}$ extends to an open embedding $\mathtt{J}\col (\mathbb{R}\times B_{2\kappa}(0)/\mathbb{Z}_7)/\mathbb{Z} \to Y_0$. The following gives a description of $\overline{F_\mathfrak{f}^\prime}$ in the neighbourhood $\mathtt{J}(\mathcal{V}_{2\kappa})$ (where $\mathcal{V}_{2\kappa}$ is defined analogously to $\mathcal{V}_\kappa$ prior to \autoref{def:OrbifoldResolution}).
\begin{proposition}
There exists $\mathbb{R}$-linear maps $\varphi_i \col \mathbb{C}^3 \to \mathbb{R}$ for $i=1,\dots,7$ with the following properties: 
\begin{itemize}
\item $\vert \varphi_i(z_1,z_2,z_3)\vert \leq \vert (z_1,z_2,z_3) \vert$ and $ (\varphi_1,\dots,\varphi_7) \circ \alpha = (\varphi_2,\varphi_3,\varphi_7,-\varphi_6,-\varphi_4,\varphi_1,\varphi_5)$ (where $\alpha$ acts on $\mathbb{C}^3$ via \eqref{eq:example18_action_of_Z7}).
\item Over $(\mathbb{R}\times B_{2\kappa}(0)/\mathbb{Z}_7)/\mathbb{Z}$ the isomorphism $\mathtt{J}^*\overline{F_{\mathfrak{f}}^\prime} \col \mathtt{J}^*E_0 \to \mathtt{J}^*E_{0,\theta}$ is given by \[ \big[(s,z),w \big] \mapsto \big[(s,z),M_\mathfrak{f}(s,z) \cdot w \big] \] where \[ M_\mathfrak{f}(s,z) \coloneqq \textup{diag}\Big(\e^{i\mathfrak{f}\tfrac{s}{\sqrt{7}}}\cdot\e^{i\mathfrak{f}\varphi_1(z_1,z_2,z_3)},\dots,\e^{i\mathfrak{f}\tfrac{s}{\sqrt{7}}}\cdot\e^{i\mathfrak{f}\varphi_7(z_1,z_2,z_3)} \Big) \] (where in the 4-th entry there is a minus in the exponent of $\e^{-i\mathfrak{f}\tfrac{s}{\sqrt{7}}}$).
\end{itemize}
\end{proposition}
\begin{proof}
Let $\{b_1,\dots,b_7\} \subset \mathbb{R}^7$ be the orthogonal basis described in \autoref{prop:example18_description_of_singularset} and define $\varphi_i$ for $(z_1,z_2,z_3) \coloneqq (x_1+iy_1,\dots,x_3+iy_3)$ as \[ \varphi_i(z_1,z_2,z_3) \coloneqq \pr_{i} (x_1b_2 + y_1 b_3+x_2 b_4+y_2 b_5+ x_3 b_6 + y_3 b_7)\] where $\pr_i \col \mathbb{R}^7 \to \mathbb{R}$ denotes the i-th coordinate projection. The proposition then follows from the fact that $\{b_2,\dots,b_7\}$ are orthonormal, the construction of $\mathtt{J}$ in \autoref{prop:example18_description_of_singularset}, and the explicit form of $F_\mathfrak{f}^\prime$ given in \autoref{prop:example18-identifying_flat_bundles-I}.
\end{proof}

We will now modify the isomorphism $\overline{F_\mathfrak{f}^\prime}$ over $\mathcal{V}_{2\kappa}/\mathcal{V}_\kappa$ so that $(\overline{F_\mathfrak{f}^\prime})^*A_{0,\theta}$ attains over the gluing region $\mathcal{V}_\kappa$ the form required in \autoref{prop:pregluing-family} and \autoref{ass: very simple family of gluing data}.

\begin{definition}\label{def:example18_identifying_flat_bundles-II}
For $\mathfrak{f}\in \mathbb{R}$ let $\overline{F_\mathfrak{f}^\prime} \col E_0 \to E_{0,\e^{i\mathfrak{f}}}$ be the isomorphism constructed in \autoref{prop:example18-identifying_flat_bundles-I}. We now define the following (orientation preserving and isometric) isomorphism $\overline{F_\mathfrak{f}}\col E_0 \to E_{0,\e^{i\mathfrak{f}}}$:
\begin{align*}
\overline{F_\mathfrak{f}}([y,w]) = \begin{cases}
\overline{F_\mathfrak{f}^\prime}([y,w]) \textup{ for $[y] \in Y_0 \setminus \mathtt{J}(\mathcal{V}_{2\kappa})$} \\
\big[y,M_\mathfrak{f}^\prime(s,z)w\big] \textup{ for $y=s b_1+\dots+ \im (z_3) b_7 \in \mathtt{J}(\mathcal{V}_{2\kappa})$}
\end{cases}
\end{align*}
where $\{b_1,\dots,b_7\}$ is as in \autoref{prop:example18_description_of_singularset} and  \[ M_\mathfrak{f}^\prime(s,z) \coloneqq \textup{diag}\Big(\e^{i\xi\tfrac{s}{\sqrt{7}}}\cdot\e^{i\xi \chi(\vert(z_1,z_2,z_3)\vert)\varphi_1(z_1,z_2,z_3)},\dots,\e^{i\xi\tfrac{s}{\sqrt{7}}}\cdot\e^{i\xi\chi(\vert(z_1,z_2,z_3)\vert)\varphi_7(z_1,z_2,z_3)} \Big) \] with $\varphi_1,\dots,\varphi_7$ as in the previous proposition (and where in the 4-th entry there is a minus in the exponent of $\e^{-i\mathfrak{f}\tfrac{s}{\sqrt{7}}}$) and $\chi \col [0,2\kappa]\to [0,1]$ a non-decreasing function satisfying 
\begin{align*}
\chi(s) = \begin{cases}
0 \textup{ for $s \leq 5\kappa/4$} \\
1 \textup{ for $s \geq 7\kappa/4$.}
\end{cases}
\end{align*}
Note that this isomorphism is well-defined and smooth by the previous proposition. With this we define the following family of flat connections on $\O(E_0)$ (and $\SO(E_0)$): \[(A_0+a_{0,\mathfrak{f}})_{\mathfrak{f}\in \mathbb{R}} \quad \textup{where} \quad a_{0,\mathfrak{f}} \coloneqq A_0 - (\overline{F_\mathfrak{f}})^*(A_{0,\e^{i\mathfrak{f}}}).\]
\end{definition}

\begin{remark}\label{rem:example18-flat_family_is_Z_2-invariant}
Since all of the $\varphi_i$ in the previous definition are $\mathbb{R}$-linear, one can easily see that the corresponding isomorphisms $\overline{F_\mathfrak{f}}$ are still $\mathbb{Z}_2$-equivariant. Consequently, all $A_0 +a_{0,\mathfrak{f}}$ are $\mathbb{Z}_2$-invariant connections.
\end{remark}

We saw in \autoref{prop:example18-monodromies-not-conjugated} that $A_0+a_{0,\mathfrak{f}_1}$ and $A_0 + a_{0+\mathfrak{f}_2}$ are gauge equivalent in $\SO(E_0)$ if and only if $\mathfrak{f}_1 = \mathfrak{f}_2 +2\pi m$ for some $m\in \mathbb{Z}$ (and are gauge equivalent in $\O(E_0)$ if and only if $\mathfrak{f}_1 = \pm \mathfrak{f}_2 +2\pi m$; cf. \autoref{rem:example18_bundles_are_isomorphic}). The following proposition shows that the family $(A_0+a_{0,\mathfrak{f}})_{\mathfrak{f}\in \mathbb{R}}$ is not even tangent to the gauge orbits.

\begin{proposition}\label{prop:example18_family_not_tangent_to_gaugeorbig}
Let $(A_0+a_{0,\mathfrak{f}})_{\mathfrak{f}\in \mathbb{R}}$ be the family of flat connections on $E_0$ defined in \autoref{def:example18_identifying_flat_bundles-II}. Then there does not exist an $\mathfrak{f} \in \mathbb{R}$ and $\xi \in \Omega^0(Y_0,\mathfrak{so}(E_{0}))$ such that \[ \diff_{A_0+a_{0,\mathfrak{f}}} \xi = (\partial_{\mathbb{R}} a_{0,\, \cdot\, })(\mathfrak{f}),\] where $\partial_{\mathbb{R}} a_{0,\, \cdot\, }$ denotes the derivative of the map $\mathfrak{f} \mapsto a_{0,\mathfrak{f}}$.
\end{proposition}
\begin{proof}
\autoref{prop:example18-identifying_flat_bundles-I} and the construction in \autoref{def:example18_identifying_flat_bundles-II} imply that \[a_{0,\mathfrak{f}} = (\overline{F_{\mathfrak{f}}})^*A_{0,\e^{i\mathfrak{f}}} - A_{0} = \diag(i\mathfrak{f} \diff y^1,\dots, i \mathfrak{f} \diff y^7) + \tilde{a}_{\mathfrak{f}} \] where $\tilde{a}_{\mathfrak{f}} \in \Omega^1(Y_0,\mathfrak{so}(E_{0}))$ is supported in $\mathtt{J}(\mathcal{V}_{2\kappa})$ and is given in the (local) coordinates $\mathtt{J} \col (\mathbb{R}\times B_{2\kappa}(0)/\mathbb{Z}_7)/\mathbb{Z} \to Y_0$ by
\begin{align*}
\mathtt{J}^* \tilde{a}_{\mathfrak{f}} &= \mathfrak{f} \diff_{A_{0}} \Big( i\  \diag\big((-1+\chi(\vert z_1,z_2,z_3\vert) \varphi_1(z_1,z_2,z_3),\dots,(-1+\chi(\vert z_1,z_2,z_3\vert) \varphi_7(z_1,z_2,z_3)\big)\Big).
\end{align*}
The section \begin{align*}
i\ \diag((-1+\chi \varphi_1),\dots,(-1+\chi \varphi_7)) & \in  \Omega^0(\mathbb{R}\times B_{2\kappa}(0),\mathfrak{so}_\mathbb{R}(\mathbb{C}^7))^{\mathbb{Z}\times \mathbb{Z}_7}\\
& \equiv \Omega^0((\mathbb{R}\times B_{2\kappa}(0)/\mathbb{Z}_7)/\mathbb{Z},\mathfrak{so}(E_{0}))
\end{align*}  
can be extended by zero to a section over the entire orbifold which we schematically write as $\diag(h_1,\dots,h_7) \in \Omega^0(\mathbb{R}^7,\mathfrak{so}_{\mathbb{R}}(\mathbb{C}^7))^\Gamma \equiv \Omega^0(Y_0,\mathfrak{so}(E_{0}))$ for functions $h_1,\dots,h_7 \col \mathbb{R}^7 \to i\mathbb{\mathbb{R}}$.

We therefore have for every $\mathfrak{f} \in \mathbb{R}$: 
\begin{align*}
(\partial_{\mathbb{R}} a_{0,\, \cdot\, })(\mathfrak{f}) &=  \diag(i \diff y^1,\dots,i \diff y^7) +  \diff_{A_{0}} \diag(h_1,\dots,h_7) \\
&=  \diag(i \diff y^1,\dots,i \diff y^7) + \diff_{A_0+a_{0,\mathfrak{f}}} \diag(h_1,\dots,h_7)
\end{align*}
where the last equality comes from the fact that both $a_{0,\mathfrak{f}}$ and $\diag(h_1,\dots,h_7)$ are diagonal matrices which therefore commute. The first term of the right-hand side satisfies \[ \diff_{A_0+a_{0,\mathfrak{f}}}^* (\diag(i \diff y^1,\dots, i \diff y^7)) = 0 \] and lies therefore $L^2$-perpendicular to the image of $\diff_{A_0+a_{0,\mathfrak{f}}}$. This implies the claim.
\end{proof}

\subsection{Augmenting the flat connections to a family of $\mathbb{Z}_2$-equivariant gluing data}

Next, we augment the $\mathbb{Z}_2$-invariant family of flat connection $(A_0+a_{0,\mathfrak{f}})_{\mathfrak{f}\in \mathbb{R}} \subset \mathcal{A}(\SO(E_0)) = \mathcal{A}(\O(E_0))$ constructed in \autoref{def:example18_identifying_flat_bundles-II} to a set of $\mathbb{Z}_2$-equivariant gluing data. We begin by finding a set of $\mathbb{Z}_2$-equivariant resolution data for the $\G_2$-orbifold $(Y_0,\phi_0)$.

\begin{proposition}\label{prop:example18_equivariant_resolutiondata_1}
Define a $\mathbb{Z}_2$ action on $(\mathbb{R}\times \mathbb{C}^3/\mathbb{Z}_7)/\mathbb{Z}$ by $(-1)\cdot [s,z]\coloneqq[-s,-z]$. With respect to this action and $\lambda_0$, as defined in \autoref{prop:Example18_Z2-action}, $\mathtt{J}$ as defined in \autoref{prop:example18_description_of_singularset} becomes $\mathbb{Z}_2$-equivariant.
\end{proposition}
\begin{proof}
This follows immediately from the description of $\mathtt{J}$ given in \autoref{prop:example18_description_of_singularset}.
\end{proof}

Now, recall from \autoref{sec:Ricci-Flat ALE metrics on crepant resolutions} that for each generic $\zeta \in \mathfrak{z}^* \cap (\mathfrak{pu}(\mathbb{C}[\mathbb{Z}_7])^{\mathbb{Z}_7})_\mathbb{Q}^*$, there exists an ALE Calabi--Yau $3$-fold $(\hat{Z}_\zeta,\tau_\zeta,\hat{\omega}_\zeta,\hat{\Omega}_\zeta)$ asymptotic to $\mathbb{C}^3/\mathbb{Z}_7$. Furthermore, if $\Ad_{\textup{conj}_{-1}}^*\zeta = \zeta$, then there exists a lift of the induced $\mathbb{Z}_2$-action on $\mathbb{C}^3/\mathbb{Z}_7$ to $(\hat{Z}_\zeta,\tau_\zeta,\hat{\omega}_\zeta,\hat{\Omega}_\zeta)$.
\begin{lemma}\label{lem:Example18_Z2-action_on_ALE}
We have that $\Ad_{\textup{conj}_{-1}}^* = 1$ so that the $\mathbb{Z}_2$-action lifts to $\hat{Z}_\zeta$ for any $\zeta \in \mathfrak{z}^*$. Every generic choice of $\zeta \in \mathfrak{z}^*\cap (\mathfrak{pu}(\mathbb{C}[\mathbb{Z}_7])^{\mathbb{Z}_7})_\mathbb{Q}^*$ leads therefore to a $\mathbb{Z}_2$-equivariant resolution data of $(Y_0,\phi_0)$.
\end{lemma}
\begin{proof}
Recall from \autoref{bul: lifting group actions to resolution} in \autoref{sec:Ricci-Flat ALE metrics on crepant resolutions} that $\Ad^*_{\textup{conj}_{-1}}$ was the coadjoint action of the element $\textup{conj}_{-1} \in \U(\mathbb{C}[\mathbb{Z}_7])$ defined as the complex linear extension of the conjugation map $\textup{conj}_{-1} \col \mathbb{Z}_7 \to \mathbb{Z}_7$ (where $\mathbb{Z}_7$ is considered as a subgroup of $\SU(3)$ via~\eqref{eq:example18_action_of_Z7}). Since $-1$ lies in the center of $\U(\mathbb{C}^3)$, conjugation leaves each element in $\mathbb{Z}_7 \subset \U(\mathbb{C}^3)$ invariant. Thus $\textup{conj}_{-1} =1 $ and, consequently, $\Ad_{\textup{conj}_{-1}}^* = 1$.
\end{proof}

\begin{remark}
In~\cite[Example~18]{Joyce-GeneralisedKummer2} the $\G_2$-orbifold $(Y_0,\phi_0)$ is resolved using resolution data consisting of a crepant resolution $\hat{Z}$ of $\mathbb{C}^3/\mathbb{Z}_7$ found in~\cite[Example~3]{Markushevich-examples-resolutions} and the ALE Calabi--Yau metric on $\hat{Z}$ 
discovered in~\cite{TianYau-Complete-RicciFlat-Kaehler-Mfds}. While~\autoref{bul:ALE-Crepant-Resolutions-via-KählerQuotients} of \autoref{sec:Ricci-Flat ALE metrics on crepant resolutions} argues that $\hat{Z}$ is biholomorphic to one of the manifolds constructed in~\autoref{sec:Ricci-Flat ALE metrics on crepant resolutions}, it is not clear to the author whether the Calabi--Yau metric 
in~\cite{TianYau-Complete-RicciFlat-Kaehler-Mfds} is isometric to any of the metrics in~\autoref{sec:Ricci-Flat ALE metrics on crepant resolutions}. The Kummer construction which we construct in the following is (for a suitable choice of $\zeta \in \mathfrak{z}^*\cap (\mathfrak{pu}(\mathbb{C}[\mathbb{Z}_7])^{\mathbb{Z}_7})_\mathbb{Q}^*$) therefore diffeomorphic to the manifold in~\cite[Example~18]{Joyce-GeneralisedKummer2}. However, the $\G_2$-structures may be different.
\end{remark}

We now fix one choice of generic $\zeta \in \mathfrak{z}^*\cap (\mathfrak{pu}(\mathbb{C}[\mathbb{Z}_7])^{\mathbb{Z}_7})_\mathbb{Q}^*$ and consider the corresponding set of $\mathbb{Z}_2$-equivariant resolution data consisting of $\mathtt{J}$ defined in \autoref{prop:example18_description_of_singularset}, the ALE-space $(\hat{Z}_\zeta,\tau_\zeta,\hat{\omega}_\zeta,\hat{\Omega}_\zeta)$, and the $\mathbb{Z}_2$-actions induced by $\lambda_0$ and the $\mathbb{Z}_2$-action on $\hat{Z}_\zeta$ described in \autoref{bul: lifting group actions to resolution} of \autoref{sec:Ricci-Flat ALE metrics on crepant resolutions}.

\begin{proposition}\label{prop:Example18_bundle_around_singular_set}
Recall from \autoref{prop:example18_description_of_singularset} that there exists an $R\in \SO(\mathbb{R}^7)$ such that $R^{-1}\alpha R$ is of the form \eqref{eq:example18_action_of_Z7}. Define the isomorphism \[F_1 \col (z_1,\dots,z_7)\mapsto (z_1,z_2,z_3,-\bar{z}_4,z_5 , z_6 ,z_7) \in \O_{\mathbb{R}}(\mathbb{C}^7)\] and with this the map
\begin{equation*}
\begin{tikzcd}
F \col \mathbb{C}\oplus (\mathbb{C}^3\otimes_{\mathbb{R}} \mathbb{C}) \equiv (\mathbb{R}\oplus \mathbb{C}^3) \otimes_{\mathbb{R}} \mathbb{C} \arrow[r, "R"] & \mathbb{R}^7 \otimes_{\mathbb{R}} \mathbb{C} \equiv \mathbb{C}^7 \arrow[r, "F_1"] &  \mathbb{C}^7.
\end{tikzcd}
\end{equation*}
Furthermore, define the flat $\SO_\mathbb{R}(\mathbb{C}^7)$-bundle \[(\mathbb{R}\times B_\kappa(0)) \times_{\mathbb{Z}\times \mathbb{Z}_7} (\mathbb{C}\oplus \mathbb{C}^3 \otimes_\mathbb{R}\mathbb{C}) \to (\mathbb{R}\times B_\kappa(0)/\mathbb{Z}_7)/\mathbb{Z} \] where $\mathbb{Z}$ acts trivially on $\mathbb{C}\oplus \mathbb{C}^3\otimes_\mathbb{R}\mathbb{C}$ and $n\in \mathbb{Z}_7$ acts via 
\begin{align*}
f_1(n) \big(w,v_1\otimes &w_1,v_2\otimes w_2,v_3\otimes w_3\big) = \big(w,e^{2n\pi i/7}v_1 \otimes  w_1,e^{4n\pi i/7}v_2 \otimes w_2,e^{8n\pi i/7}v_3\otimes w_3\big).
\end{align*}
The map 
\begin{align*}
\tilde{\mathtt{J}} \col (\mathbb{R}\times B_\kappa(0)) \times_{\mathbb{Z}\times \mathbb{Z}_7} (\mathbb{C}\oplus \mathbb{C}^3 \otimes_\mathbb{R}\mathbb{C}) &\to \mathtt{J}^*E_0 \\
[(s,z),w] &\mapsto [y,F(w)]
\end{align*}
where $y = s b_1+ \re(z_1) b_2 + \dots + \im(z_3) b_7$ for the orthonormal basis $\{b_1,\dots,b_7\} \subset \mathbb{R}^7$ of \autoref{prop:example18_description_of_singularset} is an isomorphism that identifies $\mathtt{J}^*A_0$ with the canonical flat connection on $(\mathbb{R}\times B_\kappa(0)) \times_{\mathbb{Z}\times \mathbb{Z}_7} (\mathbb{C}\oplus \mathbb{C}^3 \otimes_\mathbb{R}\mathbb{C})$. If we define a $\mathbb{Z}_2$-action on $(\mathbb{R}\times B_\kappa(0)) \times_{\mathbb{Z}\times \mathbb{Z}_7} (\mathbb{C}\oplus \mathbb{C}^3 \otimes_\mathbb{R}\mathbb{C})$ via \[ (-1) \cdot [(s,z),(w,z_1\otimes w_1,z_2\otimes w_2,z_3 \otimes w_3)] = [(-s,-z),(\overline{w},z_1\otimes \overline{w}_1,z_2\otimes \overline{w}_2,z_3 \otimes \overline{w}_3)], \] then $\tilde{\mathtt{J}}$ becomes $\mathbb{Z}_2$-equivariant.
\end{proposition}

\begin{proof}
We will prove that $F$ conjugates between the monodromy representations of $\mathtt{J}^*A_0$ and the flat connection on $(\mathbb{R}\times B_\kappa(0)) \times_{\mathbb{Z}\times \mathbb{Z}_7} (\mathbb{C}\oplus \mathbb{C}^3 \otimes_\mathbb{R}\mathbb{C})$. Since $\mathbb{R}\times B_\kappa(0)$ is the orbifold covering space of $(\mathbb{R}\times B_\kappa(0)/\mathbb{Z}_7)/\mathbb{Z}$ this implies the first part of the proposition.

The singular set is precisely the image of the fixed-point set of $\alpha$ in $Y_0$. Thus, the action of $\mathbb{Z}_7$ is generated by $f_1(\alpha)$. Similarly, the action of $\mathbb{Z}$ is generated by $f_1(\tau_1\tau_2\tau_3\tau_4^{-1}\tau_5\tau_6\tau_7) = \id$. Conjugating these two actions by $F_1$ and identifying $\mathbb{C}^7= \mathbb{R}^7 \otimes_\mathbb{R}\mathbb{C}$ leads to \[ F_1^{-1} f_\theta(\alpha) F_1 = \alpha \otimes 1_\mathbb{C} \qquad \textup{and} \qquad F_1^{-1} f_\theta(\tau_1\tau_2\tau_3\tau_4^{-1}\tau_5\tau_6\tau_7) F_1 = 1_{\mathbb{R}^7} \otimes 1_\mathbb{C}. \] Conjugating this a second time by $R$ leads (under the identification $\mathbb{R}^7 \cong \mathbb{R}\oplus \mathbb{C}^3$) to the action of $f_1(1)$ as described above. This concludes the first part of the proposition. The $\mathbb{Z}_2$-equivariance of $\tilde{\mathtt{J}}$ follows from \[(F^{-1}\circ \lambda_0(-1) \circ F)(w,v_1\otimes w_1, v_2 \otimes w_2, v_3 \otimes w_3) = (\overline{w},v_1\otimes \overline{w}_1, v_2 \otimes \overline{w}_2, v_3 \otimes \overline{w}_3). \qedhere\]
\end{proof}

\begin{proposition}\label{prop:example18-a_(0,f)-in-gluing-region}
Let $\tilde{\mathtt{J}} \col (\mathbb{R}\times B_\kappa(0)) \times_{\mathbb{Z}\times \mathbb{Z}_7} (\mathbb{C}\oplus \mathbb{C}^3\otimes_{\mathbb{R}}\mathbb{C}) \to \mathtt{J}^*E_0$ be the isomorphism of the previous proposition and $(A_0+a_{0,\mathfrak{f}})_{\mathfrak{f}\in \mathbb{R}} \subset \mathcal{A}(\SO(E_0)) = \mathcal{A}(\O(E_0))$ be the family of flat connections defined in \autoref{def:example18_identifying_flat_bundles-II}. Then $\tilde{\mathtt{J}}^*a_{0,\mathfrak{f}}$ is given (under the identification $(\mathbb{C}\oplus \mathbb{C}^3\otimes_\mathbb{R} \mathbb{C} \cong (\mathbb{R}\oplus \mathbb{C}^3)\otimes_{\mathbb{R}} \mathbb{C}$) by \[ \tilde{\mathtt{J}}^*a_{0,\mathfrak{f}} = \tfrac{\mathfrak{f}}{\sqrt{7}} (\textup{id}_{\mathbb{R} \oplus \mathbb{C}^3} \otimes_\mathbb{R} i  )  \diff s. \] 
\end{proposition}

\begin{proof}
From the construction of $a_{0,\mathfrak{f}}$ in \autoref{def:example18_identifying_flat_bundles-II} and \autoref{prop:example18-identifying_flat_bundles-I} follows that over $\mathcal{V}_\kappa$ the 1-form $\mathtt{J}^*a_{0,\mathfrak{f}}$ is given by \begin{align*}
\mathtt{J}^*a_{0,\mathfrak{f}} = \tfrac{\mathfrak{f}}{\sqrt{7}} \diag\big(i,i,i,-i,i,i,i\big) \diff s &\in \Omega^1(\mathbb{R}\times B_\kappa(0),\mathfrak{so}_\mathbb{R}(\mathbb{C}^7))^{\mathbb{Z}\times \mathbb{Z}_7} \\
& \equiv \Omega^1((\mathbb{R}\times B_{\kappa}(0)/\mathbb{Z}_7)/\mathbb{Z},\mathtt{J}^*E_0). 
\end{align*}
Conjugating $\mathtt{J}^*a_{0,\mathfrak{f}}$ by $F$ defined in \autoref{prop:Example18_bundle_around_singular_set} gives \[ \tilde{\mathtt{J}}^*a_{0,\mathfrak{f}} = \tfrac{\mathfrak{f}}{\sqrt{7}} (\textup{id}_{\mathbb{R} \oplus \mathbb{C}^3} \otimes_\mathbb{R} i  )  \diff s. \qedhere \]
\end{proof}

Denote by $\nu \col \mathbb{Z}_7 \to \U(3)$ the representation generated by $\textup{diag}(\e^{2\pi i/7},\e^{4\pi i/7},\e^{8\pi i/7}) \in \U(3).$ Recall from \autoref{prop:irred_HYM_on_tautological_bundles} that there exists a unitary vector bundle $\pi_\nu \col \hat{E}_\nu \to \hat{Z}_\zeta$ together with a local framing $\tilde{\tau}_\nu \col \hat{E}_{\nu\vert \hat{Z}_\zeta \setminus \tau^{-1}(0)} \to (\mathbb{C}^3\setminus\{0\}) \times_{\mathbb{Z}_7, \nu} \mathbb{C}^3$.

Next, let $\pi \col \hat{E}_\zeta \to \hat{Z}_\zeta$ be the vector bundle with \[\hat{E}_\zeta \coloneqq \underline{\mathbb{C}}\oplus \hat{E}_\nu \otimes_\mathbb{R} \underline{\mathbb{C}} \cong \underline{\mathbb{C}} \oplus \hat{E}_{\nu} \oplus \hat{E}_{\bar{\nu}}.\] The local framing on $\hat{E}_\nu$ and the canonical framing on the trivial bundle combine to \[\tilde{\tau} \col \hat{E}_{\zeta\vert \hat{Z}_\zeta \setminus \tau^{-1}(0)} \to (\mathbb{C}^3\setminus\{0\}) \times_{\mathbb{Z}_7}(\mathbb{C} \oplus \mathbb{C}^3 \otimes_\mathbb{R}\mathbb{C})\] (where $\mathbb{Z}_7$ acts on $\mathbb{C}^3 \times (\mathbb{C} \oplus \mathbb{C}^3 \otimes_\mathbb{R}\mathbb{C})$ as in \autoref{prop:Example18_bundle_around_singular_set}). Recall from \autoref{prop:irred_HYM_on_tautological_bundles} and \autoref{prop:rigid_SO(n)-connections} that $\SO(\hat{E}_\zeta)$ carries an infinitesimally rigid Hermitian Yang--Mills connection $\hat{A}_\zeta$ which is asymptotic to the flat connection on $(\mathbb{C}^3\setminus\{0\}) \times_{\mathbb{Z}_7}(\mathbb{C} \oplus \mathbb{C}^3 \otimes_\mathbb{R}\mathbb{C})$ with rate $-5$. Since for $E_0$ the group $\mathbb{Z}$ acts trivially on $\mathbb{C}\oplus \mathbb{C}^3\otimes_\mathbb{R}\mathbb{C}$ (cf. \autoref{prop:example18_description_of_singularset}), the following proposition is immediate:
\begin{proposition}
Denote by $\hat{\tilde{\rho}}_\zeta \col \mathbb{Z} \to \textup{Isom}(\hat{E}_\zeta)$ the trivial $\mathbb{Z}$-action. Since $\mathbb{Z}$ acts trivially in the monodromy representation of $\mathtt{J}^*A_0$ (cf. \autoref{prop:example18_description_of_singularset}), $\hat{\tilde{\rho}}_\zeta$ is a lift of the (trivial) $\mathbb{Z}$-action on $\mathbb{C}^3\times_{\mathbb{Z}_7}(\mathbb{C}\oplus \mathbb{C}^3\otimes_\mathbb{R}\mathbb{C})$ to $\hat{E}_\zeta$ (in the sense of \autoref{def:gluing_data}). This action trivially preserves $\hat{A}_\zeta$.
\end{proposition}

\begin{proposition}\label{prop:example18-families-over-resolution}
For $\mathfrak{f}\in \mathbb{R}$ define the following endomorphism-valued 1-form: \[ \hat{a}_{\zeta,\mathfrak{f}} \coloneqq \tfrac{\mathfrak{f}}{\sqrt{7}}(\id_{\mathbb{R}\oplus \hat{E}_\nu} \otimes_\mathbb{R} i) \diff s \in \Omega^1((\mathbb{R}\times \hat{Z}_\zeta)/\mathbb{Z},\mathfrak{so}(\hat{E}_\zeta)). \] This satisfies $(\tilde{\tau}_\zeta)_* \hat{a}_{\zeta,\mathfrak{f}} = \tilde{\mathtt{J}}^*a_{0,\mathfrak{f}}$. 
\end{proposition}
\begin{proof}
the local framing $\tilde{\tau}_\nu \col \hat{E}_\nu \to (\mathbb{C}^3\setminus \{0\})\times_{\mathbb{Z}_7,\nu} \mathbb{C}^3$ identifies the respective identity-automorphisms. Similarly, the trivial framing on $\hat{Z}_\zeta \times \underline{\mathbb{C}}$ maps multiplication by $i$ onto multiplication by $i$ on $(\mathbb{C}^3\setminus \{0\}/\mathbb{Z}_7) \times \underline{\mathbb{C}}$. The proposition follows now from the description of $\tilde{\mathtt{J}}^*a_{0,\mathfrak{f}}$ given in \autoref{prop:example18-a_(0,f)-in-gluing-region}.
\end{proof}

The following proposition follows directly from the previous discussion:
\begin{proposition}\label{prop:Example18_gluing_data}
The set\footnote{Note that since $\tilde{\mathtt{J}}$ is orientation reversing, we need to equip $\hat{E}_\zeta$ with the negative of its canonical orientation when considering the bundle of positively oriented orthonormal frames $\SO(\hat{E}_\zeta)$.} $\{(\pi_0\col \SO(E_{0})\to Y_0,A_0),(\tilde{\mathtt{J}},\pi\col \SO(\hat{E}_\zeta) \to \hat{Z}_\zeta,\hat{A}_\zeta,\hat{\tilde{\rho}}_\zeta\col \mathbb{Z}\to \textup{Isom}(\SO(\hat{E}_\zeta)))\}$ is gluing data in the sense of \autoref{def:gluing_data} compatible with the resolution data consisting of $\mathtt{J}$ and $(\hat{Z}_\zeta,\tau_\zeta,\hat{\omega}_\zeta,\hat{\Omega}_\zeta)$. Similarly, $\{(\pi_0 \col \O(E_0)\to Y_0, A_0),(\tilde{\mathtt{J}},\pi\col \O(\hat{E}_\zeta) \to \hat{Z}_\zeta,\hat{A}_\zeta,\hat{\tilde{\rho}}_\zeta\col \mathbb{Z}\to \textup{Isom}(\O(\hat{E}_\zeta)))\}$ consists of compatible gluing data for the flat bundle $\O(E_{0,\theta})$. Both sets of gluing data together with the families $(A_0+a_{0,\mathfrak{f}},\hat{A}_\zeta+\hat{a}_{\zeta,\mathfrak{f}})_{\mathfrak{f}\in \mathfrak{F}}$ defined in \autoref{def:example18_identifying_flat_bundles-II} and \autoref{prop:example18-families-over-resolution} where $\mathfrak{F}\subset \mathbb{R}$ is any compact interval satisfy the conditions of \autoref{prop:pregluing-family}.
\end{proposition}

\begin{proposition}\label{prop:example18 lift of Z2 action to resolution data}
There exists a lift of the $\mathbb{Z}_2$-action to $(\mathbb{R}\times \hat{E}_\zeta)/\mathbb{Z} \to (\mathbb{R}\times \hat{Z}_\zeta)/\mathbb{Z}$ which preserves $\hat{A}_\zeta$ and $\hat{a}_{\zeta,\mathfrak{f}}$ for every $\mathfrak{f}\in \mathbb{R}$.
\end{proposition}
\begin{proof}
Note that since $-1\in C_{\U(\mathbb{C}^3)}(\mathbb{Z}_7)$ (where $C_{\U(\mathbb{C}^3)}(\mathbb{Z}_7)$ denotes the centraliser of $\mathbb{Z}_7$), by \autoref{lem:group_lift_to_HYMbundle} there exists a lift of the $\mathbb{Z}_2$-action to $\hat{E}_\nu$ (covering the $\mathbb{Z}_2$-action of \autoref{lem:Example18_Z2-action_on_ALE}) which preserves $\hat{A}_\nu$ and is asymptotic to the action on $\mathbb{C}^3\times_{\mathbb{Z}_7,\nu} \mathbb{C}^3$ generated by $(-1) \cdot [z,w]\coloneqq[-z,w]$. We now define a $\mathbb{Z}_2$-action on the trivial bundle $\underline{\mathbb{C}} = \hat{Z}_\zeta \times \mathbb{C}$ by $(z,w) \mapsto (-1\cdot z, \overline{w})$ (where $-1 \cdot z$ denotes the induced action on $\hat{Z}_\zeta$ from \autoref{lem:Example18_Z2-action_on_ALE}). Together, these induce an action on $\hat{E}_\zeta$ which, when combined with the $\mathbb{Z}_2$ action on $\mathbb{R}$, satisfies the claim.
\end{proof}

\begin{proposition}\label{prop:example18_equivariant_gluing-data}
The collection $\{(\pi_0 \col \O(E_0)\to Y_0, A_0),(\tilde{\mathtt{J}},\pi\col \O(\hat{E}_\zeta) \to \hat{Z}_\zeta,\hat{A}_\zeta,\hat{\tilde{\rho}}_\zeta\col \mathbb{Z}\to \textup{Isom}(\O(\hat{E}_\zeta)))\}$ together with the $\mathbb{Z}_2$-actions described in \autoref{prop:Example18_Z2-action} and the previous proposition, respectively, form a set of $\mathbb{Z}_2$-equivariant gluing data in the sense of \autoref{rem:equiv_gluing_data} compatible with the $\mathbb{Z}_2$-equivariant resolution data consisting of $\mathtt{J}$ and $(\hat{Z}_\zeta,\tau_\zeta,\hat{\omega}_\zeta,\hat{\Omega}_\zeta)$ and the lift of the $\mathbb{Z}_2$-action to $(\mathbb{R}\times\hat{Z}_\zeta)/\mathbb{Z}$. The construction of $a_{0,\mathfrak{f}}$ and $\hat{a}_{\zeta,\mathfrak{f}}$ in \autoref{def:example18_identifying_flat_bundles-II} and \autoref{prop:example18-families-over-resolution} imply that all connections in $(A_0+a_{0,\mathfrak{f}},\hat{A}_\zeta+\hat{a}_{\zeta,\mathfrak{f}})_{\mathfrak{f}\in \mathbb{R}}$ are $\mathbb{Z}_2$-invariant (see also \autoref{rem:example18-flat_family_is_Z_2-invariant} and the previous proposition).
\end{proposition} 

\subsection{Turning the gluing data into $\G_2$-instantons}

We now fix a generic $\zeta \in \mathfrak{z}^* \cap (\mathfrak{pu}(\mathbb{C}[\Gamma])^{\mathbb{Z}_7})_\mathbb{Q}^*$ as defined in \autoref{sec:Ricci-Flat ALE metrics on crepant resolutions} and let $(\hat{Z}_\zeta,\tau_\zeta,\hat{\omega}_\zeta,\hat{\Omega}_\zeta)$ be the corresponding ALE space. In \autoref{prop:Example18_Z2-action} we have defined a $\mathbb{Z}_2$-action $\lambda_0 \col \mathbb{Z}_2 \to \textup{Isom}(Y_0,g_0)$ that preserves $\psi_0$. Moreover, we have argued in \autoref{prop:example18_equivariant_resolutiondata_1} and \autoref{lem:Example18_Z2-action_on_ALE} that there exists a $\mathbb{Z}_2$-action on $(\mathbb{R}\times \hat{Z}_\zeta)/\mathbb{Z}$ such that the data in \autoref{prop:example18_description_of_singularset}, the ALE space $(\hat{Z}_\zeta,\tau_\zeta,\hat{\omega}_\zeta,\hat{\Omega}_\zeta)$, and these $\mathbb{Z}_2$-actions give rise to an $\mathbb{Z}_2$-equivariant resolution data for $Y_0$. By \autoref{theo:torsionfree_G2_structure} and \autoref{rem:EquivariantGeneralisedKummer} we therefore have a $1$-parameter family of generalised Kummer constructions $(\hat{Y}_t,\phi_t)_{t\in(0,T_K)}$ and $\hat{\lambda}\col \mathbb{Z}_2 \to \textup{Isom}(\hat{Y}_t,\tilde{g}_t)$ that satisfies $\hat{\lambda}(h)^*\psi_t=\psi_t$ for every $h\in \mathbb{Z}_2$.

\begin{proposition}\label{prop:example18_almost-instantons+estimates}
For each $t\in (0,T_K)$ there exist a principal $\SO(14)$- and a $\O(14)$-bundle $\pi_t^{\SO} \col \SO(\hat{E}_t) \to \hat{Y}_t$ and $\pi_t^\O \col \O(\hat{E}_t) \to \hat{Y}_t$, respectively, together with smooth 1-parameter families of connections $(\tilde{A}_t^\SO + \tilde{a}_{t,\mathfrak{f}})_{\mathfrak{f}\in \mathbb{R}} \subset \mathcal{A}(\SO(\hat{E}_t))$ and $(\tilde{A}_t^\O + \tilde{a}_{t,\mathfrak{f}})_{\mathfrak{f}\in \mathbb{R}} \subset \mathcal{A}(\O(\hat{E}_t))$ that satisfy:
\begin{enumerate}
\item $\SO(\hat{E}_{t})\subset \O(\hat{E}_{t})$ is a subbundle and the connections $\tilde{A}_{t}^\O + \tilde{a}_{t,\mathfrak{f}}$ reduce to $\tilde{A}_{t}^{\SO}+\tilde{a}_{t,\mathfrak{f}}$ for every $\mathfrak{f}\in \mathbb{R}$.
\item There exists a lift of $\hat{\lambda} \col \mathbb{Z}_2 \to \textup{Isom}(\hat{Y}_t,\tilde{g}_t)$ to $\O(\hat{E}_{t})$ preserving $\tilde{A}_{t}^{\O}+\tilde{a}_{t,\mathfrak{f}}$ for every $\mathfrak{f}\in \mathbb{R}$.
\item For every compact interval $\mathfrak{F}\subset \mathbb{R}\setminus \pi \mathbb{Z}$ there exist $t$-independent constants $C>0$ and $0<T<T_K$ (which depend on $\mathfrak{F}$) such that for every $\mathfrak{f}\in \mathfrak{F}$:
\begin{itemize}
\item $\Vert (\partial_\mathfrak{F}^\ell F_{\tilde{A}_{t}^{\O}+\tilde{a}_{t,\mathfrak{f}}})\wedge\psi_t \VertWHt{0}{-5/2}{} \leq C t$ for $\ell =0,1,2$, where $\partial_\mathfrak{F}^\ell F_{\tilde{A}_{t}^{\O}+\tilde{a}_{t,\mathfrak{f}}}$ denotes the $\ell$-th derivative of the function $\mathfrak{f} \mapsto F_{\tilde{A}_{t}^{\O}+\tilde{a}_{t,\mathfrak{f}}}$,
\item for $\ell =0,1,2$ and any $\beta \in \mathbb{R}$ we have
\begin{align*}
\Vert (\partial_\mathfrak{F}^\ell L_{\tilde{A}_t^\O+\tilde{a}_{t,\mathfrak{f}}}) \underline{a} \VertWHt{0}{\beta-1}{} \leq C \Vert \underline{a} \VertWHt{1}{\beta}{} &\textup{ for any $\underline{a} \in \Omega^1(\hat{Y}_t,\mathfrak{so}(\hat{E}_t))\oplus \Omega^7(\hat{Y}_t,\mathfrak{so}(\hat{E}_t))$}\\
\Vert (\partial_\mathfrak{F}^\ell L^*_{\tilde{A}_t^\O+\tilde{a}_{t,\mathfrak{f}}}) \underline{b} \VertWHt{1}{\beta-1}{} \leq C \Vert \underline{b} \VertWHt{2}{\beta}{} &\textup{ for any $\underline{b} \in \Omega^6(\hat{Y}_t,\mathfrak{so}(\hat{E}_t))\oplus \Omega^0(\hat{Y}_t,\mathfrak{so}(\hat{E}_t))$}
\end{align*} 
where $\partial_\mathfrak{F}^\ell L_{\tilde{A}_t^\O + \tilde{a}_{t,\mathfrak{f}}}$ (and similarly $\partial_\mathfrak{F}^\ell L_{\tilde{A}_t^\O + \tilde{a}_{t,\mathfrak{f}}}^*$) denotes the $\ell$-th derivative of the function $\mathfrak{f} \mapsto  L_{\tilde{A}_t^\O + \tilde{a}_{t,\mathfrak{f}}} \in \textup{Lin}(C^{1,\alpha}_{\beta,t},C^{0,\alpha}_{\beta-1,t})$,
\item for any $\xi \in \Omega^0(\hat{Y}_t,\mathfrak{so}(\hat{E}_t))$ and $\beta \in \mathbb{R}$ we have \[ \Vert \diff_{\tilde{A}_t^\O + \tilde{a}_{t,\mathfrak{f}}} \xi \VertWHt{0}{\beta-1}{} \leq C \Vert \xi \VertWHt{1}{\beta}{},\]
\item for $t<T^\prime$ any $\mathbb{Z}_2$-invariant $\underline{b}\in \Omega^1(\hat{Y}_t,\mathfrak{so}(\hat{E}_{\theta,t}))\oplus \Omega^7(\hat{Y}_t,\mathfrak{so}(\hat{E}_{\theta,t}))$ satisfies \[ \Vert \underline{b} \VertWHt{2}{-1/2}{} \leq C \Vert L_{\tilde{A}_{t}^{\O}+\tilde{a}_{t,\mathfrak{f}}}L_{\tilde{A}_{t}^{\O}+\tilde{a}_{t,\mathfrak{f}}}^* \underline{b} \VertWHt{0}{-5/2}{}\] and any (not necessarily $\mathbb{Z}_2$-invariant) $\xi \in \Omega^0(\hat{Y}_t,\mathfrak{so}(\hat{E}_t))$ satisfies \[\Vert \xi \VertWHt{2}{-1/2}{} \leq C \Vert \diff_{\tilde{A}^O_{t}+\tilde{a}_{t,\mathfrak{f}}}^*\diff_{\tilde{A}^O_{t}+\tilde{a}_{t,\mathfrak{f}}} \xi \VertWHt{0}{-5/2}{}.\] 
\end{itemize}
\end{enumerate}
All norms in this statement are taken with respect to the metric and the Levi--Civita connection on $(\hat{Y}_t,\tilde{g}_t)$ and the (negative) Killing-form and $\tilde{A}_{t}^{\O}$ on $\mathfrak{o}(\hat{E}_{\theta,t})=\mathfrak{so}(\hat{E}_{\theta,t})$.
\end{proposition}
\begin{proof}
In the beginning of \autoref{subsec:family of flat connections} we have defined the flat bundles $(\SO(E_{0}),A_{0}^{\SO})$ and $(\O(E_{0}),A_{0}^{\O})$ which we have equipped in \autoref{def:example18_identifying_flat_bundles-II} with a smooth 1-parameter family of flat connection $(A_0+a_{0,\mathfrak{f}})_{\mathfrak{f}\in \mathbb{R}} \subset \mathcal{A}(\SO(E_0))=\mathcal{A}(\O(E_0))$. In the previous section we have then augmented these to gluing data $\{(\pi_0\col \SO(E_{0})\to Y_0,A_0),(\tilde{\mathtt{J}},\pi\col \SO(\hat{E}_\zeta) \to \hat{Z}_\zeta,\hat{A}_\zeta,\hat{\tilde{\rho}}_\zeta\col \mathbb{Z}\to \textup{Isom}(\SO(\hat{E}_\zeta)))\}$ and $\{(\pi_0 \col \O(E_0)\to Y_0, A_0),(\tilde{\mathtt{J}},\pi\col \O(\hat{E}_\zeta) \to \hat{Z}_\zeta,\hat{A}_\zeta,\hat{\tilde{\rho}}_\zeta\col \mathbb{Z}\to \textup{Isom}(\O(\hat{E}_\zeta)))\}$ together with families of endomorphism valued 1-forms $(a_{0,\mathfrak{f}},\hat{a}_{\zeta,\mathfrak{f}})$ which satisfy the assumptions of \autoref{prop:pregluing-family} (cf. \autoref{prop:Example18_gluing_data}). 

The existence of the bundles $\pi_{t}^{\SO} \col \SO(\hat{E}_{t}) \to \hat{Y}_t$ and $\pi_{t}^{\O} \col \O(\hat{E}_{t}) \to \hat{Y}_t$ and the respective families of connection $(\tilde{A}_t^\SO + \tilde{a}_{t,\mathfrak{f}})_{\mathfrak{f}\in \mathbb{R}} \subset \mathcal{A}(\SO(\hat{E}_t))$ and $(\tilde{A}_t^\O + \tilde{a}_{t,\mathfrak{f}})_{\mathfrak{f}\in \mathbb{R}} \subset \mathcal{A}(\O(\hat{E}_t))$ follows therefore from \autoref{prop:pregluing} and \autoref{prop:pregluing-family}. Note that since $(\SO(E_{0}),A_{0}^{\SO})\subset (\O(E_{0}),A_{0}^{\O})$ and $(\SO(\hat{E}_\zeta),\hat{A}_\zeta^{\SO})\subset (\O(\hat{E}_\zeta),\hat{A}_\zeta^{\O})$ are subbundles, we also have $(\SO(\hat{E}_{t}),\tilde{A}_{t}^{\SO} + \tilde{a}_{t,\mathfrak{f}})\subset (\O(\hat{E}_{t}),\tilde{A}_{t}^{\O}+\tilde{a}_{t,\mathfrak{f}})$ by the construction of these bundles and connections in the proof of \autoref{prop:pregluing} and \autoref{prop:pregluing-family}. This proves the first point.

We have shown in \autoref{prop:Example18_Z2-action} and \autoref{rem:example18-flat_family_is_Z_2-invariant} that the $\mathbb{Z}_2$-action on $Y_0$ lifts to $\O(E_{0})$ and preserves $A_{0}^\O+\tilde{a}_{t,\mathfrak{f}}$ for every $\mathfrak{f}\in \mathbb{R}$. Furthermore, we have proven in \autoref{prop:example18 lift of Z2 action to resolution data} that there exists a $\mathbb{Z}_2$-action on $\O((\mathbb{R}\times\hat{E}_\zeta)/\mathbb{Z})$ that preserves $\hat{A}_\zeta^\O$ and that augments $\{(\pi_0 \col \O(E_0)\to Y_0, A_0),(\tilde{\mathtt{J}},\pi\col \O(\hat{E}_\zeta) \to \hat{Z}_\zeta,\hat{A}_\zeta,\hat{\tilde{\rho}}_\zeta\col \mathbb{Z}\to \textup{Isom}(\O(\hat{E}_\zeta)))\}$ to a set of $\mathbb{Z}_2$-equivariant gluing data (cf. \autoref{prop:example18_equivariant_gluing-data}). The existence of the lift of $\hat{\lambda}$ to $\O(\hat{E}_{\theta,t})$ preserving $\tilde{A}_{t}^\O + \tilde{a}_{t,\mathfrak{f}}$ follows therefore also from \autoref{prop:pregluing} and \autoref{prop:pregluing-family}. The estimates listed under the first three bullets of point three also follow from the same propositions (see also \autoref{rem: estimate on d_A^*}). 

The estimate $\Vert \underline{b} \VertWHt{2}{-1/2}{} \leq c \Vert L_{\tilde{A}_{t}^{\O}+\tilde{a}_{t,\mathfrak{f}}}L_{\tilde{A}_{t}^{\O}\tilde{a}_{t,\mathfrak{f}}}^* \underline{b} \VertWHt{0}{-5/2}{}$ for every $\mathbb{Z}_2$-invariant element $\underline{b}\in \Omega^1(\hat{Y}_t,\mathfrak{so}(\hat{E}_{\theta,t}))\oplus \Omega^7(\hat{Y}_t,\mathfrak{so}(\hat{E}_{\theta,t}))$ and sufficiently small $t\in (0,T)$ follows from \autoref{prop:LinearEstimate}. Our description of $\hat{a}_{\zeta,\mathfrak{f}}$ in \autoref{prop:example18-families-over-resolution} together with \autoref{prop:example18_ker(L_A)}, and \autoref{prop:rigid_SO(n)-connections} imply here-for that \autoref{ass:invertible_linearisations} is satisfied.

Similarly, \autoref{prop:example18_ker(L_A)} implies that the condition in \autoref{prop:linearEstimate-for-rigidity} is satisfied which then implies $\Vert \xi \VertWHt{2}{-1/2}{} \leq c \Vert \diff_{\tilde{A}^O_{t}+\tilde{a}_{t,\mathfrak{f}}}^*\diff_{\tilde{A}^O_{t}+\tilde{a}_{t,\mathfrak{f}}} \xi \VertWHt{0}{-5/2}{}$ for any $\xi \in \Omega^0(\hat{Y}_t,\mathfrak{so}(\hat{E}_t))$ and possibly smaller $t<T$.
\end{proof}
\begin{remark}
As in \autoref{rem:example18_bundles_are_isomorphic} one can see that the connections $\tilde{A}_t^{\SO}+\tilde{a}_{t,\mathfrak{f}_1}$ and $\tilde{A}_t^{\SO}+\tilde{a}_{t,\mathfrak{f}_2}$ are gauge equivalent if and only if $\mathfrak{f}_1 = \mathfrak{f}_2 + 2\pi m$ for some $m\in \mathbb{Z}$ (this is because the restriction of $\tilde{A}_{t}^{\SO}+\tilde{a}_{t,\mathfrak{f}_i}$ to $\hat{Y}_t \setminus \hat{\mathcal{V}}_{\kappa}^t$ agrees with the flat connection $A_{0}^\SO + a_{0,\mathfrak{f}}$). The corresponding statement about $\O(\hat{E}_{t})$ holds also true, where $\tilde{A}_{t}^{\O}+\tilde{a}_{t,\mathfrak{f}_1}$ and $\tilde{A}_t^\O+ \tilde{a}_{t,\mathfrak{f}_2}$ are gauge equivalent if and only if $\mathfrak{f}_1 = \pm \mathfrak{f}_2 + 2\pi m$ for some $m\in \mathbb{Z}$.
\end{remark}

\begin{proposition}\label{prop:example18-existence-instantons}
Let $\pi_t^{\SO} \col \SO(\hat{E}_t) \to \hat{Y}_t$ be the principal bundle constructed in the previous proposition. Furthermore, for a fixed compact interval $\mathfrak{F}\subset \mathbb{R}\setminus \pi \mathbb{Z}$ let $(\tilde{A}_{t}^{\SO} + \tilde{a}_{t,\mathfrak{f}})_{\mathfrak{f}\in \mathfrak{F}} \subset \mathcal{A}(\SO(\hat{E}_{t}))$ be the family of connections described in the previous proposition. Then there are constants $0<T^\prime<T$ and $C^\prime>0$ such that for each $t\in (0,T^\prime)$ there exists a smooth family of $\G_2$-instantons $(A_{t}^\SO + a_{t,\mathfrak{f}})_{\mathfrak{f}\in \mathfrak{F}} \in \mathcal{A}(\SO(\hat{E}_{t}))$ with 
\begin{align}\label{eq:example18-estiamtes-on-perturbed-instantons}
\Vert \partial_\mathfrak{F}^\ell (A_{t}^\SO + a_{t,\mathfrak{f}}) - \partial_\mathfrak{F}^{\ell}(\tilde{A}_{t}^{\SO} + \tilde{a}_{t,\mathfrak{f}}) \VertWHt{1}{-3/2}{} \leq C^\prime t \qquad \textup{for $\ell =0,1,2$},
\end{align} 
where $\partial_\mathfrak{F}^{\ell}(A_{t}^{\SO} + a_{t,\mathfrak{f}})$ (and similarly $\partial_\mathfrak{F}^\ell (\tilde{A}_{t}^\SO + \tilde{a}_{t,\mathfrak{f}})$) denotes the $\ell$-th derivative of the function $\mathfrak{f}\mapsto A_{t}^{\SO} + a_{t,\mathfrak{f}}$. For sufficiently small $T^\prime$, all $A^{\SO}_t+a_{t,\mathfrak{f}}$ are infinitesimally irreducible.
\end{proposition}

\begin{proof}
In the previous proposition we have also constructed the bundle $\pi_t^\O \col \O(\hat{E}_t) \to \hat{Y}_t$ together with the smooth family of connections $(\tilde{A}_t^\O+\tilde{a}_{t,\mathfrak{f}})_{\mathfrak{f}\in \mathfrak{F}}$. Furthermore, we have shown that \autoref{theo:perturbing_almost_instantons}, \autoref{prop: derivative estimates on deformed family of instantons}, and \autoref{prop:infinitesimal_irreducibility} can be applied to $(\tilde{A}_t^\O+\tilde{a}_{t,\mathfrak{f}})_{\mathfrak{f}\in \mathfrak{F}}$. This implies the existence of a family of $\G_2$-instantons $(A_{t}^\O+a_{t,\mathfrak{f}})_{\mathfrak{f}\in \mathfrak{F}} \subset \mathcal{A}(\O(\hat{E}_t))$ that satisfies all the properties stated in the proposition (where \eqref{eq:example18-estiamtes-on-perturbed-instantons} are taken with respect to $\tilde{A}_t^\O+\tilde{a}_{t,\mathfrak{f}}$). Since $\SO(\hat{E}_{t}) \subset \O(\hat{E}_{t})$ is a subbundle and $\mathfrak{so}(14)=\mathfrak{o}(14)$, the instanton $A_{t}^\O + a_{t,\mathfrak{f}}$ reduces for any $\mathfrak{f}\in \mathfrak{F}$ to $\SO(\hat{E}_{t})$ which is the instanton that we call $A_{t}^\SO + a_{t,\mathfrak{f}} \in \mathcal{A}(\SO(\hat{E}_{\theta,t}))$. 

Since $\tilde{A}_t^{\SO} + \tilde{a}_{t,\mathfrak{f}}$ is the reduction of $\tilde{A}_t^\O + \tilde{a}_{t,\mathfrak{f}}$, \eqref{eq:example18-estiamtes-on-perturbed-instantons} follows from the corresponding estimate on $\partial_\mathfrak{F}^{\ell}(\tilde{A}_{t}^{\O} + \tilde{a}_{t,\mathfrak{f}}) -\partial_\mathfrak{F}^\ell (A_{t}^\O + a_{t,\mathfrak{f}})$. Furthermore, since the adjoint bundles of $\SO(\hat{E}_t)$ and $\O(\hat{E}_t)$ agree, $\tilde{A}_{t}^{\SO} + \tilde{a}_{t,\mathfrak{f}}$ is infinitesimally irreducible if and only if $\tilde{A}_{t}^{\O} + \tilde{a}_{t,\mathfrak{f}}$ is.
\end{proof}

\begin{remark}
The construction of the connection $A^\O_{t}$ in the previous proof is a technical step because the $\mathbb{Z}_2$-action does not lift directly to $\SO(\hat{E}_{t})$. However, as already noted in \autoref{rem:example18-explanation O(14)-bundle I} we obtain an induced action on the adjoint bundle $\mathfrak{so}(\hat{E}_{t})$ and the map $\Upsilon_{\tilde{A}^{\SO}_{t}}$ (as defined in \autoref{subsec: gauge theory on Kummer}) associated to the connection $\tilde{A}^{\SO}_{t}$ is equivariant with respect to this action. Thus, one can avoid the auxiliary construction of $A^\O_{t}$ by appealing to \autoref{bul: perturbing for more general H-actions} of \autoref{rem:perturbing almost instantons weakening assumptions}.
\end{remark}

\begin{proposition}\label{prop:example18-non-flat-instantons}
Let $\mathfrak{F}\subset \mathbb{R}\setminus \pi \mathbb{Z}$ be a compact interval and let $(A_{t}^\SO+a_{t,\mathfrak{f}})_{\mathfrak{f}\in \mathfrak{F}} \subset \mathcal{A}(\SO(\hat{E}_{t}))$ be the family of $\G_2$-instantons of the previous proposition. For every $\mathfrak{f}\in \mathfrak{F}$ the connection $A_{t}^\SO+ a_{t,\mathfrak{f}}$ is non-flat.  
\end{proposition}
\begin{proof}
We will show that due to its topology the bundle $\SO(\hat{E}_t)$ does not admit any flat connection. Assume therefore that $A \in \mathcal{A}(\SO(\hat{E}_t))$ is flat to produce a contradiction. Fix a point $s_0 \in \mathbb{R}$ and define \[ \iota \col (t\tau_\zeta)^{-1}(B_\kappa(0)/\mathbb{Z}_7)\subset \hat{Z}_\zeta \to \{s_0\} \times (t\tau_\zeta)^{-1}(B_\kappa(0)/\mathbb{Z}_7) \to \hat{\mathcal{V}}_{\kappa}^t \to \hat{Y}_t \] to be the canonical inclusion (see the paragraph before \autoref{def:OrbifoldResolution} for the definition of these sets). By the construction in \autoref{prop:pregluing}, we have that $\iota^*\SO(\hat{E}_{t}) \cong \SO(\hat{E}_\zeta)_{\vert (t\tau_\zeta)^{-1}(B_\kappa(0)/\mathbb{Z}_7)}$. Pulling back $A$, we obtain a flat connection on the (real) vector bundle $(\hat{E}_{\zeta})_{{\vert (t\tau_\zeta)^{-1}(B_\kappa(0)/\mathbb{Z}_7)}}$ and therefore a flat unitary connection on \[(\hat{E}_\zeta\otimes_\mathbb{R} \mathbb{C})_{{\vert (t\tau_\zeta)^{-1}(B_\kappa(0)/\mathbb{Z}_7)}} \cong (\hat{E}_\zeta\oplus \overline{\hat{E}}_\zeta)_{{\vert (t\tau_\zeta)^{-1}(B_\kappa(0)/\mathbb{Z}_7)}}\cong (2 \oplus_{i=0}^{6} \hat{E}_{\nu_i})_{{\vert (t\tau_\zeta)^{-1}(B_\kappa(0)/\mathbb{Z}_7)}} \] where $\pi_\ell \col \hat{E}_{\nu_\ell} \to \hat{Z}_\zeta$ denotes to complex vector bundle associated to the representation $\nu_\ell \col \mathbb{Z}_7 \to \U(1)$ which maps $1 \mapsto \e^{2\pi i \ell/7}$ (this follows from the definition of $\hat{E}_\zeta$ in the previous section). Since $(t\tau_\zeta)^{-1}(B_\kappa(0)/\mathbb{Z}_7)$ and $\hat{Z}_\zeta$ are homotopy-equivalent, this implies by Chern--Weil theory for the reduced Chern-character $\widetilde{\textup{ch}} = \textup{ch} - \textup{rk}$: 
\begin{equation}\label{eq:example18_reduced_Chern-character}
 0 = \widetilde{\textup{ch}}(\hat{E}_\zeta \otimes_\mathbb{R} \mathbb{C}) = 2 \sum_{i=0}^6 \widetilde{\textup{ch}}(\hat{E}_{\nu_i}) = 2 \sum_{i=1}^6 \widetilde{\textup{ch}}(\hat{E}_{\nu_i}) \in \H^*(\hat{Z}_\zeta,\mathbb{R})
\end{equation}
(where we have used in the last equation that $\pi_0 \col \hat{E}_{\nu_0} \to \hat{Z}_\zeta$ is the trivial vector bundle).

Next, we define vectors $v_1,\dots,v_6 \in \mathbb{R}^6$ via \[ (v_i)_j \coloneqq \int_{\hat{Z}_\zeta} \widetilde{\textup{ch}}(\hat{E}_{\nu_i})\widetilde{\textup{ch}}(\hat{E}_{\nu_j}^*) \quad \textup{for $j=1,\dots,6$} \] where the integral is to be understood as in~\cite[Section~7]{WalpuskiDegeratu-HYM_on_crepant_resolutions} via $\H^k(\hat{Z}_\zeta,\mathbb{R})\cong \H^k_{\textup{cpt}}(\hat{Z}_\zeta,\mathbb{R})$ for $k=2,4$ (cf.~\cite[Section~7]{WalpuskiDegeratu-HYM_on_crepant_resolutions}). Multiplying~\eqref{eq:example18_reduced_Chern-character} with $\widetilde{\textup{ch}}(\hat{E}_{\nu_j}^*)$ for $j=1,\dots,6$ and integrating yields \[0 = 2 \sum_{i=1}^6 v_i.\] However, this is in contradiction to~\cite[Equation~(1.8)]{WalpuskiDegeratu-HYM_on_crepant_resolutions} which implies that $v_1,\dots,v_6$ are linearly independent.
\end{proof}

\subsection{Injectivity of the associated curve into the moduli space}

We fix a compact interval $\mathfrak{F}  \subset \mathbb{R}\setminus \pi \mathbb{Z}$ and let $(A_{t,\mathfrak{f}})_{\mathfrak{f} \in \mathfrak{F}} \subset \mathcal{A}(\SO(\hat{E}_{t}))$ be the corresponding family of $\G_2$-instantons constructed in \autoref{prop:example18-existence-instantons}. The following proves that for sufficiently small $t$ no two connections $A_{t,\mathfrak{f}_1}$, $A_{t,\mathfrak{f}_2}$ for distinct $\mathfrak{f}_1,\mathfrak{f}_2 \in \mathfrak{F}$ are gauge equivalent.

\begin{proposition}\label{prop:example18-non-gauge-equivalent-instantons}
There exists a $0<T^{\prime\prime}< T^\prime$ such that for $t<T^{\prime\prime}$ the curve 
\begin{align*}
\mathfrak{F} &\to \mathcal{A}(\SO(\hat{E}_t))/\mathcal{G}(\SO(\hat{E}_t)) \\
f &\mapsto [A_{t,\mathfrak{f}}]
\end{align*}
into the space of connections modulo gauge is injective.
\end{proposition}
\begin{proof}
Assume that this was not the case. Then there exists a sequence $(t_n)_{n\in \mathbb{N}} \subset (0,T^\prime)$ with $t_n\to 0$ and for every $n\in \mathbb{N}$ two distinct $\mathfrak{f}_n,\mathfrak{f}_n^\prime \in \mathfrak{F}$ and a gauge transformation $u_n\in \mathcal{G}(\SO(\hat{E}_{t_n}))$ such that $u_n^*(A_{t_n,\mathfrak{f}_n^\prime}) = A_{t_n,\mathfrak{f}_n}$. We will prove in the following proposition that this implies the existence of an $\mathfrak{f} \in \mathfrak{F}$ and $\xi \in \Omega^0(Y_0,\mathfrak{so}(E_0))$ such that $\diff_{A_{0}+a_{0,\mathfrak{f}}} \xi = (\partial_\mathbb{R}a_{0,\, \cdot\, })(\mathfrak{f})$ (where $A_0$ and $a_{0,\mathfrak{f}}$ are as in \autoref{subsec:family of flat connections} and \autoref{def:example18_identifying_flat_bundles-II}). However, by \autoref{prop:example18_family_not_tangent_to_gaugeorbig} this cannot happen and we have arrived therefore at a contradiction.
\end{proof}

\begin{proposition}
Suppose there exist $(t_n)_{n\in \mathbb{N}}\subset (0,T^\prime)$, $(\mathfrak{f}_n,\mathfrak{f}_n^\prime)_{n\in\mathbb{N}}\subset \mathfrak{F}\times \mathfrak{F}$ with $\mathfrak{f}_n \neq \mathfrak{f}_n^\prime$ and for every $n\in \mathbb{N}$ a gauge transformation $u_n \in \mathcal{G}_{\hat{Y}_{t_n}}(\SO(\hat{E}_{t_n}))$ such that \[t_n \to 0 \qquad \textup{and} \qquad u_n^*(A_{t_n,\mathfrak{f}_n^\prime}) = A_{t_n,\mathfrak{f}_n}. \] Then the following hold: 
\begin{enumerate}
\item There is an $\mathfrak{f}_\infty \in \mathfrak{F}$ such that (up to taking subsequences) $\mathfrak{f}_n,\mathfrak{f}_n^\prime \to \mathfrak{f}_\infty$,
\item The restriction of $(u_n)_{n\in \mathbb{N}}$ to $\hat{Y}_{t_n}\setminus (t_n\tau)^{-1}(S) \cong Y_0 \setminus S$ (where $S$ denotes the singular set of $Y_0$) converges (up to taking a subsequence) in $C^2_{\textup{loc}}(Y_0\setminus S)$ to $\pm \id \in \mathcal{G}_{Y_0}(\SO(E_{0}))$
\item There exists a $\xi \in \Omega^0(Y_0,\mathfrak{so}(E_0))$ such that $\diff_{A_{0}+a_{0,\mathfrak{f}_\infty}} \xi = (\partial_{\mathbb{R}} a_{0,\, \cdot} \, )(\mathfrak{f}_\infty)$. 
\end{enumerate}
\end{proposition}

\begin{proof}
Since $\mathfrak{F}$ is compact, we can assume that up to taking subsequences $\mathfrak{f}_n \to \mathfrak{f}_\infty$ and $\mathfrak{f}_n^\prime \to \mathfrak{f}_\infty^\prime$. In order to prove the first point we therefore only need to show that $\mathfrak{f}_\infty^\prime = \mathfrak{f}_\infty$.

For this we focus on the domain $\hat{Y}_t\setminus (t\tau)^{-1}(S)$ (where $S\subset Y_0$ denotes the singular set and $t\tau \col \hat{\mathcal{V}}_\kappa^t \to \mathcal{V}_\kappa$ was defined prior to \autoref{def:OrbifoldResolution}) which we will subsequently identify with $Y_0 \setminus S$ using $t\tau$. Likewise, we identify the bundles \[\SO(\hat{E}_{t})_{\vert \hat{Y}_t \setminus (t\tau)^{-1}(S)} \cong \SO(E_0)_{\vert Y_0 \setminus S}\] (cf. the construction in the proof of \autoref{prop:pregluing}). 

We will show in the subsequent paragraphs that there exists a converging subsequence $u_{n \vert Y_0\setminus S} \to u$ (in $C^2(Y_0\setminus S)$) where $u\col \SO(E_{0})_{\vert Y_0 \setminus S} \to \SO(E_{0})_{\vert Y_0 \setminus S}$ is a (smooth) bundle isomorphism that satisfies $u^*(A_{0, \mathfrak{f}_\infty^\prime})_{\vert Y_0 \setminus S} = (A_{0,\mathfrak{f}_\infty})_{\vert Y_0 \setminus S}$. This implies that the monodromy representations of $(A_{0,\mathfrak{f}_\infty^\prime})_{\vert Y_0\setminus S}$ and $(A_{0,\mathfrak{f}_\infty})_{\vert Y_0\setminus S}$ are conjugated in $\SO(14)$ (cf.~\cite[Proposition~2.2.3]{DonaldsonKronheimer}). It is not difficult to see that $\pi_1(Y_0 \setminus  S ) \cong \Gamma$ (this follows, for example, by realising that $\Gamma$ acts freely on the pre-image of $Y_0 \setminus S$ under the quotient map $\mathbb{R}^7 \to \mathbb{R}^7/\Gamma =Y_0$ and that this pre-image is simply connected) and that the monodromy representation of each $(A_{0,\mathfrak{f}})_{\vert Y_0 \setminus S}$ for any $\mathfrak{f} \in \mathfrak{F}$ corresponds to the homomorphism $f_{\e^{i{\mathfrak{f}}}}$ as defined in \autoref{prop:example18-definition-monodromy-representation}. The first point of the proposition follows therefore from \autoref{prop:example18-monodromies-not-conjugated} where we have proven that the homomorphisms associated to two distinct $\theta_1,\theta_2 \in S^1$ are not conjugated in $\SO(14)$ and therefore $\mathfrak{f}_\infty = \mathfrak{f}_\infty^\prime$.  Similarly, the fact that $\pm \id \in \SO(14)$ are the only elements whose conjugation fixes $f_{\theta}$ for any $\theta \in S^1 \setminus \{\pm 1\}$ (also proven in \autoref{prop:example18-monodromies-not-conjugated}) implies that $u= \pm \id$.

In order to prove the existence of a gauge transformation $u\in \mathcal{G}_{Y_0\setminus S}\SO(E_0))$ as described in the previous paragraph, we write for any $t\in(0,T^\prime)$ and $\mathfrak{f} \in \mathfrak{F}$ the connection $(A_{t,\mathfrak{f}})_{\vert Y_0 \setminus S}$ as $(A_{t,\mathfrak{f}})_{\vert Y_0 \setminus S} = (A_{0})_{\vert Y_0 \setminus S} + a_{t,\mathfrak{f}}^{\circ}$ for $a_{t,\mathfrak{f}}^\circ \in \Omega^1(Y_0\setminus S,\mathfrak{so}(E_0))$. Next, we use the isometric inclusions $\mathcal{G}(\SO(E_{0})) \subset \Gamma(\End(E_{0}))$ and $\mathfrak{so}(E_{0}) \subset \End(E_{0})$ (induced by the respective inclusions $\SO(\mathbb{R}^{14}),\mathfrak{so}(\mathbb{R}^{14})\subset \End(\mathbb{R}^{14})$) to regard $u_n$ and $a_{t,\mathfrak{f}}^\circ$ as sections of this (linear) bundle. The condition that $u_n^*(A_{t_n,\mathfrak{f}_n^\prime}) = A_{t_n,\mathfrak{f}_n}$ implies 
\begin{align*}
(A_{0})_{\vert Y_0 \setminus S} + a_{t_n,\mathfrak{f}_n}^\circ &= (A_{t_n,\mathfrak{f}_n})_{\vert Y_0 \setminus S} = u_n^*(A_{t_n,\mathfrak{f}_n^\prime})_{\vert Y_0 \setminus S} \\
&= (A_{0,1})_{\vert Y_0 \setminus S} + u_n^{-1} \circ \diff_{A_{0}}(u_n) + u_n^{-1} \circ a_{t_n,\mathfrak{f}_n^\prime}^\circ \circ u_n
\end{align*}
and therefore 
\begin{equation}\label{eq:example18-gauge_trafos-bootstrap}
\diff_{A_{0}} u_n = u_n \circ a_{t_n,\mathfrak{f}_n}^\circ - a_{t_n,\mathfrak{f}_n^\prime}^\circ \circ u_n.
\end{equation}

Next, we pick an exhaustion \[ \overline{Y_0\setminus \mathtt{J}(\mathcal{V}_\kappa}) \coloneqq K_1 \subset K_2 \subset K_3 \subset \cdots \subset Y_0 \setminus S \] by compact subsets. The restriction of the weighted norm $\Vert \cdot \VertWHt{k}{\beta}{(K_i)}$ to any $K_i$ is uniformly (in $t$) equivalent to the norm $\Vert \cdot \VertH{k}{(K_i)}$ because the weight function $w_t$ is bounded from below and above on this domain. \autoref{theo:torsionfree_G2_structure} implies therefore that the restriction $g_{t_n \vert Y_0\setminus S}$ converge in $C^\infty_{\text{loc}}$ to $g_{0 \vert Y_0 \setminus S}$. Similarly, \eqref{eq:example18-estiamtes-on-perturbed-instantons} for $\ell=0$ and the construction of $\tilde{A}_{t,\mathfrak{f}}$ in \autoref{prop:pregluing} and \autoref{prop:pregluing-family} imply that for any $i\in \mathbb{N}$ and $\mathfrak{f} \in \mathfrak{F}$ we have $(A_{t,\mathfrak{f}})_{\vert K_i} \to (A_{0,\mathfrak{f}})_{\vert K_i}$ in $C^{1,\alpha}(K_i)$, or equivalently, $a_{t,\mathfrak{f}}^\circ \to a_{0,\mathfrak{f}}$ in $C^{1,\alpha}(K_i)$ (where $a_{0,\mathfrak{f}} \in \Omega^1(Y_0,\mathfrak{so}(E_0))$ is as in \autoref{def:example18_identifying_flat_bundles-II} and $a_{t_n,\mathfrak{f}}^\circ\in \Omega^1(Y_0\setminus S,\mathfrak{so}(E_0))$ was defined above so that $(A_{t,\mathfrak{f}})_{\vert Y_0 \setminus S} = (A_{0})_{\vert Y_0 \setminus S} + a_{t,\mathfrak{f}}^\circ$).

In particular, this implies that for any $i\in \mathbb{N}$ the $C^{1,\alpha}(K_i)$-norms of $a_{t_n,\mathfrak{f}}^\circ$ are uniformly bounded. Since $\SO(14)$ is compact, we obtain an a priori bound on the $C^0$-norms of $u_n$ and by~\eqref{eq:example18-gauge_trafos-bootstrap} a uniform bound on their $C^1$-norms. Bootstrapping improves this to a $C^{2,\alpha}$-bound. As $K_i$ is compact, the Arzelà--Ascoli Theorem implies that there is a $C^2$-converging subsequence $u_n \to u$ in $K_i$ and since $\SO(\mathbb{R}^{14}) \subset \End(\mathbb{R}^{14})$ is closed, $u \in \mathcal{G}_{K_i}(\SO(E_0))$. Taking a diagonal sequence over all $K_i$, we obtain a $u\in \mathcal{G}_{Y_0\setminus S}(\SO(E_0))$ such that $u_n \to u$ in $C^2_{\text{loc}}(Y_0 \setminus S)$. Continuity then gives $u^*(A_{0,\mathfrak{f}_\infty^\prime})_{\vert Y_0\setminus S}=(A_{0,\mathfrak{f}_\infty})_{\vert Y_0\setminus S}$ and bootstrapping via~\eqref{eq:example18-gauge_trafos-bootstrap} improves the regularity of $u$ to $C^\infty$.

In order to prove the last point note that after possibly multiplying $u_n$ by $-\id$ we can assume that $u = \id$. For each of the compact subsets $K_i \subset Y_0\setminus S$ there thus exists an $N\in \mathbb{N}$ such that for each $n>N$ the restriction $(u_n)_{\vert K_i}$ lies in the image of the exponential map. Thus, for sufficiently large $n$ there exists a unique $\xi_n \in \Omega^{0}(K_i,\mathfrak{so}(E_0))$ such that \[ (A_{t_n,\mathfrak{f}_n})_{\vert K_i} = (A_{t_n,\mathfrak{f}_n^\prime})_{\vert K_i} + \e^{-\xi_n} \circ (\diff_{A_{t_n,\mathfrak{f}_n^\prime}} \e^{\xi_n}) = (A_{t_n,\mathfrak{f}_n^\prime})_{\vert K_i} + \Upsilon(\xi_n) \circ (\diff_{A_{t_n,\mathfrak{f}_n^\prime}} \xi_n) \] where \[\Upsilon(\xi_n) \coloneqq \frac{1-\e^{-\ad_{\xi_n}}}{\ad_{\xi_n}}. \] This implies \begin{align*}
\diff_{A_{t_n,\mathfrak{f}_n^\prime}} \big(\tfrac{\xi_n}{\mathfrak{f}_n-\mathfrak{f}_n^\prime} \big) = \Upsilon(\xi_n)^{-1} \circ \tfrac{A_{t_n,\mathfrak{f}_n}-A_{t_n,\mathfrak{f}_n^\prime}}{\mathfrak{f}_n-\mathfrak{f}_n^\prime}
\end{align*}
and since $(A_{t,\mathfrak{f}})_{\vert K_i} \to (A_{0,\mathfrak{f}})_{\vert K_i}$ in $C^{1,\alpha}(K_i)$ for every $\mathfrak{f} \in \mathfrak{F}\subset \mathbb{R}\setminus \pi \mathbb{Z}$ we have by \autoref{prop:example18_Poincare-type-estimate} 
\begin{align}\label{equ:example18_estimate_on_xi_n/f_n}
\Vert \tfrac{\xi_n}{\mathfrak{f}_n-\mathfrak{f}_n^\prime} \VertC{2}{(K_i)} \leq c \Vert \Upsilon(\xi_n)^{-1} \circ \tfrac{A_{t_n,\mathfrak{f}_n}-A_{t_n,\mathfrak{f}_n^\prime}}{\mathfrak{f}_n-\mathfrak{f}_n^\prime} \VertC{1}{(K_i)}
\end{align} 
for some $c>0$ independent of $K_i$ and $n>N$. Since $u_n \to \id$ in $C^2(K_i)$ we obtain $\xi_n \to 0$ in $C^2(K_i)$. Furthermore, there exists an $\tilde{\mathfrak{f}}_n\in \mathfrak{F}$ between $\mathfrak{f}_n$ and $\mathfrak{f}_n^\prime$ such that \[ \tfrac{A_{t_n,\mathfrak{f}_n}-A_{t_n,\mathfrak{f}_n^\prime}}{\mathfrak{f}_n-\mathfrak{f}_n^\prime} = \partial_\mathfrak{F} A_{t_n,\mathfrak{f}_n^\prime} + \tfrac{1}{2} \partial_\mathfrak{F}^2 A_{t_n,\tilde{\mathfrak{f}}_n} (\mathfrak{f}_n-\mathfrak{f}_n^\prime). \] Since $\mathfrak{f}_n,\mathfrak{f}_n^\prime \to \mathfrak{f}_\infty$ by the first point of this proposition and \[(\partial_\mathfrak{F} A_{t_n,\mathfrak{f}_n^\prime})_{\vert K_i} \to (\partial_\mathfrak{F} A_{0,\mathfrak{f}_\infty})_{\vert K_i} \quad \textup{and} \quad (\partial_\mathfrak{F}^2 A_{t_n,\tilde{\mathfrak{f}}_n})_{\vert K_i} \to (\partial_\mathfrak{F}^2 A_{0,\mathfrak{f}_\infty})_{\vert K_i} \quad \textup{in $C^{1,\alpha}(K_i)$} \] by \eqref{eq:example18-estiamtes-on-perturbed-instantons} in \autoref{prop:example18-existence-instantons} and the construction in \autoref{prop:pregluing-family}, we obtain that the quotient $\tfrac{A_{t_n,\mathfrak{f}_n}-A_{t_n,\mathfrak{f}_n^\prime}}{\mathfrak{f}_n-\mathfrak{f}_n^\prime}$ converges in $C^{1,\alpha}(K_i)$ to $(\partial_\mathfrak{F} A_{0,\mathfrak{f}_\infty})_{\vert K_i}$. The right-hand side of \eqref{equ:example18_estimate_on_xi_n/f_n} is therefore uniformly bounded and the Arzelà--Ascoli Theorem implies that $\tfrac{\xi_n}{\mathfrak{f}_n-\mathfrak{f}_n^\prime}$ converges (up to taking a subsequence) in $C^1(K_i)$. Taking a diagonal sequence over all $K_i \subset Y_0\setminus S$ we obtain a $\xi_\infty \in \Omega^0(Y_0\setminus  S, \mathfrak{g}_P)$ such that $\xi_n \to \xi_\infty$ in $C^1_{\textup{loc}}(Y_0\setminus S)$. By continuity $\diff_{A_{0,\mathfrak{f}_\infty}} \xi_\infty = \partial_\mathfrak{F} A_{0,\mathfrak{f}_\infty}$ over $Y_0 \setminus S$ and \autoref{prop:example18_Poincare-type-estimate} implies that $\Vert \xi_\infty \VertC{0}{(Y_0\setminus S)} < c$. Integration by parts implies that $\diff_{A_{0,\mathfrak{f}_\infty}} \xi_\infty = \partial_\mathfrak{F} A_{0,\mathfrak{f}_\infty}$ holds over the entire orbifold $Y_0$ in the sense of distributions. Elliptic regularity for the (overdetermined elliptic) operator $\diff_{A_{0,\mathfrak{f}_\infty}}$ then implies that $\xi_\infty \in \Omega^0(Y_0,\mathfrak{g}_P)$ and writing $A_{0,\mathfrak{f}} = A_0 +a_{0,\mathfrak{f}}$ for any $\mathfrak{f}\in \mathfrak{F}$ proves the third statement.
\end{proof}
\begin{remark}
A similar proof shows that the connections $A_{t,\mathfrak{f}}$ for compact $\mathfrak{F}\subset \mathbb{R}\setminus \pi \mathbb{Z}$ are all irreducible (i.e. the only $u\in \mathcal{G}(\SO(\hat{E}_t))$ with $u^*(A_{t,\mathfrak{f}})=A_{t,\mathfrak{f}}$ are $u=\pm \id$) once $t$ is sufficiently small.
\end{remark}

\bibliography{references}{}
\bibliographystyle{alpha}

\end{document}